%% file: main.tex
\documentclass[12pt,a4paper]{scrartcl}
\usepackage{amsmath}
\usepackage{gensymb}
\usepackage{amssymb}
\usepackage{amsfonts}
\usepackage{amsthm}
\usepackage{lmodern}
\usepackage{hyperref}
\usepackage{graphicx}
\usepackage[english]{babel}
\usepackage[T1]{fontenc}
\usepackage[utf8]{inputenc}
\usepackage{csquotes}
\usepackage{units}
\usepackage[natbib = true]{biblatex}
\usepackage{subfigure}
\usepackage{wrapfig}
\usepackage{physics}
\usepackage{xparse}
\usepackage{paralist}
\usepackage{bbm}
\usepackage{mathrsfs}
\usepackage{color}
\usepackage{multirow}
\usepackage{float}
\usepackage{fancyhdr}
\usepackage{array}
\usepackage[toc,page]{appendix}
\usepackage[shortlabels]{enumitem}

\usepackage[intoc]{nomencl}
\makenomenclature

\renewcommand{\nompreamble}{The next list describes the most important symbols used}

\newcommand{\F}{\mathbb}

\newcommand{\lb}{\textup{(}}
\newcommand{\rb}{\textup{)}}
\newtheorem{innercustomthm}{Theorem}
\newenvironment{customthm}[1]
  {\renewcommand\theinnercustomthm{#1}\innercustomthm}
  {\endinnercustomthm}
\newtheorem{theorem}{Theorem}

\newtheorem{lemma}[theorem]{Lemma}
\newtheorem{corollary}[theorem]{Corollary}
\theoremstyle{definition}
\newtheorem{definition and lemma}[theorem]{Definition and Lemma}
\newtheorem{remark}[theorem]{Remark}
\newtheorem{definition}[theorem]{Definition}
\newtheorem{example}[theorem]{Example}
\newtheorem{reference_paper}[theorem]{Reference}
\newtheorem*{setting}{Setting}

\numberwithin{equation}{subsection}
\numberwithin{theorem}{subsection}
\begin{document}

\input{Titelseite}
\pagestyle{fancy}
\fancyfoot{}
\fancyhead{}

\newpage
\tableofcontents

\newpage

\setcounter{page}{1}
\fancyfoot{}
\fancyhead{}
\fancyhead[LO,RE]{\thepage}

\input{Introduction}

\input{Theoretical_background}
\input{Markov_Jump}

\input{Probability}

\input{Pertur}

\input{Concentration_Inequalities}
\input{Generalization_infinite_state_space}

\input{Appendix}

\newpage
\nomenclature{$E$}{finite state space; i.e. some finite set }
\nomenclature{$\F{X} = (X_t)_{t \geq 0}$}{Markov jump process \space \pageref{def: Markov_jump_process} }
\nomenclature{$(\F{P}_x)_{x \in E}$}{family of probability measures of an MJP  \space \pageref{def: Markov_jump_process}}
\nomenclature{$\nu, \mu, \beta $ }{probability measures on $E$}
\nomenclature{$f, g, h$}{real valued functions on $E$}
\nomenclature{$\F{P}_x( \text{\,} \cdot \text{\,} \vert \text{\,}  \cdot \text{\,} ) , \F{E}_x( \text{\,} \cdot \text{\,}   \vert \text{\,} \cdot \text{\,} ) $}{conditional probability, conditional expectation }
\nomenclature{$\pi$}{invariant probability measure of an irreducible MJP \space \pageref{th: uniqueness_existence_invariant_measure}}
\nomenclature{$(P_t)_{t \geq 0}$}{semigroup of operators; corresponding to an MJP  \pageref{eq: contraction_semigroup_Markov_process}}
\nomenclature{$L, L^*$}{infinitesimal generator, adjoint of infinitesimal generator; of an (irreducible) MJP \space \pageref{lem: contraction_semigroup_MJP}}
\nomenclature{$\mathbf{1}$}{constant $1$-function on $E$ }
\nomenclature{$\norm{\cdot}_2, \langle \cdot , \cdot \rangle $}{$L^2(\pi)$ norm  of functions and operators, scalar product on $L^2(\pi)$}
\nomenclature{$S$}{reduced resolvent of $\frac{L+L^*}{2}$ with respect to $0$ \space \pageref{def: reduced_resolvent} }
\nomenclature{$\F{R}^E,\,\mathcal{B}(E),\,L^2(\pi)$}{real function spaces: all functions $f : E \rightarrow \F{R}$, bounded and measurable, $\pi$-square-integrable}
\nomenclature{$\pi(f), \mu(f)$}{short notation for $\int_E f d\pi, \int_E f d\mu $}
\nomenclature{$F,G,H$}{real valued functions defined on some subset $D \subset \F{R}$} 
\nomenclature{$(P_t^h)_{t \geq 0}$}{Feynman-Kac semigroup \pageref{eq: def_P_t^h}}
\nomenclature{$\mathcal{M}_1(E)$}{set of probability measures on $E$}
\nomenclature{$A_t$}{integral $\int_0^t f(X_s) ds$}
\nomenclature{$M_f$}{multiplication with $f$}
\nomenclature{$\lambda_0(r)$}{largest eigenvalue of $\frac{L+L^*}{2} + rM_f$}
\nomenclature{$F^*, G^*, \lambda_0^*$}{ Fenchel conjugates \space \pageref{def: fenchel_conjugate}}
\nomenclature{$\Psi_Z(r), \Psi_{A_t}(r)$}{cumulant generating function of $Z$ resp. $A_t$; i.e. $\log \F{E} (e^{rZ}) $ resp. $\log \F{E}_\nu(e^{rA_t})$}
\nomenclature{$\Psi_Z^*$}{Cramér transform of a random variable $Z$ \space \pageref{th: cramer_chernoff}}
\nomenclature{$\F{P}_\nu, \F{E}_\nu$}{probability measure $\sum_{x \in E} \nu_x \F{P}_x$, expectation with respect to $\F{P}_\nu$ \space \pageref{eq: def_P_nu}}
\printnomenclature



\newpage

\printbibliography
\end{document}

%% file: Titelseite.tex
\begin{titlepage}
\begin{center}

\textsc{\LARGE Concentration Inequalities for Markov Jump Processes}

\vspace{2cm}

\textsc{\LARGE Georg-August-Universität Göttingen}

\vspace{1.5cm}

\textsc{\LARGE Bachelor Thesis} \\
\vspace{1.5cm}
\large
A thesis submitted in partial fulfillment of the requirements for the degree of  \\
\Large 
Bachelor of Science in Mathematics

\vspace{1cm}

\small{Supervisor:}\\

\large{Prof. Dr. Anja Sturm}

\vspace{1cm}

\small{Second Assessor:}\\

\large{Dr. Aljaz Godec}

\vspace{1cm}

\small{Date of submission:}\\
16.02.2022
\vfill 
\large{Santiago Carrero Ibanez} \\
\end{center}
\end{titlepage}

%% file: Introduction.tex
\section{Introduction}
What are  'concentration inequalities'? In general, concentration inequalities refer to inequalities that provide bounds on the probability that a random variable deviates from some characteristic value, e.g. the mean value. For example, one may consider a random variable $Z$ and may look at probabilities such as (see e.g. \cite{concentration2013})
\begin{equation}
    \F{P}(Z \geq u), \; \F{P}(Z \leq u),  \F{P}(\abs{Z} \geq u).
\end{equation}
Concentration inequalities  are of the form
\begin{equation}
    P(u) \leq B(u),
\end{equation}
where $B(u) $ is some bound and $P(u)$ is one of the aforementioned probabilities. In other words, concentration inequalities provide bounds on how much the (probability) distribution of a random variable is 'concentrated' around some characteristic value. The Markov  and Chebychev inequalities are simple, well known examples of concentration inequalities.
There are many areas of applications of concentration inequalities including statistical mechanics, statistics, information theory, and high-dimensional geometry \cite[P.1]{concentration2013}.
Several methods have been developed to prove such inequalities, e.g. martingale methods, information theoretic methods, 'the entropy method', the transportation method, etc. For an extensive, general presentation of different methods and concentration inequalities see \cite{concentration2013}. \\
What are 'concentration inequalities for Markov jump processes'? In this work we consider concentration inequalities for functionals of a special class of Markov processes, namely Markov jump processes (MJP) on finite sets. More precisely, our main goal is to derive bounds for the probability
\begin{equation}
\label{eq: first_probability}
    \mathbb{P}_\nu \left (\frac{1}{t}\int_{0}^{t}f(X_s)ds - \pi(f)  \geq u \right ),
\end{equation}
where $u \geq 0$ and $\pi(f) = \int f d\pi $, $\F{X} = (X_t)_{t \geq 0}$ is an irreducible Markov jump process on a finite state space $E$, with invariant distribution $\pi$, $\nu $ denotes the initial distribution, and $f : E \rightarrow \F{R} $ is some function  of interest (see Section 2 for the definition of these notions). The inquiry of the above probability is very natural, as by the well known ergodic theorem 
\begin{equation}
    \frac{1}{t}\int_{0}^{t}f(X_s)ds \xrightarrow[]{t \to \infty} \pi(f) \quad \F{P}_\nu- \text{a.s.}
\end{equation}
 Of course, the deviation probability $\mathbb{P}_\nu \left (t^{-1}\int_{0}^{t}f(X_s)ds - \pi(f)  \geq u \right )$ is of equal interest, however, it is sufficient to consider \eqref{eq: first_probability} as we can replace $f \to -f$. It is important to note that the ergodic theorem is asymptotic in nature and it does not provide any information about about the rate of convergence or deviations at finite time. However, in many areas of application, e.g. Monte-Carlo simulation or non-equilibrium statistical physics, it is of interest to examine the fluctuations of (finite time) averages $t^{-1}\int_0^t f(X_s) ds$, and to approximate $\int f d \pi$ by the (finite time) average  $t^{-1}\int_0^t f(X_s) ds$ (see also Section \ref{subsec: conc_inequalities_physics}). Consequently, the study of concentration inequalities is of fundamental and practical interest, as they provide some insight to the fluctuations of time averages and provide quantitative bounds for times, where time averages $t^{-1}\int_0^t f(X_s) ds$ are 'close' to $\pi(f) $ .\\
To derive concentration inequalities  we will use the so called \textit{Cramér-Chernoff method}, a general technique which is used to provide bounds for probabilities of the form $\F{P}(Z \geq u)$. The application of the Cramér-Chernoff method to derive concentration inequalities for \eqref{eq: first_probability} in the context of Markov processes is well established, and  this work is based on \cite{wu}, \cite{lezaud}, \cite{guillin}, \cite{bernstein}; works that derive concentration inequalities for \eqref{eq: first_probability} based on the explicit or implicit use of the Cramér-Chernoff method . These works all consider more general Markov processes, and in this work we present, summarize and extend some selected concentration inequalities (\cite[Thrm.\,1]{wu}, \cite[Thrm.\,1.1]{lezaud}, \cite[Prop.1.4, Thrm.\,2.3]{guillin} and \cite[Thrm.\,1.2, Thrm.\,2.2]{bernstein}) in the context of Markov jump processes. The present work can be read independently of the works just mentioned, in particular we present full proofs of the main results (contained in Section \ref{sec: concentration_inequalities}). Nonetheless, it is advisable for the reader to inspect also the mentioned works as these provide additional context and we sometimes explain the connections in the notation, the results and proofs of the present work with respect to the corresponding notation, results and proofs of \cite{wu}, \cite{lezaud}, \cite{guillin}, \cite{bernstein}. The main results of this work are summarized in Section \ref{subsec: Summary}.  \\
This thesis is divided in three main parts. First, Section \ref{sec: preliminaries} presents the general framework and the most important background information on which this thesis is based on. Hereby, Section \ref{subsec: Markov_jump_processes} treats MJPs (Markov jump processes) \textemdash the general framework of this thesis. Here we provide basic definitions and results about MJPs, including invariant distributions, infinitesimal generators and long term behavior. In Section \ref{subsec: additional_tools}  we present additional tools and background information concerning  linear algebra, perturbation theory and convex analysis. Overall, Section \ref{sec: preliminaries} contains frameworks, definitions and results, which are used in Section \ref{sec: concentration_inequalities} in the derivation of concentration inequalities. Then, in Section \ref{sec: concentration_inequalities} \textemdash  the core of this thesis \textemdash we present the Cramér-Chernoff method and  apply it to functionals of MJPs to derive concentration inequalities for the probability \eqref{eq: first_probability}. Hereby, in Section \ref{subsec: cramer_chernoff_method} we explain how the Cramér Chernoff method may be used  to derive general concentration inequalities for tail probabilities $\F{P}(Z \geq u)$. Then,  in Section \ref{subsec: application_cramer_chernoff}, based on the Cramér-Chernoff method we derive concentration inequalities for functionals of (irreducible) MJPs;  we first  derive a general concentration inequality  and then based on this inequality we derive further, more concrete concentration inequalities by using three different approaches: perturbation theory, functional inequalities  and information inequalities. A detailed outline of Section \ref{sec: concentration_inequalities} is given in Section \ref{subsec: outline_and_goal}.
Finally, in Section \ref{sec: schluss} we summarize our results, give an outlook to further theory and discuss the application in physics.  \\
To comprehend this thesis the reader should have a solid understanding of basic probability theory; including stochastic processes and conditional expectations, basic linear algebra and some functional analysis; in particular, basic knowledge about Hilbert spaces. Knowledge about Markov processes is useful but not necessary to understand the core of this thesis.

%% file: Theoretical_background.tex
\newpage

\section{Preliminaries}
\label{sec: preliminaries}
   

%% file: Markov_Jump.tex
\subsection{Markov Jump Processes}
\label{subsec: Markov_jump_processes}
Markov jump processes are a special type of Markov processes \textemdash stochastic processes  defined by the Markov property, which states that given the present the future is independent of the past. Markov processes find applications in many areas including physics, population dynamics, financial markets, etc. If the state space $E$ of the process is countable, the term 'continuous time Markov chain' is also frequently used (see  \cite{Anderson}, \cite{Bremaud},\cite{Norris1998},\cite{Liggett2010}). We shall use the term Markov jump process (MJP) and consider only finite state spaces. \\
This section is structured as follows. First, in Section \ref{subsubsec: basic_theory}  we present basic definitions and results, providing a general framework and introducing notation. Then, in Section \ref{subsubsec: invariant_distributions} we consider invariant distributions and present a central result about the existence and uniqueness. Afterwards, in Section \ref{subsubsec: infinitesimal_generators} we present the Markov semigroup of an MJP and compute the infinitesimal generator and its properties. Finally, in Section \ref{subsubsec: limit_behaviour} we present two results about the long term behavior, including  the ergodic theorem. \\
It should be mentioned that not everything covered in Section \ref{subsec: Markov_jump_processes} will be necessary for
our purposes. Nevertheless, it is still worthwhile covering all these topics as they provide  solid background information. The most important notions and results of this section, which should be kept in mind when reading the main part of this work, are: The basic framework and notions \textemdash irreducible MJPs (Definitions \ref{def: Markov_jump_process}, \ref{def: irreducibility}) and invariant distributions (Definition \ref{def: invariant_measure}), the existence and uniqueness of invariant distributions (Theorem \ref{th: uniqueness_existence_invariant_measure}), infinitesimal generators of (irreducible) MJPs  and its properties (Section \ref{subsubsec: infinitesimal_generators}) and the ergodic Theorem (Theorem \ref{th: ergodic_theorem}). A reader who is mainly interested in the core of this work  (Section \ref{sec: concentration_inequalities}) and is  familiar with basic concepts and notation of MJPs may skim through the definitions and results of Sections \ref{subsubsec: basic_theory}, \ref{subsubsec: invariant_distributions} and \ref{subsubsec: limit_behaviour}, focus on Section \ref{subsubsec: infinitesimal_generators} and on the results and definitions just mentioned. \\

\subsubsection{Basic Theory}
\label{subsubsec: basic_theory}
In the following let $E$ be a finite set endowed with the $\sigma$-algebra $ \mathcal{E} = \mathcal{P}(E)$. Furthermore, endow $E^{[0,\infty)}$ with the product $\sigma$-algebra $\mathcal{E}^{[0,\infty)}$.  Following \cite[Ch.\,2.1,Def.\,2.1]{Liggett2010} we define:

\begin{definition}[MJP and transition function]
\label{def: Markov_jump_process}
Let $\F{X} = (X_t)_{ t \geq 0}$ be an $E$-valued stochastic process, defined on an underlying measure space $(\Omega, \mathcal{F})$  and let $(\F{P}_x)_{x \in E}$ be a family of probability measures on $(\Omega,\mathcal{F})$. We call $\F{X}$ (or more precisely $(\F{X}, (\F{P}_x)_{ x \in E})$) an MJP if
\begin{enumerate}[\lb $a$\rb]
    \item $\F{X}$ is a jump process, i.e. for all $\omega \in \Omega$ and all $t \geq 0$ there is a $\varepsilon > 0$ such that
    \begin{equation*}
        X_{t+h}(\omega) = X_{t}(\omega) \: \text{for all} \: h \in [t,t+
        \varepsilon].
    \end{equation*} 
    \item $\F{P}_x(X_0 = x) = 1$ for all $x \in E$
    \item The Markov property holds, i.e.  for all $x \in E$, all $t \geq 0$ and all bounded and measurable $A : E^{[0,\infty)} \rightarrow \F{R}$
    \begin{equation}
    \label{eq: Markov_property}
    \F{E}_x[A((X_{t+s})_{ s \geq 0}) | (X_r)_{ 0 \leq r \leq t}] = \F{E}_x[
    A((X_{t+s})_{ s \geq 0}) |X_t] =  \F{E}_{X_t}[A((X_s)_{s \geq 0})],
    \end{equation}
\end{enumerate}
where $\F{E}_x$ denotes the (conditional) expectation under the measure $\F{P}_x$. Define for $t \geq 0$
\begin{equation}
    P(t) := (p_{xy}(t))_{x,y \in E},
\end{equation}
with
\begin{equation}
    p_{xy}(t) := \F{P}_x(X_t = y).
\end{equation}
 The family $(P(t))_{t \geq 0}$ of matrices is called transition function of the MJP.
 
\end{definition}
\begin{remark}(Right continuity)
As $E$ is countable, note that condition $(a)$ above, which states that the paths of the process are piecewise constant, is equivalent to the right continuity of the paths of $\F{X}$. In this work we are only interested in processes with right continuous (i.e. piecewise constant) paths. Furthermore, as the term 'jump process' is usually defined as a right continuous process with piecewise constant paths (see e.g.  \cite[Ch.\,2.2,Def.\,2.5]{Bremaud}, \cite[Ch.\,2.3]{Yin2011}, \cite[Ch.\,12]{Kallenberg2002}), we decided to use the term Markov jump process (MJP) for the stochastic process of our interest. That we assume right continuity of the paths of the stochastic process has the following reasons:
\begin{enumerate}
    \item Right continuity ensures directly that for any $f : E \rightarrow \F{R}$ our time average of interest $t^{-1}\int_0^t f(X_s) ds$ is well defined, measurable  and can be approximated by Riemann sums.
    \item In applications, in order  to characterize the distribution of the process, one specifies the so called   $Q$-matrix $Q = (q_{xy})_{x,y \in E}$ which characterizes the transition probabilities $p_{xy}(t) = \F{P}(X_t = y | X_0 = x) $ for 'small' times:
    \begin{enumerate}[$(a)$]
        \item  $p_{xy}(h) = q_{xy}h + o(h)$ as $ h \to 0$ for $x \neq y$
        \item $p_{xx}(h) = 1 + q_{xx} h + o(h)$ as $h \to 0$
    \end{enumerate}
    Given such a matrix $Q$, one can construct explicitly a continuous time Markov chain with right continuous paths whose transition probabilities satisfy $(a)$ and $(b)$  (see Theorem \ref{th: characterization_MJP}$(b)$).
    \item Right continuity ensures that the so called infinitesimal generator $L$ \textemdash a linear operator acting on functions $f : E \rightarrow \F{R}$ \textemdash of the Markov process exists, and $Lf$ is defined for all $f : E \rightarrow \F{R} $ (for details see Lemma \ref{lem: contraction_semigroup_MJP}). This property allows us to derive concentration inequalities that hold for all $f : E \rightarrow \F{R}$ ($f$ is the function used in the time average $t^{-1} \int_0^tf(X_s)ds$).
    
\end{enumerate}
\end{remark}
\begin{remark}(Equivalent definitions)
\label{rem: equivalent_definitions}
It should be remarked that there  are many different (equivalent) ways to define an MJP or, respectively  a 'continuous time Markov chain'. One may use a definition that describes the process via an underlying Markov chain and exponentially distributed  random variables (see  \cite[Ch.\,2.6]{Norris1998}, Theorem \ref{th: characterization_MJP}$(b)$). Alternatively, one may use a definition that describes an MJP as a special case of a Markov process (see \cite[Ch.\,12]{Kallenberg2002}). One may include the Markov property \eqref{eq: Markov_property}
 directly in the definition (see e.g.\cite[Def.\,2.1]{Liggett2010}), or one may use the fact that when considering a stochastic process $\F{X}$ on a countable space $E$ the Markov property  
follows also from the statement 
\begin{equation}
\label{eq: alternative_to_markov}
    \F{P}( X_{t+s} = y | X_s = x, X_{s_1} = x_1,..., X_{s_k} = x_k ) = \F{P}( X_{t+s} = y| X_s = x) = \F{P}( X_t = y | X_0 = x)
\end{equation}
for all $x,y, x_1,..x_k \in E$, all $t,s \geq 0$, and all $0 \leq s_1 \leq ....\leq s_k \leq s $, and then just require in the definition that \eqref{eq: alternative_to_markov} holds (see e.g. \cite[Ch.\,1.1]{Anderson}, \cite[Ch.\,8,Def.\,2.1]{Bremaud}). Furthermore, one may include a family $(\F{P}_x)_{x \in E}$ of ''starting measures'' , where $(\F{X},\F{P}_x)$ describes the process started at $x$, directly in the definition (see \cite[Def.\,2.1]{Liggett2010}) or define them via $\F{P}_x = \F{P}(\: \cdot \: | X_0 = x) $ (see \cite[Ch.\,8.2.1]{Bremaud}). As we shall later use the Markov property (see e.g. Lemma \ref{lem: feynman_kac}) in the form \eqref{eq: Markov_property}, we directly included it in the definition. Furthermore, we want to also work with the measures $(\F{P}_x)_{x \in E}$, so we also included them directly in the definition.
\end{remark}
To define the concept of an MJP with starting distribution $\nu$, define for any probability measure $\nu = (\nu_x)_{x \in E}$ on $E$ the probability measure
\begin{equation}
\label{eq: def_P_nu}
    \F{P}_\nu := \sum_{x \in E} \nu_x \F{P}_x.
\end{equation}
 Denote by $\F{E}_\nu$ the expectation under  $\F{P}_\nu$ and by $\mathscr{L}_\nu(Z)$ the distribution of a random element $Z$ under $\F{P}_\nu$, i.e. $\mathscr{L}_\nu(Z) = \F{P}_\nu Z^{-1}$. Notice that because of property $(b)$ in Definition \ref{def: Markov_jump_process} of an MJP we have that $\mathscr{L}_\nu(X_0) = \nu$, thus we refer to $(\F{X}, \F{P}_\nu)$ as an MJP with initial distribution $\nu$.
 Notice that the initial distribution $\nu $ and the transition function $(P(t))_{t \geq 0}$ uniquely determine the distribution of $\F{X}$ under $\F{P}_\nu$ (on the path space $(E^{[0, \infty)}, \mathcal{E}^{[0,\infty)}$)). Indeed, using the Markov property \eqref{eq: Markov_property} with $A((x_s)_{s \geq 0}) = 1_{\{ x_h = y\}}$ for $h \geq 0$, $y \in E$ and $(x_s)_{s\geq 0} \in E^{[0,\infty)}$ yields for any $x_1,...,x_n \in E$ , $0 \leq t_1 \leq ... \leq t_n$ and any $h \geq 0$ that
 \begin{align*}
     &\F{P}_x(X_{t_1} = x_1,...,X_{t_n} = x_n , X_{t_n + h} = y) = \F{E}_x[\F{E}_x(1_{ \{X_{t_n+h} = y \}}| (X_r)_{ 0 \leq r \leq t_n})1_{ \{X_{t_1} = x_1 ,...,X_{t_n} = x_n \}} ] \\
     &=\F{E}_x[\F{E}_{X_{t_n}}(1_{ \{ X_h = y \}}) 1_{ \{X_{t_1} = x_1 ,...,X_{t_n} = x_n \}}] = p_{x_ny}(h)\F{P}_x(X_{t_1} = x_1,...,X_{t_n} = x_n ).
 \end{align*}
 Consequently,
 \begin{equation*}
     \F{P}(X_{t_n+h} = y | X_{t_1}= x ,..., X_{t_{n}} = x_{n}) = p_{x_ny}(h),
 \end{equation*}
 and thus for any $0 \leq t_1 ... \leq t_n$ and any $x_1,...,x_n \in E$
\begin{align}
\label{eq: fidi_distributions}
\begin{split}
    &\F{P}_\nu(X_{t_n} = x_n,..., X_{t_1} = x_1) \\
    &= \sum_{ x \in E} \nu_x  \F{P}_x(X_{t_1} = x_1 | X_0 = x) ...\, \F{P}_x(X_{t_n} = x_n | X_0 = x , X_{t_1} = x_1,..., X_{t_{n-1}} = x_{n-1}) \\
    &=\sum_{x \in E} \nu_x  p_{xx_1}(t_1)...\,p_{x_{n-1}x_n}(t_n - t_{n-1}).
\end{split}
\end{align}
Thus, the finite dimensional distributions $\mathscr{L}_\nu(X_{t_1},...,X_{t_n})$ are determined by $\nu $ and $(P(t))_{ t \geq 0}$,  so by the uniqueness theorem for stochastic processes  $\mathscr{L}_\nu(\F{X}) $ is uniquely determined. It is straightforward to check that the transition function has the following properties:
\begin{enumerate}
    \item $P(t)$ is a stochastic matrix, i.e $p_{xy}(t) \geq 0$ and $\sum_{y \in E}p_{xy}(t) = 1$ for all $x,y \in E$ and all $t \geq 0$. 
    \item $\lim_{t \downarrow 0} p_{xy}(t) = p_{xy}(0) = \delta_{xy}$ for all $x,y \in E$
    \item $P(s+t) = P(t)P(t)$, i.e. for all $t,s \geq 0$ and all $x,y \in E$ 
    \begin{equation}
    \label{eq: chapman_kolmogorov_equation}
        p_{xy}(t+s) = \sum_{z \in E}p_{xz}(t)p_{zy}(s)
    \end{equation}
\end{enumerate}
Here 2. follows from the dominated convergence theorem and the right continuity of $\F{X}$, and \eqref{eq: chapman_kolmogorov_equation} follows from the Markov property  \cite[Ch.2,Th.\,2.12]{Liggett2010}. 
The equations \eqref{eq: chapman_kolmogorov_equation} are called \textit{Chapman-Kolmogorov equations}. Generally, any family $(P(t))_{t \geq 0}$ of matrices indexed by $E$ satisfying the above three properties is called a \textit{transition function} \cite[Ch.\,2,Def.\,2.2]{Liggett2010}. Not every transition function (in the case of infinite $E$) is given by an MJP (for a counterexample see \cite[Ch.\,2.4, Remark 2.20]{Liggett2010}). Because of the Chapman-Kolmogorov equations for the transition function  one could hope that it is possible to find a matrix $Q = (q_{xy})_{x,y \in E}$ such that $P(t) = \exp(tQ)$. This would make it possible to characterize the infinite family $(P(t))_{t\geq 0}$ (and thus the distribution of $\F{X})$  via just one matrix. If $P(t) = \exp(tQ) $, then in particular $\frac{d}{dt}P(t)|_{t = 0} = Q$. Because $\sum_{y \in E} p_{xy}(t) = 1$ and $p_{xy}(t) \geq 0 = p_{xy}(0)$ (for $x \neq y$), this implies (using finiteness of $E$ and interchanging sum and derivative)
\begin{enumerate}[$(a)$]
    \item $q_{xy} \geq 0$ for all $ x,y \in E$, $x \neq y$
    \item $\sum_{y \in E} q_{xy} = 0$ for all $x \in E$.
\end{enumerate}
This motivates the following definition (we follow the definition of \cite[Ch.\,2.1, Def.\,2.3]{Liggett2010})
\begin{definition}(Q-Matrix)
\label{def: Q_matrix}
Let $Q = (q_{xy})_{x,y \in E}$ be a matrix of real numbers. $Q$ is called a $Q$-matrix if
\begin{enumerate}[$(a)$]
    \item $q_{xy} \geq 0$ for all $ x,y \in E$, $x \neq y$
    \item $\sum_{y \in E} q_{xy} = 0$ for all $x \in E$
\end{enumerate}
For any $Q$-matrix $Q$ define $q_x := - q_{xx} = \sum_{ y \neq x} q_{xy}$.
\end{definition}
 As we will see in Theorem \ref{th: characterization_MJP}, transition functions of MJPs are exactly those that are given by $P(t) = \exp(tQ)$ for some $Q$-matrix $Q$. Before characterizing MJPs and the corresponding transition functions let us state the following result, which generalizes the Markov property \eqref{eq: Markov_property} and will be used to prove the characterization result (Theorem \ref{th: characterization_MJP}).
\begin{lemma}[Strong Markov property]
An MJP satisfies the strong Markov property, i.e. for any stopping time $\tau$, any $B \in \mathcal{E}^{[0,\infty)}$ and any $x \in E$
\begin{equation}
    \F{P}_x( (X_{t+\tau})_{t \geq 0} \in B, \tau < \infty | \mathcal{F}_\tau ) = 1_{\{\tau < \infty \} } \F{E}_{X_\tau}( (X_t)_{t \geq 0} \in B),
\end{equation}
where $\mathcal{F_\tau} = \{ A \in \mathcal{F} \, | \, A \cap \{ \tau \leq t \} \in \sigma( (X_s)_{0 \leq s \leq t}) \} $ denotes the $\sigma$-algebra of $\tau$-past.
\label{lem: strong_markov}
\end{lemma}
\begin{proof}
 \cite[Ch.12,Thrm.\,12.14]{Kallenberg2002}
\end{proof}
Let us explain intuitively how an MJP may be characterized. For that consider the path $t \mapsto X_t$ of an MJP. If the process is in state $x $ at time $t$, i.e. $X_t = x$, then because of right continuity,  $X_t$ stays at $x$ for a positive time and then $X_t$ 'jumps'  to another state $y \in E$ (or $X_t$ stays forever at $x$). The total time of staying in $x$ before making a jump is called the holding time. Because of the Markov property, if the process is in state $x$ at time $t$ the future process $(X_s)_{ s \geq t }$ should behave (in distribution) like an MJP (with same transition function) started at $x$ at time $t = 0$. Thus one expects the following properties of the holding time:
\begin{enumerate}
    \item The distribution of the holding time just depends on the state $x$
    \item The distribution of the holding time should be memoryless, i.e. if $T$ denotes the holding time then $\F{P}(T > t +s | T > t) = \F{P}( T > s)$
\end{enumerate}
This would imply that $T \sim \mathrm{Exp}(q_x)$ for some $q_x \geq 0$. This is indeed the case (see proof of Theorem \ref{th: characterization_MJP}). Furthermore, if $\tau_n$ denotes the time of the $n$-th jump by the strong Markov property one expects $(X_{\tau_n})_{ n \in \F{Z}_+}$ to be a Markov chain. Thus, an MJP should be characterized by a Markov chain \textemdash describing the jumps, and exponentially distributed holding times \textemdash describing the time the MJP stays at some state. More precisely, we have:
\begin{theorem}[Characterization of MJPs]
\label{th: characterization_MJP}
\begin{enumerate}[\lb a\rb]
    \item Let $(\F{X},(\F{P}_x)_{x \in E})$ be an MJP with transition function $(P(t))_{t \geq 0}$. Define recursively the stopping times $\tau_0 := 0$ and
\begin{equation}
    \tau_n := \inf \{ t \geq \tau_{n-1} | X_t \neq X_{\tau_{n-1}} \}
\end{equation}
for $n \in \F{N}$, where $\inf \emptyset = \infty$. Furthermore, define a stochastic process $\F{Y} = (Y_n)_{n \in \F{Z}_+}$ on $E$ (recursively) via  $Y_0 := X_{\tau_0} = X_0$ and
\begin{equation}
    Y_n := 
    \begin{cases}
    X_{\tau_n} \; ; \; \tau_n < \infty \\
    Y_{n-1} \; ; \; \tau_n = \infty 
    \end{cases}.
\end{equation}
Then:
\begin{enumerate}[1.]
    \item There is a unique $Q$-matrix $Q$ such that $P(t) = \exp(tQ)$, and it is given by 
    \begin{equation}
    \label{eq: def_Q_matrix_of_MJP}
        q_{xy} = (\F{E}_x(\tau_1))^{-1} (\F{P}_x(X_{\tau_1} = y) - \delta_{xy}),
    \end{equation}
    where $q_{xy} = 0 $ if $\F{E}_x(\tau_1) = \infty $.
    \item Let $q_{xy}$ be defined as above and $q_x = -q_{xx}$. Then,  $(\F{Y}, (\F{P}_x)_{x \in E})$ is a Markov chain with transition probabilities given by 
    \begin{equation}
    \label{eq: def_underlying_MC_1}
       p_{xx} = 1 \quad and \quad p_{xy} = 0 \quad  for \,\, y \neq x 
    \end{equation} 
    if $q_x = 0$, and by  
     \begin{equation}
     \label{eq: def_underlying_MC_2}
        p_{xx} = 0 \quad and \quad p_{xy} =  \frac{q_{xy}}{q_x} \quad  for \,\, y \neq x 
     \end{equation}
     if $q_x > 0$.
   \item There is a sequence of random variables $(R_k)_{k \in \F{N}}$ such that for all $x \in E$ and all $n \in \F{Z}_+$
    \begin{equation}
    \label{eq: tau_n}
        \tau_n = \sum_{k = 1}^n \frac{R_k}{q_{Y_{k-1}}} \quad \F{P}_x\text{- }a.s.
    \end{equation}
    Hereby, for all $x \in E$, with respect to $\F{P}_x$, the sequence $(R_k)_{k \in \F{N}}$ is independent of $\F{Y}$ and i.i.d. with 
    \begin{equation}
        \mathscr{L}_x(R_k) = Exp(1),
    \end{equation}
    where $Exp(1)$ denotes the exponential distribution with expectation $1$. In other words, for all $x \in E$ it holds that given $\F{Y}$ the  holding times $(\tau_n - \tau_{n-1})_{n \in \F{N}}$ are independent and exponentially distributed with respect to $\F{P}_x$, i.e. for all $n \in \F{N}$ and $r \geq 0$
   \begin{equation}
       \F{P}_x(\tau_{n}-\tau_{n-1} > r \,| \, \F{Y}) = e^{-q_{Y_{n-1}}r}.
   \end{equation}
  \end{enumerate}
  \item Conversely, let $Q$ be a $Q$-matrix, $(\F{Y}, (\F{P}_x)_{x \in E})$ a Markov chain with transition probabilities as above $($see \eqref{eq: def_underlying_MC_1} and \eqref{eq: def_underlying_MC_2}$)$, and  $(R_n)_{n \in \F{N}}$  a sequence of  random variables such that for all $x$, with respect to $\F{P}_x$, the sequence is independent of the Markov chain $\F{Y}$ and i.i.d. with $\mathscr{L}_x(R_1) = Exp(1)$. Furthermore, let  $\tau_n$ be defined as in \eqref{eq: tau_n} and define the stochastic process $\F{X} = (X_t)_{t \geq 0}$ by
  \begin{equation}
  \label{eq: def_MJP_holding_times}
      X_t = Y_n \quad for \,\, t \in [\tau_n, \tau_{n+1}).
  \end{equation}
  Then \lb after removing a common nullset of all $\F{P}_x$'s\rb,  the process  $(\F{X}, (\F{P}_x)_{x \in E})$ is a well defined MJP on $E$ with transition function $P(t) = \exp(tQ)$ and $P(t)$ is the unique transition function satisfying $Q = \frac{d}{dt}P(t)|_{t = 0}$.
\end{enumerate}

\end{theorem}
\begin{proof}[Proof of Theorem \ref{th: characterization_MJP}]
\text{\space} 
\begin{enumerate}[\lb a\rb]
     \item[Part ($a$)]:\begin{enumerate}[1.]
         \item Note that if $P(t) = \exp(tQ)$ then $Q = \frac{d}{dt}|_{t = 0}P(t)$, so there is at most one $Q$ matrix such that $P(t) = \exp(tQ)$. Thus, it suffices to show that $Q$ defined in \eqref{eq: def_Q_matrix_of_MJP} satisfies $P(t) = \exp(tQ)$. By the existence and uniqueness theorem for differential equations ($E$ is finite) it is sufficient to show that the equation
 \begin{equation*}
     \frac{d}{dt}P(t) = QP(t)
 \end{equation*}
is satisfied. It can be shown (\cite[Ch.\,2.1, Prop.\,1.1]{Anderson}) that this equation is equivalent to the integral equations
  \begin{equation}
  \label{eq: integral_equation}
      p_{xy}(t)= \delta_{xy}e^{-q_x t} + \int_0^t e^{-q_x s}\sum_{z \neq x}q_{xz}p_{zy}(t-s) ds.
  \end{equation}
  Moreover, using the Markov property it can be shown that for any $x \in E$ (\cite[Ch.\,12,Lemma 12.16]{Kallenberg2002}) 
 \begin{enumerate}
     \item $\mathscr{L}_x(\tau_1) = Exp(q_x)$, where $Exp(0) = \delta_\infty$. Here $Exp$ denotes the exponential distribution and $\delta$ the Dirac-distribution. 
     \item Under $\F{P}_x$ the process $(X_{t+\tau_1})_{t \geq 0}$ is independent of $\tau_1$ (for $q_x > 0)$. 
 \end{enumerate}
 To show  \eqref{eq: integral_equation} we follow the idea of \cite[Ch.\,2.1, P.\,65]{Anderson}. Assume $q_x > 0$, otherwise \eqref{eq: integral_equation} is trivially satisfied. We get (using the disintegration theorem)
\begin{align*}
&\F{P}_x(X_t = y) =  \F{P}_x(X_t = y, \tau_1 > t) + \F{P}_x(X_t = y, \tau_1 \leq  t)\\
&= \delta_{xy}e^{-q_x t} + \int_0^t\F{P}_x( X_t = y | \tau_1 = s) \F{P}_x(\tau_1 \in ds) \\
&= \delta_{xy}e^{-q_x t}+ \int_0^t\F{P}_x( X_{t-s + \tau_1} =   y | \tau_1 = s) \F{P}_x(\tau_1 \in ds) \\
&= \delta_{xy}e^{-q_x t}+ \int_0^t\F{P}_x( X_{t-s + \tau_1} =   y) \F{P}_x(\tau_1 \in ds) \\
&= \delta_{xy}e^{-q_x t}+ \int_0^t \sum_{z \neq x}\F{P}_x(X_{\tau_1} = z)p_{zy}(t-s)q_x e^{-q_xs}ds.
\end{align*}
Whereby in the fourth line we used the independence of $\tau_1$ and $(X_{t+\tau_1})_{t \geq 0}$ and in the fifth line we used the strong Markov property (Lemma \ref{lem: strong_markov}). 
Consequently \eqref{eq: integral_equation} is satisfied and the transition function of the MJP is given by $P(t) = \exp(tQ)$
\item[2. and 3.]: This is exactly the statement of \cite[Ch.\,12, Thrm.\,12.17]{Kallenberg2002}. Let us explain the connection between the notation of \cite{Kallenberg2002} and our notation. Let $\mu$ be the kernel defined via $\mu(x,B) = \F{P}_x( X_{\tau_1} \in B)$ and $c$ be the function defined via $c(x) = \F{E}_x(\tau_1)^{-1}$, this is the notation of \cite[Ch.\,12, P.\,238]{Kallenberg2002}. Using the definition \eqref{eq: def_Q_matrix_of_MJP} of the $Q$-matrix  and the definitions of the transition probabilities $p_{xy}$ it is easy to check that 
\begin{equation}
    q_x = c(x)  \quad \text{and} \quad \mu(x,{y}) = p_{xy}
\end{equation}
for all $x,y \in E$. Thus, statements 2. and 3. are a reformulation of \cite[Ch.\,12,Thrm.\,12.17]{Kallenberg2002}.

     \end{enumerate}
     \item[Part ($b$)]:\\Finally, for part $(b)$ note that $\F{X}$ defined via \eqref{eq: def_MJP_holding_times} is a well defined, right continuous process (after removing a common nullset of all $\F{P}_x$'s). Indeed, as  the $R_k$'s are exponentially distributed we may assume $R_k > 0$ and thus  $\tau_{n+1} > \tau_{n}$ whenever $\tau_n < \infty $. Consequently, $\F{X}$ has right continuous paths (as they are piecewise constant) and is well defined on $\cup_{n \in \F{Z}_+}[\tau_n,\tau_{n+1})$. Moreover, the finiteness of $E$ implies
 $C := \sup_{x \in E} q_x < \infty$, so
    \begin{equation}
        \tau_n \geq \frac{1}{C} \sum_{k = 1}^n R_k .
    \end{equation}
    But the right hand side goes to infinity $\F{P}_x$- a.s. by the Borel-Cantelli lemma. Thus, $\F{P}_x$ - a.s. we have $ \F{R} =\dot{\cup}_{n \in \F{Z}_+}[\tau_n, \tau_{n+1})$ (disjoint union) and $\F{X}$ is a well defined, right continuous process (after removing a common nullset of all $\F{P}_x$) with $\F{P}_x(X_0 = x) = \F{P}_x( Y_0 = x) = 1$ for all $x \in E$. It remains to show that $\F{X}$ satisfies the Markov property of Definition \ref{def: Markov_jump_process}. We invoke \cite[Ch.\,2.6, Thrm.\,2.8.2]{Norris1998}. By construction of $\F{X}$ the conditions of \cite[Ch.\,2.6, Thrm.\,2.8.2(a)]{Norris1998} are satisfied and thus
    \begin{equation}
    \label{eq: pf_alternative_markov_property}
    \F{P}( X_{t+s} = y | X_s = x, X_{s_1} = x_1,..., X_{s_k} = x_k ) = \F{P}( X_{t+s} = y| X_s = x)   = p_{xy}(t)
\end{equation}
for all $x,y, x_1,..x_k \in E$, all $t,s \geq 0$, and all $0 \leq s_1 \leq ....\leq s_k \leq s $, where $(p(t))_{x,y \in E} = P(t) = \exp(tQ)$. But \eqref{eq: pf_alternative_markov_property} is equivalent to the Markov property of Definition \ref{def: Markov_jump_process}, which follows by a standard $\pi-\lambda$ argument (see also Remark \ref{rem: equivalent_definitions}).
 \end{enumerate}
\end{proof}
\begin{remark}
\begin{enumerate}[\lb $a$\rb]
    \item Theorem \ref{th: characterization_MJP}$(b)$ gives an algorithm for the simulation of an  MJP with $Q$-matrix $Q$.
    \item In the above proof  we implicitely proved, that any MJP on a finite state space is non-explosive, i.e. 
    \begin{equation}
        \F{P}_x \left ( \lim_{n \to \infty} \tau_n = \infty \right ) = 1
    \end{equation}
    for all $x$. In other words, $\F{P}_x$ - almost surely there are just finitely many jumps in finite time. 
\end{enumerate}
\end{remark}

\begin{definition}
The $Q$ matrix, defined by \eqref{eq: def_Q_matrix_of_MJP} or the relation $P(t) = \exp(tQ)$ is called the $Q$-matrix of the MJP.
\end{definition}
\begin{definition}
The Markov chain $\F{Y} = (Y_n)_{n \in \F{Z}_+}$ is called the underlying chain of the MJP
\end{definition}
Finally, as a direct consequence we obtain the well known \textit{Kolmogorov equations}.
\begin{corollary}[Kolmogorov equations]
\label{cor: kolmogorov_equations}
Let $(P(t))_{t \geq 0}$ be the transition function of an MJP with $Q$-matrix $Q$. Then 
$(P(t))_{t \geq 0}$ satisfies the Kolmogorov equations
\begin{align}
    &\text{Kolmogorov Backward Equation: } \frac{d}{dt}P(t) = QP(t)  \label{eq: KBE}\\
    & \text{Kolmogorov Forward Equation: } \frac{d}{dt}P(t) = P(t)Q
\end{align}

\end{corollary}
\begin{proof}
 Follows directly from $P(t) = \exp(tQ)$.
\end{proof}

\subsubsection{Invariant Distributions}
\label{subsubsec: invariant_distributions}
We now discuss the concept of invariant distributions, which play a central role in describing the limit behavior (e.g. the ergodic theorem) of an MJP. We define (c.f. \cite[Ch.\,2.6, Def.\,2.6.1]{Liggett2010}).
\begin{definition}[Invariant distribution]
\label{def: invariant_measure}
A nontrivial probability measure $\mu = (\mu_x)_{x \in E}$ is called invariant (or stationary) distribution if for all $t \geq 0$
\begin{equation}
    \label{eq: def_invariant_measure}
    \mu^T P(t) = \mu^T.
\end{equation}

\end{definition}
Notice that $\mu$ is a stationary distribution if and only if  $\mu = \mathscr{L}_\mu(X_t)$ for all $t \geq 0$. By the  Markov property this is equivalent to 
\begin{equation*}
    \mathscr{L}_\mu((X_{t+s})_{t \geq 0}) = \mathscr{L}_\mu((X_t)_{t \geq 0})
\end{equation*}
for all $ s \geq 0$, i.e. the process is strongly stationary.
As $E$ is finite, we can interchange summation over $E$ and the differentiation, so using the Kolmogorov backward equation (Corollary \ref{cor: kolmogorov_equations}) yields
\begin{equation*}
    \frac{d}{dt}\mu^TP(t) = \mu^T\frac{d}{dt}P(t) = \mu^T Q P(t).
\end{equation*} 

It follows that $\mu$ is an invariant distribution if and only if 
\begin{equation}
    \label{eq: def_invariant_measure_infinitesimal}
    \mu^TQ = 0 ,
\end{equation}
which reduces the infinite system of equations \eqref{eq: def_invariant_measure} to just one equation in terms of the $Q$-matrix. Using that  $\sum_{ y \in E, y \neq x} q_{xy} = -q_{xx}$ for all $x \in E$, we can write \eqref{eq: def_invariant_measure_infinitesimal} equivalently as \begin{equation}
    \label{eq: prepariation_for_detailed_balance}
    0 = \sum_{ y \in E, y \neq x}\mu_y q_{yx} - \sum_{ y \in E, y \neq x}\mu_x q_{xy}  = \sum_{y \in E, y \neq x}\mu_y q_{yx} - \mu_x q_{xy}.
\end{equation}
Intuitively, if we interpret $\mu_x q_{xy}$ as the 'probability flow' from $ x$ to $y$, then $\sum_{ y \in E, y \neq x}\mu_x q_{xy}$ is the total 'probability outflow' from $ x $ and $\sum_{ y \in E, y \neq x} \mu_y q_{yx}$ is the total 'probability inflow' to $x$. Using  \eqref{eq: prepariation_for_detailed_balance} we can interpret \eqref{eq: def_invariant_measure_infinitesimal} as  'the net probability flow vanishes'. Intuitively, the 'net probability flow' would vanish if for all $ x \neq y$ the flow from $x $ to $y$ is equal to the flow from $y$ to $x$ i.e. $\mu_x q_{xy} = \mu_y q_{yx}$, then also the equivalent condition \eqref{eq: prepariation_for_detailed_balance} for $\mu$ to be an invariant distribution would be satisfied , as each summand on the right hand side  vanishes. This observation motivates the following definition.

\begin{definition}[Detailed balance]
We say a probability measure $\mu = (\mu_x)_{x \in E}$ on $E$ satisfies the detailed balance condition if for all $x,y \in E$
\begin{equation}
    \label{def: detailed_balance}
    \mu_x q_{xy} = \mu_y q_{yx}.
\end{equation}
\end{definition}
Thus, the detailed balance condition is an 'easier' set of equations that are sufficient for $\mu$ to be an invariant distribution. However, these equations are not always solvable, the following remark gives a brief insight on when a measure satisfying the detailed balance condition exists.
\begin{remark}(Kolmogorov cycle criterion)
A natural question to ask is when a probability measure satisfying the detailed balance condition exists. The so called \textit{Kolmogorov cycle criterion} provides information about the existence of a measure satisfying the detailed balance condition. The Kolmogov cycle criterion states that for all closed paths $x, x_1,...,x_n, x $ in $E$
\begin{equation}
    q_{xx_1}q_{x_1x_2}.....q_{x_nx} = q_{xx_n}q_{x_nx_{n-1}}.....q_{x_1x}.
\end{equation}
This criterion is essentially equivalent to the existence of an measure that satisfies detailed balance (see \cite[Ch.\,7.1, Thrm.\,1.3]{Anderson} for details of this equivalence). However, for our purposes this criterion is not relevant and we shall not discuss it further. For a more detailed presentation of the Kolmogorov cycle criterion see \cite[Ch.\,7.1]{Anderson}
\end{remark}
 Not every MJP admits an invariant distribution. The following condition guarantees the  existence and uniqueness an invariant distribution for an MJP (c.f. Theorem \ref{th: uniqueness_existence_invariant_measure}).
 \begin{definition and lemma}[Irreducibility]
 \label{def: irreducibility}
 For an MJP the following statements are equivalent :
 \begin{enumerate}[\lb $a$\rb]
    \item The underlying Markov chain $\F{Y}$ is irreducible, i.e. for all $x,y \in E$
    \begin{equation*}
    \F{P}_x(\exists \; n \in \F{Z}_+ : Y_n = y) > 0.
\end{equation*}
    \item For all $x,y \in E$
    \begin{equation*}
    \F{P}_x(\exists \; t \geq 0 : X_t = y) > 0.
\end{equation*}
    \item For all $x,y \in E$ and all $t  > 0$
    \begin{equation*}
        p_{xy}(t) > 0.
    \end{equation*}
    \item For all $x \neq y$ there is a $n \in \F{N}$ and a sequence $ x_0 := x,x_1,...,x_n := y$ such that 
    \begin{equation}
        q_{x_0x_1}...q_{x_{n-1}x_n} > 0.
    \end{equation}
\end{enumerate}
If any of the above conditions are satisfied we call an MJP irreducible.
 \end{definition and lemma}

\begin{proof}The equivalence of statements $(a), (c)$ and $(d)$ is shown in \cite[Ch.\,5.3,Prop.\,3.1]{Anderson}. It should be remarked that although \cite{Anderson} shows the equivalence of $(a), (c), (d)$ for  $P(t) = F(t)$, where $F(t)$ is the minimal solution of the backward equation \eqref{eq: KBE} (see \cite[Ch.\,2.2,Thrm.\,2.2]{Anderson} for a definition of the minimal solution $F(t)$), the theorem can still be applied to our setting because by the uniqueness theorem for differential equations and Theorem \ref{th: characterization_MJP} we automatically have $F(t) = \exp(tQ) = P(t)$. Finally, statements $(a)$ and $(b)$ are equivalent because by definition of the underlying jump chain $\F{Y}$ (see Theorem \ref{th: characterization_MJP}) we have
\begin{equation*}
    \{\,Y_n \,|\, n \in \F{Z}_+\,\} = \{ \,X_t\, |\, t \geq 0\,\}
\end{equation*}
(for all $\omega \in \Omega$) and consequently
\begin{equation*}
    \F{P}_x(\exists \; t \geq 0 : X_t = y) = \F{P}_x( \exists \; n \in \F{Z}_+ : Y_n = y).
\end{equation*}
\end{proof}
Thus, irreducibility means that starting from any state $x \in E$, the chain  can always reach any other state $y \in E$ with positive probability. For irreducible MJPs on a finite state space we have the following theorem on the existence and uniqueness of invariant distributions.
\begin{theorem}[Existence and Uniqueness of invariant Distributions]
Let $\F{X}$ be an irreducible MJP on $E$. Then there is a unique strictly positive  invariant distribution $\pi = (\pi_x)_{x \in E}$, i.e. $\pi_x > 0 $ for all $x \in E$.

\label{th: uniqueness_existence_invariant_measure}
\end{theorem}
\begin{proof}
Follows applying \cite[Ch.\,8.5,Thrm.\,5.1]{Bremaud} to an irreducible MJP on a finite state space $E$. The conditions of \cite[Ch.8.5,Thrm.\,5.1]{Bremaud} are satisfied, because  $\F{X}$ is irreducible in the sense of definition \cite[Ch.\,8.5,Def.\,5.1]{Bremaud} by assumption. Furthermore,  any irreducible Markov chain on a finite state space is recurrent (see \cite[Ch.\,3.3,Thrm.\,3.3]{Bremaud}) and consequently $\F{X}$ is recurrent in the sense of definition \cite[Ch.\,8.5,Def.\,5.2]{Bremaud}. Moreover, any invariant measure $\mu$ on $E$ defines a invariant distribution $\pi$ by setting $\pi = \frac{\mu}{\mu(E)}$ and thus \cite[Ch.\,8.5,Thrm.\,5.1]{Bremaud} yields the uniqueness and existence of a unique strictly positive invariant distribution.
\end{proof}

\subsubsection{Infinitesimal Generators}
\label{subsubsec: infinitesimal_generators}
In this section we present a brief overview of the concept of infinitesimal generators and explain how the situation simplifies for a finite state space $E$. Later, to derive concentration inequalities, the  infinitesimal generator of the MJP will be of central importance. We start with a more general concept of infinitesimal generators of semigroups on Banach spaces and then analyze the situation for an (irreducible) MJP. We start by defining (c.f. \cite[Ch.\,1.4]{Anderson})

\begin{definition}[Semigroup of operators] Let $(V, \norm{\cdot})$ be a (real) Banach space. A family $(P_t)_{t \geq 0}$ of operators on $V$ is called a  semigroup of operators if
\begin{enumerate}[\lb $a$\rb]
    \item $P_{s+t} =P_sP_t$ for all $s,t \geq 0$
    \item $P_0 = 1$.
\end{enumerate}
If furthermore  $\norm{P_t} \leq 1$ for all $ t \geq 0$, then $(P_t)_{t \geq 0}$ is called a contraction semigroup.
A semigroup is called continuous if for any  $v \in V$ $\norm{P_hv-v} \to 0$ as $h \to 0$.
\end{definition}
\begin{definition}[Infinitesimal Generator]
\label{def: infinitesimal_generator}
Let $(P_t)_{t \geq 0}$ be a   semigroup on $V$. Define 
\begin{equation*}
    D = \left \{ v \in V \middle  | \; \lim \limits_{ h \to 0} \frac{P_hv-v}{h} \; \text{exists in V} \right \}
\end{equation*}
and for $v \in D$ define 
\begin{equation}
    \label{eq: def_infinitesimal_generator}
    Lv =  \lim \limits_{ h \to 0}  \frac{P_hv-v}{h}.
\end{equation}
Then $L : D \rightarrow V$ is called infinitesimal generator of $(P_t)_{t\geq 0}$. Denote by $D(L) = D$ the domain of $L$.
\end{definition}
\begin{remark}(Finite dimensional $V$)
\label{rem: semigroup_generator_finite_dim}
  If $V$ is finite dimensional then given any operator $L : V \rightarrow V$ there is a unique semigroup $(P_t)_{t \geq 0}$ having infinitesimal generator $L$, that is given by $P_t = \exp(tL)$. We call $(P_t)_{ t \geq 0}$ the semigroup generated by $L$.
\end{remark}

A general Markov process (see \cite[Ch.\,17]{Klenke2013} or \cite[Ch.\,8]{Kallenberg2002} for a definition of general Markov processes) $(\F{X}, (\F{P}_x)_{x \in E})$ on some general (polish) space $E$  defines in a natural way a semigroup of operators $(P_t)_{t \geq 0}$ on the Banach space of bounded, measurable real-valued functions $(\mathcal{B}(E), \norm{\cdot}_\infty)$ by
\cite[Ch.\,1.2]{Levy_matters}
\begin{equation}
    \label{eq: contraction_semigroup_Markov_process}
    (P_tf)(x) :=  \F{E}_x(f(X_t)),
\end{equation}
where the semigroup property $P_{s+t} = P_sP_t$ follows from the Markov property. If $\pi $ is an invariant distribution of the Markov process (see  \cite[Ch.\,8]{Kallenberg2002} for a general definition of invariant distributions), then $(P_t)_{t \geq 0}$ can be seen as a (well defined) contraction semigroup on $L^p(\pi)$ for all $p \geq 1$ (see Remark \ref{rem: extension_semigroup_to_L^p}). The works \cite{wu}, \cite{guillin}, \cite{lezaud}, \cite{bernstein} considered here that derive concentration inequalities for more general Markov processes (than MJPs) consider $(P_t)_{t \geq 0}$ on the function space $L^2(\pi)$ and not directly on $\mathcal{B}(E)$. This has the reason that in order to derive concentration inequalities these works use the inner product on $L^2(\pi)$. In our setting we consider an irreducible MJP with invariant distribution $\pi$ on a finite state space $E$ and we can directly identify (see Remark \ref{rem: simplification_finite_state_space}) $L^2(\pi) = \mathcal{B}(E)$, thus it does not depend whether we formally treat $(P_t)_{ t \geq 0}$ as a semigroup on $L^2(\pi)$ or $\mathcal{B}(E)$. Remark \ref{rem: simplification_finite_state_space} explains in more detail how the situation simplifies in our setting.
\begin{remark}
\label{rem: extension_semigroup_to_L^p}
Using Jensen's inequality (Theorem \ref{th: jensen_inequality}) it can be shown (see \cite[Ch.\,4, Lemma 4.2]{eberle}) that if $\pi$ is an invariant distribution of a general Markov Process $(\F{X}, (\F{P}_x)_{x \in E})$, then  $(P_t)_{ t \geq 0}$ is  a well defined contraction semigroup on $L^p(\pi)$, where well-definedness means that for an equivalence class $[f] \in L^p(\pi)$ definition \eqref{eq: contraction_semigroup_Markov_process} does not depend on the member of $[f]$ (up to a $\pi$ nullset) and $\abs{P_tf}^p$ is  integrable with respect to $\pi$.
\end{remark}

\begin{remark}(Simplification in our setting)
\label{rem: simplification_finite_state_space}
More generally, if one considers a general Markov process on a Polish space $E$ (see \cite[Ch.17.1]{Klenke2013} for a definition of a Markov process), the Banach space $\mathcal{B}(E)$ may be 'too big' and $(P_t)_{ t \geq 0}$ may not have desirable properties like strong continuity  ($\norm{P_hf -f}_\infty \xrightarrow[]{h \to 0} 0 )$. Furthermore, one may be interested in using the scalar product of $L^2(\pi)$, where $\pi$ is an invariant measure of the Markov process. Thus, in general one can consider $(P_t)_{ t \geq 0}$ on the following (real) Banach spaces (for a precise definition of these Banach spaces see \cite{Engel2006}) 

\begin{enumerate}[$\bullet$]
    \item $\mathcal{B}(E)$ 
    \item $\mathcal{C}_b(E)$ := \{ $f : E \rightarrow \F{R}$ | $f$ is bounded and continuous \}  \\
    \item $\mathcal{C}_0(E)$ := \{ $f : E \rightarrow \F{R}$ | $f$ is continuous and vanishes at infinity \} \\
    \item $L^2(\pi)$,
\end{enumerate}
where the spaces $\mathcal{B}(E), \mathcal{C}_b(E) $ and $\mathcal{C}_0(E)$ are endowed with the sup-norm $\norm{f}_\infty = \sup_{x \in E}\abs{f(x)} $, and $L^2(\pi)$ is endowed with the usual $L^2$-norm. In general, these spaces are not equal and not isomorphic (as Banach spaces). Consequently, in a general setting one has to be precise when referring to 'the  infinitesimal generator of the Markov process' as  Definition \ref{def: infinitesimal_generator} depends on the Banach space, on which one analyses the semigroup $(P_t)_{ t \geq 0}$. However, in our case $E$ is just a finite discrete space (endowed with the discrete topology), $\pi$ is the invariant measure of an irreducible MJP, and all above function spaces coincide in the following way.\\
Denote by $\F{R}^E$ the vector space of all functions $f : E \rightarrow \F{R}$. Notice that as $E$ is endowed with the discrete topology and $E$ is finite, all $f \in \F{R}^E$ are bounded, continuous and vanish at infinity. So $\F{R}^E = \mathcal{B}(E) = \mathcal{C}_b(E) = \mathcal{C}_0(E)$. Furthermore $\pi_x > 0$ for any $x \in E$ (see Theorem \ref{th: uniqueness_existence_invariant_measure}) so $\mathscr{L}^2(\pi) = L^2(\pi)$, where $\mathscr{L}^2(\pi)$ denotes the space of measurable $\pi$ - square integrable functions (recall that elements of $L^2(\pi)$ are equivalence classes).  But by finiteness of $E$,  for any $f \in \F{R}^E$ we have  $\sum_{x \in E} f(x)^2 \pi_x < \infty   $, so $L^2(\pi) = \mathscr{L}^2(\pi) = \F{R}^E $. Furthermore as $\F{R}^E$ is finite dimensional all norms on $\F{R}^E$ are equivalent and pointwise convergence coincides with convergence in norm: for any sequence $(f_n)_{n \in \F{N}}$ and any $f$ in $\F{R}^E$  we have $ \norm{f_n - f} \xrightarrow[]{n \to \infty } 0 $ if and only if $f_n(x) \xrightarrow[]{n \to \infty}f(x)$ for all $x \in E$ (where $\norm{\cdot }$ is some arbitrary norm). In particular, for the definition of the infinitesimal generator of $(P_t)_{ t \geq 0}$ it does not matter on which function space one considers $(P_t)_{ t \geq 0}$, and in Definition \ref{def: infinitesimal_generator} the convergence in norm may be replaced by pointwise convergence.
\end{remark}
The above remark shows that  in our setting, where we consider an irreducible MJP on a finite set, the semigroup $(P_t)_{ t \geq 0}$ defined in  \eqref{eq: contraction_semigroup_Markov_process} is  also a (well defined) semigroup  on $L^2(\pi)$ (without having to refer to Remark \ref{rem: extension_semigroup_to_L^p}) as we can identify $L^2(\pi) = \mathcal{B}(E)$. Moreover, the definition of the infinitesimal generator is independent of the function space considered (on which $(P_t)_{t \geq 0}$ is defined), thus we will use the notion \textit{the infinitesimal generator of an \lb irreducible\rb \space  MJP \lb with invariant distribution $\pi$\rb}. From now on (unless stated otherwise), $L$ always denotes the infinitesimal generator of an (irreducible) MJP. Mostly we will consider $L$, to be treated as an operator on $L^2(\pi)$, but for the rest of this work we make the identification $L^2(\pi) = \mathcal{B}(E) = \mathcal{C}_b(E) = \mathcal{C}_0(E) = \F{R}^E$, whenever needed. Using Theorem \ref{th: characterization_MJP}$(a)$ we can directly compute the infinitesimal generator $L$ in terms of the $Q$-matrix of the MJP. We have

\begin{lemma}[Semigroup and infinitesimal generator of an irreducible MJP]
\label{lem: contraction_semigroup_MJP}
Let $L$ be the infinitesimal generator and $(P_t)_{ t \geq 0} $ the semigroup of an irreducible MJP. Then, $(P_t)_{ t \geq 0}$ is a continuous contraction semigroup \lb on $L^2(\pi)\rb$ and  $D(L)  = L^2(\pi)  = \mathcal{B}(E)$. Furthermore, for all $f \in \mathcal{B}(E)$ and all $x \in E$ 
\begin{equation}
    \label{eq: contraction_semigroup_MJP}
    (P_tf)(x) = \sum_{y \in E}p_{xy}(t) f(y)
\end{equation}
and
\begin{equation}
    \label{eq: infinitesimal_generator_MJP}
    (Lf)(x) = \sum_{y \in E}q_{xy}f(y).
\end{equation}
In other words, if $(e_x)_{x \in E}$ denotes the basis of $\mathcal{B}(E) = L^2(\pi)$ given by $e_x(y) = \delta_{xy}$, then the transformation matrices of $P_t$ and $L$ in this basis are given by $P(t)$ and $Q$, where $P(t)$ and $Q$ denote the transition function and $Q$-matrix of the MJP.
\end{lemma}
\begin{remark}
\label{rem: lemma_infinitesimal_generator_MJP}
\begin{enumerate}[\lb $a$\rb]
    \item  The works \cite{wu}, \cite{lezaud}, \cite{guillin}, \cite{bernstein} considered here treat $(P_t)_{ t \geq 0}$ as a contraction semigroup on $L^2(\pi)$ (for more general Markov processes than MJPs) and consequently in these works the notion of $L^2$-infinitesimal generator (defined as in Definition \ref{def: infinitesimal_generator}) is used. Note that for general Markov processes we may not have $D(L) = L^2(\pi)$.
    \item The above Lemma can be quite easily generalized (using a similar proof) to general (not necessarily irreducible) MJPs on a finite state space, where $\pi$ is replaced by some invariant distribution  $\mu$, and we would have the identification $L^2(\mu) = \mathcal{B}(\text{supp}(\mu))$, where $\text{supp}(\mu)$ denotes the support of $\mu$. Furthermore, in  \eqref{eq: contraction_semigroup_MJP} and \eqref{eq: infinitesimal_generator_MJP} the summation over $x,y \in E$ would be replaced by a summation over $x,y \in \text{supp}(\mu)$.
\end{enumerate}

\end{remark}

\begin{proof}[Proof of Lemma \ref{lem: contraction_semigroup_MJP}]
In the following proof we will throughout use  the identification  $\mathcal{B}(E) = L^2(\pi)$ and that the convergence with respect to $\norm{\cdot}_\infty$ and $\norm{\cdot}_2$ respectively, is equivalent to pointwise convergence (Remark \ref{rem: simplification_finite_state_space}). Equality \eqref{eq: contraction_semigroup_MJP} follows directly from the definitions.
  Thus, using \eqref{eq: contraction_semigroup_MJP}, finiteness of $E$ and continuity of $p_{xy}(t) $ we obtain the continuity of $(P_t)_{ t \geq 0}$. Furthermore, $(P_t)_{t \geq 0}$ is a contraction semigroup, because by Jensen's inequality (Theorem \ref{th: jensen_inequality}) for all $f \in L^2(\pi)$ and $x \in E$ we have $(P_tf)(x)^2= \F{E}_x(f(X_t))^2 \leq \F{E}_x( f^2(X_t)) $. Thus,
 \begin{equation}
    \norm{P_tf}_2^2 =  \F{E}_\pi( (P_tf)^2) \leq \F{E}_\pi(f^2(X_t)) = \norm{f}_2^2,
 \end{equation}
 where we used that $\pi$ is a stationary distribution.
 Furthermore for all $f \in \mathcal{B}(E)$  and $x \in E$ it holds
\begin{equation*}
    \frac{d}{dt}\Bigr|_{t = 0}(P_tf)(x) = \frac{d}{dt}\Bigr|_{t = 0} \sum_{y \in E}p_{xy}(t)f(y) = \sum_{y \in E}\frac{d}{dt}\Bigr|_{t = 0}p_{xy}(t)f(y) = \sum_{y \in E}q_{xy}f(y),
\end{equation*}
where Theorem \ref{th: characterization_MJP}.1$(a)$ and finiteness of $E$ (to interchange sum and derivative) were used.
So $D(L) =  \mathcal{B}(E) = L^2(\pi)$ and \eqref{eq: infinitesimal_generator_MJP} holds.

\end{proof}

The above lemma shows that in our setting (of an irreducible MJP) we can identify $L \; \widehat{=} \; Q$.
We will now discuss some properties of the infinitesimal generator $L$, that will be used later to derive concentration inequalities.
Let $\textbf{1} : E \rightarrow \{1 \}$ denote the constant $1$-function, and denote by $L^*$ the adjoint of the infinitesimal generator $L : L^2(\pi) \rightarrow L^2(\pi)$ of an irreducible MJP with invariant distribution $\pi$. Moreover, denote by $\langle \cdot, \cdot \rangle$ the scalar product on $L^2(\pi)$ and by $\norm{\cdot}_2$ the $L^2$ norm (for functions and operators).
\begin{lemma}[Properties of the infinitesimal Generator]
\label{lem: properties_infinitesimal_generator}
Let $L$ be the infinitesimal generator of an irreducible MJP and denote by $\sigma(\cdot)$ the spectrum of an operator. Then,
\begin{enumerate}[\lb a\rb]
    \item $\mathrm{Ker}(L) = \mathrm{Ker}(L+ L^*) = \mathrm{span}( \textup{\textbf{1}} )$, in particular $0$ is a simple eigenvalue.
    \item $\mathrm{Im}(L) = \mathrm{Im}(L + L^*) = \{ \textup{\textbf{1}} \}^{\perp}  = \{ f \in L^2(\pi) | \pi(f) = \langle \textup{\textbf{1}}, f \rangle_{L^2(\pi)} = 0 \}$
    \item $-L, -( L+L^*)$ are positive semidefinite.
    \item $\sigma(L+ L^*) \subset \F{R}_{\leq 0} $
    
\end{enumerate}
\end{lemma}
\begin{proof}
Let $L(x,y) = (Le_y)(x) = q_{xy}$ and $L^*(x,y) = (L^*e_y)(x) $, where $(e_x)_{x \in E}$ denotes the basis given by $e_x(y) = \delta_{xy}$, i.e. $(L(x,y))_{x,y \in E}$ and $(L^*(x,y))_{x,y \in E}$ are the transformation matrices of $L$ and $L^*$ with respect to that basis. A calculation using  $\langle f, g \rangle = \sum_{x \in E} f(x)g(x)\pi_x$ and Lemma \ref{lem: contraction_semigroup_MJP}, shows  $L^*(x,y) = \frac{L(y,x) \pi_y}{\pi_x} = \frac{q_{yx} \pi_y}{\pi_x}$. Furthermore, using  invariance of $\pi$, $\sum_{y \in E} q_{xy} = 0$, $L^*(x,y) \geq 0$ for $x \neq y$, and Lemma \ref{def: irreducibility}$(d)$ it is easily checked that the matrix $\Tilde{Q}$ defined via $\Tilde{q}_{xy} = L(x,y) + L^*(x,y)$ is a $Q$-Matrix (Definition \ref{def: Q_matrix}) defining an irreducible MJP  with (unique) invariant distribution $\pi$. Thus, it suffices to prove statements $(a) - (c)$ just for $L$, which we identify with the $Q$-matrix $Q$. We start by proving statement $(a)$. For all $\lambda \in \sigma(Q)$ we have $\Re \lambda \leq 0$. Indeed, by the Gershgorin disc theorem \cite[Ch.\,6.1,Thrm.\,6.1.1]{horn_matrix_analysis} for all $\lambda \in \sigma(Q)$ there is a $x \in E$ such that
\begin{equation*}
   \abs{\lambda + q_x} =  \abs{\lambda - q_{xx}} \leq \sum_{y \neq x}q_{xy} = q_x.
\end{equation*}
It follows that
\begin{equation*}
 \Re \lambda + q_x \leq  \abs{\Re \lambda + q_x} \leq \abs{ \lambda + q_x} \leq q_x
\end{equation*}
and thus 
\begin{equation}
\label{eq: Real_part_eigenvalue_Q}
\text{Re} \lambda \leq 0
\end{equation}
for any $\lambda \in \sigma(Q)$. Let $ t > 0$, note that  $P(t) = \exp(tQ)$ (Theorem \ref{th: characterization_MJP}.1$(a)$) has spectrum $\exp(t\sigma(Q))$, which follows, for example, by using the Jordan normal form of $Q$ to calculate $\exp(tQ)$. Thus, by \eqref{eq: Real_part_eigenvalue_Q}  we have $\rho(P(t)) \leq 1$, where $\rho$ denotes the spectral radius. But clearly $ 1 \in \sigma(P(t))$ (as $L\textbf{1} = 0$) so $\rho(P(t)) = 1$. By irreducibility we have    $p_{xy}(t) > 0$ for all $x,y \in E$ (Lemma \ref{def: irreducibility}), so by Perrons theorem \cite[Ch.8.2, P.667]{Meyer2000} $1 = \rho(P(t))$ is a simple eigenvalue. As any $v \in \mathrm{Ker}(Q)$ satisfies $P(t)v = v$, this implies that $\dim \mathrm{Ker}(L) = \dim \mathrm{Ker}(Q) \leq 1$. But $L\textbf{1} = 0$ and  consequently $\mathrm{Ker}(L) = \mathrm{span}(\textbf{1})$, which is statement $(a)$. As by the rank-nullity theorem $\dim \mathrm{Im}(L) = \dim L^2(\pi)- \dim \mathrm{Ker}(L) = \dim L^2(\pi) - 1 = \dim \{ \textbf{1} \}^{\perp}$, in order to show $(b)$ it suffices to show that $\mathrm{Im}(L) \subset \{ \textbf{1} \}^{\perp}$. This follows from $\pi^TQ = 0 $. Indeed, for any $f = Lg$ we have that
\begin{equation*}
    \langle \textbf{1}, f \rangle  = \langle \textbf{1}, Lg \rangle = \sum_{x \in E}\pi_x \sum_{y \in E} q_{xy}g(y) = \sum_{y \in E}g(y) \sum_{x \in E}\pi_x q_{xy} = 0,
\end{equation*}
which proves $\text{Im}(L) \subset \{\textbf{1}\}^\perp$ and thus $(b)$. To prove $(c)$, note that
for any $f \in L^2(\pi)$ we have
\begin{equation}
    \langle f , P_t f \rangle \leq \norm{f}_2^2 = \langle f, P_0f \rangle,
\end{equation}
where we used the Cauchy-Schwartz inequality and the fact that $P_t$ is a contraction (see Lemma \ref{lem: contraction_semigroup_MJP} or Remark \ref{rem: extension_semigroup_to_L^p}. Consequently \begin{equation}
    \langle f, Lf \rangle = \frac{d}{dt} \langle f, P_tf \rangle |_{t = 0}  \leq 0
\end{equation}
for any $f \in L^2(\pi)$, i.e. $-L$ is positive semidefinite. Finally, statement $(d)$ follows from the selfadjointness of $L+L^*$ and the negative semidefiniteness of $L+L^*$.
\end{proof}

We finish this section by analyzing the special case if the detailed balance condition (Definition \ref{def: detailed_balance}) is satisfied. We have
\begin{theorem}[Characterization of Detailed Balance]
\label{th: characterization_detailed_balance}
Let $\mu = (\mu_x)_{x \in E}$ be a probability measure on $E$ and $(\F{X}, (\F{P}_x)_{x \in E})$ an MJP. The following statements are equivalent:
\begin{enumerate}[\lb a\rb]
    \item $\mu$ satisfies the detailed balance condition.
    \item The process $(\F{X}, \F{P}_\mu)$ is reversible i.e.
    \begin{equation*}
      \mathscr{L}_\mu((X_t)_{0 \leq t \leq s }) = \mathscr{L}_\mu((X_{s-t})_{0 \leq t \leq s}) 
    \end{equation*}
    for any $s \geq 0$.
    \item $P_t$ is $\mu$ symmetric, i.e. for all $A, B \subset E$
    \begin{equation*}
        \int_A (P_t 1_{B})(x) \mu(dx) = \int_B (P_t 1_A)(x) \mu(dx).
    \end{equation*}
    \item $\mu$ is an invariant distribution and the infinitesimal generator on $L^2(\mu)$ (defined as in Definition \ref{def: infinitesimal_generator}) is selfadjoint, i.e.
\begin{equation*}
    \langle Lf, g \rangle_{L^2(\mu)} = \langle f, Lg \rangle_{L^2(\mu)}
\end{equation*}
for all $f, g \in L^2(\mu)$.
       
\end{enumerate}

\end{theorem}
\begin{proof}
We follow the proof idea of \cite[Sec.\,4.3,Thrm.\,4.20]{eberle}. We show $(d) \to (c) \to (b) \to (a) \to (d)$. Assume $(d)$, then $(P_t)_{ t \geq 0}$ can be seen as a contraction semigroup on $L^2(\mu)$ (Remark \ref{rem: extension_semigroup_to_L^p}) and $D(L) = L^2(\mu)$ (Remark \ref{rem: lemma_infinitesimal_generator_MJP}$(b)$), so by finite dimensionality of $L^2(\mu)$ (see Remark \ref{rem: semigroup_generator_finite_dim}) we have $P_t = \exp(tL)$, as  operators on $L^2(\mu)$. Consequently, by selfadjointness of $L$, $P_t$ is also selfadjoint. In particular for $A,B \subset E$
\begin{equation}
    \int_A (P_t 1_{B})(x) \mu(dx) = \langle 1_A, P_t1_B \rangle_{L^2(\mu)} = \langle P_t1_A, 1_B \rangle_{L^2(\mu)} = \int_B (P_t 1_A)(x) \mu(dx),
\end{equation}
which proves $(c)$. Now assume $(c)$ and let $P(t) = (p_{xy}(t))_{x,y \in E}$ denote the transition function of the MJP, then in particular (set $A = \{x \}, B = \{y\}$)
\begin{equation}
\label{eq: detailed_balance_p_t}
    \mu_x p_{xy}(t) = \mu_y p_{yx}(t).
\end{equation}
Note that by summing over $y \in E$, this implies invariance of $\mu$, in particular $\mu = \mathscr{L}_\mu (X_t)$ for all $t \geq 0$.
As finite dimensional distributions uniquely determine the distribution of a stochastic process, to show $(b)$ it is sufficient to show that  for all $n \in \F{N}$ all $x_1,...,x_n \in E$ and all $0 \leq t_1 .... \leq t_n \leq s$
\begin{equation}
      \F{P}_\mu(X_{t_1} = x_1,..., X_{t_n} = x_n) = \F{P}_\mu ( X_{s - t_n } = x_n,...., X_{s - t_1} = x_1 ),
\end{equation}
which is equivalent to (we use \eqref{eq: fidi_distributions} and invariance of $\mu$)
\begin{equation}
    \mu_{x_1} p_{x_1x_2}(t_2-t_1) ... p_{x_{n-1}x_n}(t_n-t_{n-1}) =  \mu_{x_n} p_{x_nx_{n-1}}(t_n - t_{n-1}) ... p_{x_2x_1}(t_2 - t_1) .
\end{equation}
But the above equality follows from applying repeatedly  \eqref{eq: detailed_balance_p_t} with $x = x_j, y = x_{j+1}$ for $j = 1,...,n-1$ (from left to right). Assume $(b)$, then  in particular $\mu$ is invariant (as $(\F{X},\F{P}_\mu)$ is strongly stationary), and for all $s \geq 0$, $x,y \in E$
\begin{equation}
    \mu_x p_{xy}(s) = \F{P}_\mu(X_0 = x, X_s = y) = \F{P}_\mu( X_s = x, X_0 = y) = \mu_y p_{xy}(s).
\end{equation}
Thus, taking the derivative at $s = 0$ yields the detailed balance condition $(a)$. Finally, assume $(a)$. In particular, $\mu$ is an invariant distribution. Denote by $\text{supp}(\mu)$ the support of $\mu$ and let $f,g \in L^2(\mu)$. By Lemma \ref{lem: properties_infinitesimal_generator} and Remark \ref{rem: lemma_infinitesimal_generator_MJP} we have  (using the detailed balance condition $(a)$)
\begin{equation} 
    \langle f, Lg \rangle_{L^2(\mu)} = \sum_{x,y \in \text{ supp}(\mu)} f(x) g(y) \mu_x q_{xy} = \sum_{x, y \in \text{ supp}(\mu)} g(y) f(x) \mu_y q_{yx} =  \langle Lf, g \rangle_{L^2(\mu)},
\end{equation}
which proves self adjointness of $L$.
\end{proof}
\begin{remark}
The equivalence of the above statements is not explicitly used in this work. However, since the works \cite{lezaud}, \cite{wu}, \cite{guillin}, \cite{bernstein}  use (equivalent) notions of  'reversible' and 'symmetric' Markov processes, we decided to include the above lemma in order to clarify  these notions.
\end{remark}

\subsubsection{Limit Behavior}
\label{subsubsec: limit_behaviour}
We present here two central theorems that describe the limit behavior of an irreducible MJP. In this thesis we focus mainly on the ergodic theorem, the central limit theorem serves as background information, as for reversible MJPs the so called asymptotic variance (see Theorem \ref{th: CLT}) will appear as a parameter in a bound of $\mathbb{P}( \frac{1}{t}\int_{0}^{t}f(X_t)dt  \geq u)$. 
\begin{theorem}[Ergodic theorem]
\label{th: ergodic_theorem}
Let $(\F{X}, (\F{P}_x)_{x \in E})$ be an irreducible MJP with unique invariant distribution $\pi$ and let $f \in \mathcal{B}(E)$. Then for any probability measure $\nu$ on $E$
\begin{equation}
    \frac{1}{t}\int_0^t f(X_s) ds \xrightarrow[]{t \to \infty} \pi(f) 
\end{equation}
$\F{P}_\nu $ almost surely, where $\pi(f) = \int_E f d\pi$.
\end{theorem}
\begin{proof}
 Follows from \cite[Ch.\,8.6.1,Thrm.\,6.2]{Bremaud}. The requirement of \cite[Ch.\,8.6.1,Thrm.\,6.2]{Bremaud} that $\F{X}$ is ergodic is satisfied (see \cite[Ch.\,8.5.1,Def.\,5.4]{Bremaud} for a definition of ergodicity) because of \cite[Ch.\,8.5.1,Thrm.\,5.3]{Bremaud} and the alternative characterization \eqref{eq: def_invariant_measure_infinitesimal} of invariance of $\pi$.
\end{proof}
\begin{remark}
\begin{enumerate}[\lb $a$\rb]
    \item Note that the ergodic theorem is a asymptotic result which does not state anything about the rate of convergence or the deviation probability $\F{P} \left( \abs{\frac{1}{t}\int_0^t f(X_s) ds - \pi(f) } \geq u \right)$.
    \item There are more general versions of the above theorem considered in the context of so called ergodic theory. In the the context of ergodic theory a general probability space $(\mathcal{X}, \mathcal{A}, \F{P})$ with a product measurable map $\Theta : [0,\infty) \cross \Omega \rightarrow \Omega $, satisfying $\Theta_0 = \text{Id}_\Omega$ and $\Theta_t(\omega)\Theta_s(\omega) = \Theta_{s+t}(\omega)$ is considered (see \cite[Ch.\,2.2]{eberle} for details). Then for any $F \in L^p(\Omega, \mathcal{F}, \F{P})$ (c.f. \cite[Sec.\,2.2, Thrm.\,2.9]{eberle})
    \begin{equation}
        \frac{1}{t}\int_0^t F \circ \Theta_s ds \xrightarrow[]{t \to \infty} \F{E}(F | \mathcal{I}),
    \end{equation}
$\F{P}$ - almost surely and in $L^p$, where $\mathcal{I} = \{ A \in \mathcal{A} | \Theta_t^{-1}(A) = A \text{ for all $t \geq 0$} \}$ is the $\sigma$-algebra of invariant sets. \\ In our context $\mathcal{X}$ can be chosen to be the space of right continuous functions $(x_t)_{t \geq 0} : \F{R}_{\geq 0} \rightarrow E$, $\mathcal{A}$ is the $\sigma$-Algebra generated by the evaluation maps $(x_t)_{t \geq 0} \mapsto x_s $ for $s \geq 0$, $\F{P}$ is defined by $\F{P}(A) = \F{P}_\pi (\F{X} \in A)$ (where $(\F{P}_\pi, \F{X})$ is an MJP with invariant distribution $\pi$) and $F((x_t)_{t \geq 0}) = f(x_0)$.
\end{enumerate}
 
\end{remark}

\begin{theorem}[Central Limit Theorem]
\label{th: CLT}
Let $\F{X}$ be an irreducible MJP with stationary initial distribution $\pi$ and $f \in L^2(\pi)$ with $\pi(f) = 0$. Then
\begin{equation}
    \frac{1}{\sqrt{t}} \int_0^t f(X_s) ds  \xrightarrow[]{\mathcal{D}} \mathcal{N}(0, \sigma^2),
\end{equation}
where 
\begin{equation}
    \sigma^2 = - 2\langle g, Lg \rangle = - 2 \langle g, f \rangle ,
\end{equation}
for (all) $g \in L^2(\pi)$ with $Lg = f$. Furthermore, if $\F{X}$ is reversible then $\sigma^2$ is called asymptotic variance and
\begin{equation}
\label{eq: asymptotic variance}
   \sigma^2  = -2 \langle Sf, f \rangle  = \lim_{t \to \infty}t^{-1}\mathrm{Var}_\pi\left ( \int_0^tf(X_s) ds \right ),
\end{equation}
where $S$ is the reduced resolvent of $L$ with respect to the eigenvalue $0$.
\end{theorem}
\begin{proof}
Note that as $0 = \pi(f) = \langle f, \textbf{1} \rangle$ by Lemma \ref{lem: properties_infinitesimal_generator} $f \in \text{Im}(L)$. So the first statement follows by
 \cite[P.\,164-167]{summer_school}. The second statement follows by applying the first statement with $g = Sf$ and \cite[P.\,167]{summer_school}
 
\end{proof}

\begin{remark}
The reduced resolvent is defined in Section \ref{subsubsec: perturbation_theory}
\end{remark}

%% file: Probability.tex
\subsection{Additional Tools}

\label{subsec: additional_tools}
In this short Section we provide additional notions and tools used later to derive concentration inequalities. In Section \ref{subsubsec: fenchel_conjugate} we present the concept of \textit{Fenchel conjugate}; an important concept which will be used throughout the main part  of this work. Hereby, we also consider an important example (Example \ref{ex: Legendre_transform_subgamma}) which is relevant later for Bernstein-type concentration inequalities (see Lemma \ref{lem: bernstein_inequality}). In Section \ref{subsubsec: perturbation_theory} we present some formulas and identities of linear algebra and perturbation theory, that will be needed later in the context of the vector space $L^2(\pi)$ ($\pi $ is the invariant distribution of an irreducible MJP). Here, we also define the concept of \textit{reduced resolvent} and present a result about the perturbation of a simple eigenvalue.

\subsubsection{Fenchel Conjugate}
\label{subsubsec: fenchel_conjugate}
Later, when deriving concentration inequalities with the Cramér-Chernoff method, we shall need the notion of the so called Fenchel conjugate. This section is based on \cite{convex_analysis}. In the following we always use the convention $r \pm \infty = \pm \infty$ for all $r \in \F{R}$. We define (c.f. \cite[Ch.\,3.3,P.\,49]{convex_analysis})

\begin{definition}(Fenchel conjugate)
\label{def: fenchel_conjugate}
Let $D \subset \F{R}$ and $F : D \rightarrow [-\infty, \infty]$ be a function. Define $F^* : \F{R} \rightarrow [-\infty, \infty]$ as
\begin{equation}
    F^*(u) := \sup_{r \in D}(ru - F(r))
\end{equation}
We call $F^*$ the Fenchel conjugate of $F$ (with respect to $D$).
\end{definition}
Furthermore, as we will work with functions that may possibly attain the value $\infty$, we clarify the notion of convexity. We define  (c.f. \cite[Ch.\,3.1, P.\,33]{convex_analysis})
\begin{definition}(Convexity)
\label{def: convexity}
We say a function $F : D \rightarrow (-\infty, \infty]$ is convex if 
\begin{equation}
\label{eq: def_convexity}
    F(s_1 r_1 + s_2 r_2) \leq s_1 F(r_1) + s_2 F(r_2)
\end{equation}
for all $r_1, r_2  \in \{ F < \infty \}$ and all $s_1, s_2 \geq 0$ such that $s_1 + s_2 = 1$, where we define $F(r) := \infty $ for $r \not \in D$.

\end{definition}
\begin{remark}
\begin{enumerate}[\lb $a$\rb]
    \item If $F$ is convex, then $F^*$ is also called the \textit{Legendre transform}.
    \item In \cite[Ch.\,3.1, P.\,33]{convex_analysis} the Fenchel conjugate is defined for functions $F : \F{R} \rightarrow \F{R}$, i.e. $D = \F{R}$. However, for our purposes we want to take Fenchel conjugates with respect to subsets $D \subset \F{R}$ (e.g. in Example \ref{ex: Legendre_transform_subgamma}) and thus we use the above Definition \ref{def: fenchel_conjugate}.
\end{enumerate}
\end{remark}
\begin{lemma}[Properties of the Fenchel conjugate]
\label{lem: properties_Legendre_transform}
Let $D,D' \subset \F{R} $ be some subsets with $D' \subset D$ and let  $F: D \rightarrow (-\infty, \infty]$, $G : D' \rightarrow (-\infty, \infty]$ be some functions. Then, the following statements hold

\begin{enumerate}[\lb a\rb]
    \item $F^*$ is convex.
    \item Suppose that $D = [a,b]$ for some $a <  b$, $F$ is continuous, and continuously differentiable on $(a,b)$ with strictly increasing derivative $F'$. Then,
    \begin{equation}
        F^*(u) = u \cdot (F')^{-1}(u) - F((F')^{-1}(u))
    \end{equation}
    for all $u \in F'((a,b))$.
     \item \lb Fenchel Biconjugation Theorem\rb Suppose that $D$ is a closed set, and $F$ is convex and lower semicontinuous on $D$. Hereby, lower semicontinuity means that  
     \begin{equation}
         \liminf_{n \to \infty } F(r_n) \geq F(r)
     \end{equation}
     for all $r \in D$ and all sequences $(r_n)_{n\in \F{N}}$ converging to $r$.
     Then,
    \begin{equation}
        F^{**} = F\;,
    \end{equation}
    where the convention $F(r) = \infty $ if $r \not \in D$ is used.
    \item If $F \leq G$ on $D'$, then $ F^* \geq G^*$
\end{enumerate}

\end{lemma}
\begin{remark}
In this work we will only make direct use of part $(c)$ and $(d)$, part $(a)$ and $(b)$ are included for completeness.
\end{remark}
\begin{proof}[Proof of Lemma \ref{lem: properties_Legendre_transform}]

    Statements $(a),(d)$ are standard results found in \cite[Ch.\,3.3,P.\,49]{convex_analysis}  and  $(b)$ follows from an elementary calculation , so we just show $(c)$. Let $F$ be  convex and lower semicontinuous on $D$.
     Extend $F$ to $F : \F{R} \rightarrow [- \infty, \infty]$ by defining $F(r) = \infty$ whenever $r \not \in D$. Note that the Fenchel conjugate is remains unchanged by this extension, because  
     \begin{equation}
         \sup_{ r \in D}(ru - F(r)) = \sup_{ r \in \F{R}}(ru - F(r))
     \end{equation}
     by the convention $ru - \infty = -\infty$.
     Furthermore it is straightforward to check that  $F : \F{R} \to (-\infty, \infty ]$ is  lower semicontinous (follows from the closedness of $D$) and convex on $\F{R}$ (follows from Definition \ref{def: convexity}). Consequently, by the Fenchel-Biconjugation theorem \cite[Ch.\,4.2, Thrm.\,4.2.1]{convex_analysis} 
    \begin{equation*}
        F = F^{**}.
    \end{equation*}
    Finally, it should be remarked that although \cite[Ch.4.2, Thrm.4.2.1]{convex_analysis} requires that $F$ is closed (see \cite[Ch.\,4.2,P.\,76]{convex_analysis} for a definition), this requirement is equivalent to lower semi-continuity of $F$ (see \cite[Ch.\,4.2,P.\,76]{convex_analysis}).

\end{proof}
The following example of a Fenchel conjugate will become important later, when we treat sub-gamma random variables (see Definition \ref{def: sub-gamma}).
\begin{example} Let $v,c > 0$ and define
\label{ex: Legendre_transform_subgamma}
\begin{equation}
    F(r) := \frac{r^2v}{2(1-rc)}
\end{equation}
for $r \in [0, \frac{1}{c})$. Let $u \geq  0$ ,
\begin{equation}
    r_0 := \frac{1}{c} \left ( 1-  \left ( 1 + \frac{2 u c}{v} \right )^{-\frac{1}{2}} \right)
\end{equation}
and 
\begin{equation}
    H(r) := ru - F(r)
\end{equation}
for $r \in [0, \frac{1}{c})$. Then $H'$ has a unique root in $[0, \frac{1}{c})$, which is given by $r_0$ and  $H'' < 0$ on $[0, \frac{1}{c})$. Consequently
\begin{equation}
    \label{eq: Legendre_transform_subgamma}
    F^*(u) = \sup_{r \in [0, \frac{1}{c})} H(r) =  H(r_0) = \frac{v}{c^2}\left (1 + \frac{uc}{v} - \sqrt{1 + \frac{2uc}{v}} \right)  = \frac{2u^2}{v(1+ \sqrt{1 + \frac{2uc}{v}})^2}
\end{equation}
for all $u \geq 0$.
\end{example}
\begin{proof}
That $H'$ has a unique root on $[0, \frac{1}{c})$, given by $r_0$ and $H'' < 0$ is checked by an elementary calculation. Thus, it follows that $H$ reaches at $r_0$ a unique maximum on $[0,\frac{1}{c})$ , and consequently
\begin{equation}
     \sup_{r \in [0, \frac{1}{c})} H(r) =  H(r_0) = \frac{v}{c^2}\left (1 + \frac{uc}{v} - \sqrt{1 + \frac{2uc}{v}} \right)  = \frac{2u^2}{v(1+ \sqrt{1 + \frac{2uc}{v}})^2} ,
\end{equation}
where the last two equalities may also be checked by elementary calculations (c.f. \cite[P.\,28]{concentration2013}).
\end{proof}

%% file: Pertur.tex
\subsubsection{Linear Algebra and Perturbation Theory}
\label{subsubsec: perturbation_theory}
In the following (unless otherwise stated) let $V$ denote a vector space and $H$ a   Hilbert space. All vector spaces are assumed to be real and finite dimensional. Denote by $\langle \cdot , \cdot \rangle$ the inner product on $H$ and by $\norm{\cdot}$ the induced norm (for vectors and operators). We call a linear map $T: V \rightarrow V$ an operator on V. This section is mainly based on Appendix \ref{subsec: app_perturbation_theory}, which considers a setting of complex vector spaces. However, we consider just real vector spaces as in our context later we will just work with the real vector space  $V = H = L^2(\pi)$ ($\pi$ is the invariant measure of an irreducible MJP). Furthermore, as an extensive rigorous presentation and proof of some of the corresponding results (in the case of complex vector spaces) is long and involves technicalities which are not important for our purposes, the more detailed and general presentation is found in Appendix \ref{subsec: app_perturbation_theory}. \\
\begin{lemma}
\label{lem: biggest_eigenvalue}
Let $T : H \rightarrow H$ be a selfadjoint operator. Then, the largest eigenvalue $\lambda_0$ is given by
\begin{equation*}
    \lambda_0  = \max  \left \{ \frac{\langle Tv, v \rangle }{\norm{v}^2} \middle | v \in H, v \neq 0 \right \}.
\end{equation*}
\end{lemma}
\begin{proof}
It is a well known fact that any selfadjoint operator on $H$ is orthogonally diagonalizable (see \cite[Ch.\,6.7.2]{Fischer2020}) with real eigenvalues, so let $(e_i)_{i = 1,...n}$ be an orthonormal basis of $H$, diagonalizing $T$ with corresponding eigenvalues $\lambda_i \in \F{R}$. The statement follows directly by using 
\begin{equation}
    \langle Tv, v \rangle = \sum_{i = 1}^n\lambda_i v_i^2,
\end{equation}
where $v = \sum_{i=1}^n v_i e_i$.
\end{proof}

\begin{definition}(Reduced resolvent)
\label{def: reduced_resolvent}
Let $T: V \rightarrow V$ be a diagonalizable operator (with real eigenvalues) and $\sigma(T)$ denote the spectrum. For $\lambda \in \sigma(T)$ let $\text{pr}_\lambda$ denote the eigenprojection onto the eigenspace $\text{Ker}(T-\lambda)$ according to the decomposition $V = \bigoplus_{ \lambda \in \sigma(T)} \text{Ker}( T - \lambda )$. For $\lambda \in \sigma(T)$ define an operator $S_\lambda$ by
 \begin{equation}
 \label{eq: reduced_resolvent}
     S_\lambda v = 
 \begin{cases}
   (T - \lambda)|_{ \mathrm{Im}(1- \text{pr}_\lambda)}^{-1}v  \;; \; v \in \mathrm{Im}(1 - \text{pr}_\lambda) \\
   0 \;;\; v \in \mathrm{Im}(\text{pr}_\lambda)
 \end{cases}
 \end{equation}
 $S_\lambda$ is called the reduced resolvent of $T$ (with respect to $\lambda$).
\end{definition}
\begin{remark}
\label{rem: reduced_resolvent}
Note that $S_\lambda$ is well defined. Indeed, this is easily seen by writing $T = \sum_{\mu \in \sigma(T)} \mu \text{pr}_\mu $, then an elementary calculation shows that Definition \ref{def: reduced_resolvent} is equivalent to 
\begin{equation}
    S_\lambda = \sum_{ \mu \in \sigma(T)\backslash \{\lambda \}} \frac{\text{pr}_\mu}{\mu - \lambda }
\end{equation}
\end{remark}
\begin{remark}
We just consider the reduced resolvent of a diagonalizable operator as we will just work with the reduced resolvent of the selfadjoint, diagonalizable operator $\frac{L+L^*}{2}$ ($L$ is the infinitesimal generator of an (irreducible) MJP). See Appendix \ref{subsec: app_perturbation_theory} for a more extensive presentation of reduced resolvents.
\end{remark}

Perturbation theory is a widely used tool  that enables the calculation of eigenvalues and eigenvectors of a perturbed operator of the form 
\begin{equation*}
    T (r) = T + r T',
\end{equation*}
where $r \in \F{R}$ is some 'small' parameter. More generally, one can consider an analytic operator valued function $T(r) = \sum_{n=0}^\infty r^n T^{(n)}$ and ask how the eigenvalues  of $T(r)$ may be computed in terms of $( T^{(n)} )_n$ and $r$. We now give two results (Lemma \ref{lem: eigenvalues_continuous} and Theorem \ref{th: perturbation_simple_eigenvalue}) about the eigenvalues of $T(r)$. These results follow from the corresponding results  in the case of complex vector spaces (Lemma \ref{lem: app_eigenvalues_continuous} and Theorem \ref{th: app_perturbation_simple_eigenvalue}). The results of Appendix \ref{subsec: app_perturbation_theory} transfer to our setting of real vector spaces essentially by applying these results to the complexification  (c.f. Lemma \ref{lem: app_complexification}) of the vector spaces and operators. 

\begin{lemma}[Continuous dependence of eigenvalues]
 \label{lem: eigenvalues_continuous}
 Let $T(r)$ be a continuous operator-valued function \lb on $V$\rb \space  defined on some interval $I \subset \F{R}$. Furtheremore, let $N = \dim V$. Then, there are continuous functions $\lambda_k : I \rightarrow \F{C}$, $k = 1,...,n$ such that the $N$-tuple 
 \begin{equation*}
    (\lambda_1(r), ... , \lambda_N(r)).
 \end{equation*}
 represents the eigenvalues of $T(r)$, where the eigenvalues are repeated according to their algebraic multiplicity. 
\end{lemma}

\begin{proof}
Follows from the corresponding complex version Lemma \ref{lem: app_eigenvalues_continuous}. For a more detailed explanation see Remark \ref{rem: app_eigenvalues_continuous}.
\end{proof}

\begin{theorem}[Perturbation of a simple eigenvalue]
\label{th: perturbation_simple_eigenvalue}
Let $T : H \rightarrow H$ be a self adjoint operator. Assume that $0$ is a simple eigenvalue of $T$. Let $\lambda_1 := \min_{\lambda \in \sigma(T) \backslash \{0\}} \abs{\lambda}$ be the spectral gap, $\mathrm{pr}$ the orthogonal projection onto the eigenspace with eigenvalue $0$ and $S = S_0$ the corresponding reduced resolvent. Furthermore, let $ T' : H \rightarrow H$ be some operator and define $ T(r) := T + r T'$ for $r \in \F{R}$. Then for all $\abs{r} < \frac{\lambda_1}{2 \norm{T}}$ the ball  $B_{\frac{\lambda_1}{2}}(0) \subset \F{C}$  contains exactly one simple eigenvalue $\mu_0(r)$ of $T(r)$ and we have
\begin{equation}
\label{eq: mu_0(r)}
    \mu_0(r) = \sum_{n = 1}^\infty \mu_0^{(n)}r^n,
\end{equation}
with 
\begin{equation}
 \mu_0^{(n)} =  \frac{(-1)^n}{n} \sum\limits_{\substack{k_1,..k_n \in \F{Z}_+ \\ k_1 + ... k_n = n-1}} \Tr( T'S^{(k_1)} ... T'S^{(k_n)}),
\end{equation}
where
\begin{equation}
    S^{(0)} = - \mathrm{pr} \quad \text{and} \quad S^{(k)} = S^k
\end{equation}
for $ k \geq 1$.
\end{theorem}
\begin{proof}
Follows from the corresponding complex analogue \ref{th: app_perturbation_simple_eigenvalue}. For a detailed transfer to the real setting see Remark \ref{rem: app_perturbation_simple_eigenvalue}.
\end{proof}

%% file: Concentration_Inequalities.tex
\section{Concentration Inequalities}
\label{sec: concentration_inequalities}
\subsection{Outline and Goal}
\label{subsec: outline_and_goal}
The ergodic theorem (Theorem \ref{th: ergodic_theorem}) states that for an irreducible  MJP $(X_t)_{t \geq 0}$ with invariant distribution $\pi$ and any initial distribution $\nu$  

\begin{equation}
\label{eq: ergodic_jump}
   \frac{1}{t}\int_{0}^{t}f(X_t)dt \xrightarrow[]{t \rightarrow \infty} \pi(f)  \quad \mathbb{P}_\nu \text{ - a.s.} 
\end{equation}
holds for any $f \in \mathcal{B}(E)$.
The main goal of this section will be to give bounds for the deviation of the (finite) time average in the above equation from the long term limit $\pi(f)$ for irreducible MJPs. In more precise mathematical terms: For any $t >0 $ we want to bound the tail probabilities

\begin{equation}
\label{eq: probability_deviation}
    \mathbb{P}_\nu \left (\frac{1}{t}\int_{0}^{t}f(X_s)ds - \pi(f)  \geq u \right )
\end{equation}
for $u \geq 0$,
and 
\begin{equation}
\label{eq: lower_tail}
   \mathbb{P}_\nu \left (\frac{1}{t}\int_{0}^{t}f(X_s)ds - \pi(f)  \leq  u \right ) 
\end{equation}
for $u \leq 0$, in terms of some function of $u$. Since we consider a general $f \in \mathcal{B}(E)$ and the replacement of $f$ by $-f$ in the upper tail probability \eqref{eq: probability_deviation} yields the lower tail probability  \eqref{eq: lower_tail} we consider without loss of generality just the upper tail probability \eqref{eq: probability_deviation}. To derive concentration inequalities for this probability  we will apply the Cramér-Chernoff method, a method that can be quite generally used to bound upper tail  probabilities
\begin{equation}
    \F{P}( Z \geq u).
\end{equation}
This section consists of two main parts \textemdash Section \ref{subsec: cramer_chernoff_method} and Section \ref{subsec: application_cramer_chernoff}. First, the general Cramér-Chernoff method is introduced in Section \ref{subsec: cramer_chernoff_method}.   Here, the most important results are the Chernoff inequality and the application of it, which is the basis for all concentration inequalities of this thesis, and  Bernstein's inequality, which is an example of Chernoff's inequality. In the second part \textemdash Section \ref{subsec: application_cramer_chernoff}, based on the established Cramér-Chernoff method, we derive concentration inequalities for (irreducible) MJPs. Hereby, we summarize and combine results of the works \cite{wu}, \cite{lezaud}, \cite{guillin}, and  \cite{bernstein}, which  use the Cramér-Chernoff method to derive concentration inequalities for \eqref{eq: probability_deviation}. More precisely, in Section \ref{subsubsec: general_concentration_inequality} we apply directly the Cramér-Chernoff method to the setting of MJPs and obtain a general concentration inequality (see Theorem \ref{th: main_conc_inequality}). Then, in Sections \ref{subsubsec: further_concentration_inequalities_MJP} - \ref{subsubsec: concentration_via_information} based on this general concentration inequality we derive more explicit concentration inequalities using three different approaches:  perturbation theory (Section \ref{subsubsec: concentration_via_perturbation}), functional inequalities (Section \ref{subsubsec: concentration_via_functional_inequalities}) and information inequalities (Section \ref{subsubsec: concentration_via_information}).  \\
Finally, some reading advice: Throughout Section \ref{sec: concentration_inequalities} we will make use of (certain) results, which are presented in Section \ref{sec: preliminaries}. Thus, we pointed out at the beginning of many (sub)sections the most relevant results and notions of Section \ref{sec: preliminaries} (and also Section \ref{sec: concentration_inequalities}) used, so that the reader might have again a look at the mentioned results, before engaging with the material of the subsection. Furthermore, throughout Section \ref{subsec: application_cramer_chernoff} we will explain connections and give references to the works \cite{wu}, \cite{lezaud}, \cite{guillin} and \cite{bernstein}, which are denoted by \textbf{Reference}. These remarks are not essential for the content and can be skipped if the reader is not interested in putting this work into the context of \cite{wu}, \cite{lezaud}, \cite{guillin}, \cite{bernstein}.
\subsection{Cramér-Chernoff Method}
\label{subsec: cramer_chernoff_method}
The Cramér-Chernoff method is a very general way of bounding tail probabilities of the form $ \F{P}(Z \geq u) $. The bounds represent 'exponential' decay inequalities and are usually sharper than  polynomial decay inequalities, like the Chebychev inequality. Furthermore, the  Cramér-Chernoff method is the starting point to derive many concentration inequalities (c.f. \cite[Ch.\,2]{concentration2013}) e.g. Hoeffding's inequality, Bernstein's inequality, and Benett's inequality, which are applicable in large deviation theory (c.f. \cite[Ch.\,2]{Dembo2009}), learning theory  (c.f. \cite[Ch.\,2]{learning_theory}), randomized algorithms (c.f. \cite[Ch.\,4]{ramdomized_algorithms}) to name but a few, and in this work the Cramér-Chernoff method is the basis for all concentration inequalities derived. \\
In this section we first derive the Cramér-Chernoff method by motivating it via the optimization of the Markov inequality and we arrive at a general concentration inequality; Chernoff's inequality. Then, we explain how the method may be used to derive more explicit concentration inequalities and consider an important example; Bernstein's inequality. Finally, we also briefly discuss the application of the method to sums of i.i.d. random variables. The most relevant notion of Section \ref{sec: preliminaries} is the \textit{Fenchel conjugate} (defined in Section \ref{subsubsec: fenchel_conjugate}).
\subsubsection{Motivation}
\label{subsubsec: motivation}
In the following  assume that all random variables in this section are defined on an underlying probability space $(\Omega, \mathcal{F}, \F{P})$. Recall two well known basic inequalities: 
\begin{theorem}[Markov Inequality]
\label{th: markov_inequality}
Let $H : D \rightarrow [0,\infty)$ be a nondecreasing nonnegative function, where $D \subset \F{R} $ is some measurable subset. Let $Z$ be a $D$-valued random variable. Then for all $u \in D$ such that $H(u) > 0$
\begin{equation}
\label{eq: markov_inequality}
    \F{P}(Z \geq u) \leq \frac{\F{E}(H(Z))}{H(u)}.
\end{equation}
\end{theorem}
\begin{proof}
As $H$ is nondecreasing and nonnegative we have 
\begin{equation*}
    H(u)1_{\{ Z \geq u\}} \leq H(Z) 1_{\{ Z \geq u\}} \leq H(Z).
\end{equation*}
Taking the expectation on both sides and then dividing by $H(u)$ yields the statement.
\end{proof}

\begin{theorem}[Jensen's Inequality]
\label{th: jensen_inequality}
Let $I \subset \F{R}$ be some interval and $Z$ an $I$-valued random variable with $\F{E}(\abs{Z}) < \infty$. Furthermore let $H: I \rightarrow \F{R}$ be some convex function. Then $\F{E}(H(Z)^{-}) < \infty$ and 
\begin{equation*}
    \F{E}(H(Z)) \geq H(\F{E}(Z)),
\end{equation*}
where $H(Z)^{-} = \max(0, - H(Z))$ denotes the nonpositive part of $H(Z)$.

\end{theorem}
\begin{proof}
See \cite[Ch.\,7.2, Thrm.\,7.9]{Klenke2013}.
\end{proof}

Although the Markov inequality is quite simple, we can use it to obtain sharper bounds on the tail probability $\F{P}(Z \geq u)$ as follows. Let $(H_r)_{r \in I} $ be a family of nondecreasing nonnegative functions (defined on a common domain), parametrized by some parameter $r \in I$, where $I$ is some index set. Suppose that $H_r(u) > 0$ for any $r $, so that the Markov inequality implies
\begin{equation}
    \F{P}(Z \geq u) \leq \frac{\F{E}(H_r(Z))}{H_r(u)}
\end{equation}

for all $H_r $. Then, we can optimize  the bound in the Markov inequality  by minimizing over $r \in I$, i.e.
\begin{equation}
\label{eq: Markov_minimum_tail}
    \F{P}(Z \geq u) \leq \inf\limits_{r \in I} \frac{\F{E}(H_r(Z))}{H_r(u)}.
\end{equation}
The Cramér-Chernoff method consists of applying this principle using the family $H_r(u) =e^{r u}$ for $r \geq 0 $. This family  has a useful property: Notice that for all $u_1,...,u_n \in \F{R}$ we have $H_r(u_1 + ... + u_n) = H_r(u_1) H_r(u_2)...H_r(u_n)$, so if $Z_1,...Z_n$ are i.i.d.  then the bound in \eqref{eq: Markov_minimum_tail} for $Z = Z_1 + ... + Z_n$ factorizes via $\F{E}(H_r(Z) ) = \F{E}(H_r(Z_1))^n $.

\subsubsection{General Cramér-Chernoff Method}

\label{subsubsec: general_cramer_chernoff_method}
The following presentation of the Cramér-Chernoff method is based on \cite[Ch.\,2.1-2.2]{concentration2013}.
The Cramér-Chernoff method gives the optimal bound for the tail probability $\F{P}(Z \geq u)$ that can be obtained by minimizing the Markov inequality over the family $H_r(u) =e^{r u}$ for $r \geq 0$ as in  \eqref{eq: Markov_minimum_tail}.  In the following we use the conventions $\log(\infty) = \infty$, $\exp(-\infty) = 0$, $\exp(\infty) = \infty $ , $r \cdot \infty = \infty $ for any $r \in \F{R}\backslash \{0\}$, $r \pm \infty = \pm \infty $ for any $r \in \F{R}$ and $\infty^n = \infty$ for any $n \in \F{N}$. This avoids having to make a case distinction whether $\F{E}(e^{rZ}) = \infty$. The Cramér-Chernoff method is summarized in the following theorem:
\begin{theorem}[Cramér-Chernoff method and Chernoff inequality]
\label{th: cramer_chernoff}
Let Z be a  random variable and  $u \in \F{R}$. Furthermore, let $\Psi_Z(r) := \log \F{E}(e^{r Z})$ for $r \in \F{R}$ be the cumulant generating function and  $\Psi_Z^*(u) := \sup \limits_{r \geq 0}(r u - \Psi_Z(r))$ the Fenchel conjugate with respect to $\F{R}_{\geq 0}$. Then, Chernoff's inequality 
\begin{equation}
    \label{eq: Chernoff_inequality}
    \F{P}(Z \geq u) \leq e^{-\Psi_Z^*(u)}
\end{equation}
holds. If furthermore $\F{E}(Z)$ exists and $ u \geq \F{E}(Z)$, then
\begin{equation*}
    \Psi_Z^*(u) = \sup \limits_{r \in \F{R}}(u r - \Psi_Z(r)).
\end{equation*}
\end{theorem}
\begin{remark}
\begin{enumerate}[$(a)$]
    \item Note that we always have $\Psi_Z^*(u) \geq 0 \cdot u - \Psi_Z(0) = 0$, so the bound in Chernoff's inequality is always $\leq 1$.
    \item  Chernoff's inequality is trivial whenever $\Psi_Z^*(u) = 0$. If $\F{E}(Z) $ exists this happens for $u \leq \F{E}(Z)$, because then by Jensen's inequality (Theorem \ref{th: jensen_inequality}) we have that $r \F{E}(Z) \leq \Psi_Z(r)$  and consequently $ur - \Psi_Z(r) \leq 0$ for any $r \geq 0$. Furthermore, if $\F{E}(e^{r Z}) = \infty$ for all $ r > 0$, then also $\Psi_Z^*(u) = 0$.
\end{enumerate}
\label{rem: Chernoff_inequality}
\end{remark}
\begin{proof}[Proof of Theorem \ref{th: cramer_chernoff}]
Chernoff's inequality can be proved via Markov's inequality as already described in Section \ref{subsubsec: motivation}.
Applying Markov's inequality  with   $u \mapsto e^{ur}$ yields 
\begin{equation*}
     \F{P}(Z \geq u) \leq e^{-r u} \F{E}(e^{r Z}) = e^{-(ru - \Psi_Z(r))}
\end{equation*}
for all $r \geq 0$. Taking the infimum yields 
\begin{equation*}
 \F{P}(Z \geq u) \leq  \inf_{r \geq 0} e^{- (ru - \Psi_Z(r))} = e^{\inf_{r \geq 0}- (ru - \Psi_Z(r))} = e^{- \sup_{r \geq 0}(ru - \Psi_Z(r))} = e^{-\Psi_Z^*(u)},
\end{equation*}
where in the first equality the infimum can be pulled into the exponential function, as the  exponential function is nondecreasing and continuous.
 If $\F{E}(Z)$ exists, then by the convexity of the exponential function and Jensen's inequality, for all $r \in \F{R}$ it holds that $e^{r \F{E}(Z)} \leq \F{E}(e^{rZ})$ and thus $ r \F{E}(Z)  \leq \Psi_Z(r)$. But then for any $r \leq 0$ and $u \geq \F{E}(Z)$ it holds that
\begin{equation*}
    r u - \Psi_Z(r) \leq r \F{E}(Z) - \Psi_Z(r) \leq 0.
\end{equation*}
But $\Psi_Z^*(u) \geq 0\cdot u - \Psi_Z(0) = 0$, so $\Psi_Z^*(u) = \sup \limits_{r \in \F{R}}( r t - \Psi_Z(r))$ for any $ u \geq \F{E}(Z)$.
\end{proof}
The Fenchel conjugate $\Psi_Z^*$ is called the \textit{Cramér transform} of $Z$.
In other words, the Cramér-Chernoff method consists of computing the Cramér transform $\Psi_Z^*$ and then applying Chernoff's inequality.
In general, the cumulant generating function $\Psi_Z$ is not directly computable (as in the exact distribution of $Z$ is not known or too complicated) and thus the Cramér transform $\Psi_Z^*$ cannot be computed explicitly. However, the following lemma describes a method on how more explicit concentration inequalities can be obtained. We have
\begin{lemma}
\label{lem: more_explicit_concentration_inequality}
 Let $D \subset \F{R}_{\geq 0}$ and $\Phi : D \rightarrow \F{R}$ some function such that 
 \begin{equation}
     \Psi_Z(r)\leq \Phi(r)
 \end{equation}
for all $r \in D$, and denote by $\Phi^*$ its Fenchel conjugate \lb with respect to $D$\rb. Then, $\Psi_Z^* \geq \Phi^*$ and consequently 
\begin{equation}
\label{eq: concentration_inequality_with_Phi}
    \F{P}(Z \geq u) \leq e^{-\Psi_Z^*(u)} \leq e^{-\Phi^*(u)}.
\end{equation}
for all $u \in \F{R}$
\end{lemma}
\begin{proof}
The bound $\Psi_Z \leq \Phi$ (on $D$) implies by Lemma \ref{lem: properties_Legendre_transform}$(d)$ $\Psi_Z^*(u) \geq \Phi^*(u)$ for all $u \in \F{R}$. Thus, \eqref{eq: concentration_inequality_with_Phi} follows directly by Chernoff's inequality
\end{proof}
For applications $\Phi$ must be such that $\Phi^*$ can be computed. The following example of a bound $\Phi$ will be relevant for our later discussion. We define (c.f.\cite[Ch.\,2.4]{concentration2013}) :
\begin{definition}
\label{def: sub-gamma}
We call a random variable $Z$ sub-gamma (on the right tail) with variance factor $v$ and scale parameter $c$ if
\begin{equation}
    \Psi_Z(r) \leq \frac{r^2v}{2(1-cr)}
\end{equation}
for all $0 \leq r < \frac{1}{c}$.
\end{definition}
\begin{example} We will show later that for an irreducible MJP the integral \\ $\int_0^t (f(X_s) - \pi(f)) ds$ is sub-gamma (on the right tail) with respect to stationary initial conditions, i.e. with respect to $\F{P}_\pi$. For details see Remarks \ref{rem: lemma_bound_lammba_(r)_via_perturbation_theory}$(b)$ and \ref{rem: poincare_sharp_bound_lambda}$(a)$.
\end{example}
As an immediate consequence of Lemma \ref{lem: more_explicit_concentration_inequality}  we obtain for sub-gamma random variables the concentration inequality, which will appear later again (e.g. in Theorems \ref{th: concentration_via_perturbation}, \ref{th: concentration_via_poincare}).

\begin{lemma}[Bernstein inequality]
\label{lem: bernstein_inequality}
Suppose $Z$ is sub-gamma on the right tail with variance factor $v$ and scale parameter $c$. Then,
\begin{equation}
    \label{eq: Bernstein_inequality}
    \F{P}(Z \geq u) \leq \exp(-\frac{v}{c^2}\left (1 + \frac{uc}{v} - \sqrt{1 + \frac{2uc}{v}} \right))  = \exp(-\frac{2u^2}{v(1+ \sqrt{1 + \frac{2uc}{v}})^2})
\end{equation}
for all $ u \geq 0$.
\end{lemma}
\begin{proof}
This follows immediately by Lemma \ref{lem: more_explicit_concentration_inequality} with $\Phi(r) = \frac{r^2v}{2(1-cr)}$ and Example \ref{ex: Legendre_transform_subgamma}.
\end{proof}

\begin{reference_paper}
It should be remarked that although the works considered here (i.e.  \cite{wu}, \cite{guillin}, \cite{bernstein}, \cite{lezaud}) do not all explicitly mention or state the Cramér-Chernoff method, all these works are based on the Cramér-Chernoff method as follows. To prove \cite[Thrm.\,1]{wu} (concentration inequality in \cite{wu}) Wu \cite{wu}  refers to the Cramér-Theorem (see \cite[P.\,438]{wu}) and then uses it in \cite[Eq.\,(10)]{wu}. Furthermore, the works \cite{guillin} and \cite{bernstein} are based on \cite[Thrm.\,1]{wu} (which is based on the Cramér-Chernoff method) and refer to it (see \cite[Thrm.\,1.1]{guillin} and \cite[Thrm.\,2.1]{bernstein}). Finally, Lezaud \cite{lezaud} does not directly refer to \cite{wu} but to prove his concentration inequality \cite[Thrm.\,2.4]{lezaud} he proceeds as in the proof of the Cramér-Chernoff method (Theorem \ref{th: cramer_chernoff}) as follows. His proof is based on  the Markov inequality
\begin{equation}
    \F{P}_\nu \left( \frac{\int_0^t f(X_s)ds }{t} - \pi(f)  \geq u \right)  \leq e^{-tu r} \F{E}_\nu \left( e^{r(t^{-1}\int_0^t f(X_s)ds - \pi(f))} \right) 
\end{equation}
which is a reformulation of \cite[Lemma\,2.1]{lezaud} (using \eqref{eq: E_nu_rewriting}). Then, he derives further upper bounds of the above Markov inequality (see \cite[Lemma\,2.2,Lemma\,2.3]{lezaud}) before finally computing a Fenchel conjugate by maximizing (see \cite[P.\,190]{lezaud}) a function of the form
\begin{equation}
   r \mapsto ru - F(r).
\end{equation}
Thus, our presentation here, which emphasizes the Cramér-Chernoff method, is a reformulation of the proofs and results contained in these works, which puts these results on a common footing.
\end{reference_paper}

\subsubsection{Cramér-Chernoff Method for Sums of independent Variables}
As already mentioned in Section 3.3.1, for $H_r(u) = e^{r u}$ and $Z_1,...,Z_n$ i.i.d.,  we have that $\F{E}(H_r(Z_1+...+Z_n)) = \F{E}(H_r(Z_1))^n $, so Chernoff's inequality generalizes quite easily to sums of i.i.d. random variables. More precisely we have the following corollary.
\begin{corollary}[Chernoff's inequality for sums of i.i.d. random variables]
\label{cor: cramer_chernoff_sum_iid}
Let $Z_1,...Z_n$ be i.i.d. random variables and define $ Z := \sum_{i=1}^n Z_i$. Then, the Cramér transform $\Psi_Z^*$ of $Z$ is given by 
\begin{equation*}
    \Psi_Z^*(u) = n\Psi_{Z_1}^*\left( \frac{u}{n} \right),
\end{equation*}
where $u \in \F{R}$ and $\Psi_{Z_1}^*$ is the Cramér transform of $Z_1$. Consequently,
\begin{equation}
    \label{eq: chernoff_for_iid}
    \F{P}( Z \geq u) \leq e^{-n \Psi_{Z_1}^*(\frac{u}{n})}.
\end{equation}

\end{corollary}
\begin{proof}
As $Z_1,...,Z_n$ are i.i.d. we have

\begin{equation*}
    \Psi_Z(r ) = \log ( \F{E}\left(\prod_{i=1}^n  e^{r  Z_i} \right)) = \log (\F{E}(e^{r  Z_1})^n) = n\Psi_{Z_1}(r ),
\end{equation*}
so \begin{equation*}
 \Psi_Z^* (u)=    \sup \limits_{r \geq 0}  ( r u - \Psi_Z(r)) = \sup \limits_{r \geq 0}  n \left (r \frac{u}{n}- \Psi_{Z_1}(r) \right) = n\Psi_{Z_1}^*\left (\frac{u}{n} \right).
\end{equation*}
Finally, the inequality \eqref{eq: chernoff_for_iid} follows by an application of Chernoff's inequality.
\end{proof}

\subsection{Application of the Cramér-Chernoff Method to Functionals of MJPs}
\label{subsec: application_cramer_chernoff}
\subsubsection{Setting and Notation}
\label{subsubsec: setting_notation}
Throughout Section \ref{subsec: application_cramer_chernoff} we consider  the following setting.  Let $E$ be a finite set, and  all functions on $E$ are considered to be real valued. For any probability measure $\mu $ on $E$ and any $f \in \mathcal{B}(E)$  let
\begin{equation}
    \mu(f) := \int f d\mu
\end{equation}
Furthermore, let $(\F{X}, (\F{P}_x)_{x \in E})$ be an irreducible  MJP on  $E$ with unique invariant distribution $\pi$ (according to Theorem \ref{th: uniqueness_existence_invariant_measure}).  Fix some $f \in \mathcal{B}(E) = \F{R}^E$ (c.f. Remark \ref{rem: simplification_finite_state_space}),
some probability measure $\nu = (\nu_x)_{x \in E} $ on $E$, and define $A_t := \int_0^t (f(X_s) - \pi(f))ds$. By centering $f$ if needed, we may also assume 
\begin{equation}
    \label{eq: pi(f) = 0}
    \pi(f) = 0,
\end{equation}
so in particular
\begin{equation}
    \label{eq: A_t}
    A_t = \int_0^t f(X_s) ds.
\end{equation}
To avoid trivialities we assume that $\# E \geq 2$ and $f$ is not constant (i.e. $f \neq 0$). Finally, denote by $\norm{\cdot}_2$ the $L^2(\pi)$-norm or the operator norm induced by the $L^2(\pi)$-norm, and by $\langle \cdot, \cdot \rangle  $ the inner product on $L^2(\pi)$. Furthermore, let $\Psi_{A_t}(r) = \log \F{E}_\nu(e^{rA_t})$.

\subsubsection{A general Concentration Inequality}
\label{subsubsec: general_concentration_inequality}

In this section based on the Cramér-Chernoff method we derive a general concentration inequality; Theorem \ref{th: main_conc_inequality}, which is our version of \cite[Thrm.\,1]{wu}.
The most relevant notions and results of Sections \ref{sec: preliminaries} and \ref{sec: concentration_inequalities} are infinitesimal generator (of an MJP) (Section \ref{subsubsec: infinitesimal_generators}) and Cramér-Chernoff method; in particular Lemma \ref{lem: more_explicit_concentration_inequality} (Section \ref{subsec: cramer_chernoff_method}).  This section is based on \cite{wu}  and \cite{lezaud}, in particular  \cite[Proof of Thrm.\,1]{wu} and \cite[Proof of Lemma\,2.3]{lezaud}.\par 
We proceed as follows. We use the Cramér-Chernoff method, more precisely Lemma \ref{lem: more_explicit_concentration_inequality}, i.e. we want to find a bound $\Phi(r) \geq \Psi_{A_t}(r)$ to obtain  a concentration inequality  
\begin{equation}
    \F{P}_\nu \left (\frac{A_t}{t} \geq u \right) = \F{P}_\nu (A_t \geq tu) \leq e^{-\Phi^*(tu)} .
\end{equation}
To find a bound $\Phi(r)$  we use a so called \textit{Feynman-Kac semigroup} $(P_t^{rf})_{ t \geq 0}$ on $L^2(\pi)$ to rewrite $\Psi_{A_t}(r)$ in terms of this semigroup. Then, by bounding $\norm{P_t^{rf}}_2$ we obtain a  bound $\Phi(r)$ and by computing $\Phi^*(tu)$ a concentration inequality. \\
  Define for $t \geq 0$, $h \in L^2(\pi)$ an operator $P_t^h$ on $L^2(\pi)$ via (c.f. \cite[Eq.\,(7)]{wu})
\begin{equation}
\label{eq: def_P_t^h}
    (P_t^hg)(x) := \F{E}_x \left (e^{\int_0^th(X_s) ds}g(X_t) \right )
\end{equation}
for $g \in L^2(\pi)$ and $x \in E$.  Recall that in our context  $L^2(\pi) = \mathcal{B}(E)$ (see Remark \ref{rem: simplification_finite_state_space}) and $\pi_x > 0$ for all $x \in E$ (see Theorem \ref{th: uniqueness_existence_invariant_measure}), so the above operator is clearly well defined. We will see in Lemma \ref{lem: feynman_kac} that $(P_t^h)_{t \geq 0}$ is indeed a semigroup. This semigroup is usually referred to as a \textit{Feynman-Kac semigroup} (see e.g. \cite[Eq.\,(7)]{wu}, \cite[S.\,13]{guillin}). Let  $\textbf{1} : E \rightarrow \{ 1 \}$ denote the constant 1-function. To rewrite $\Psi_{A_t}(r) = \log \F{E}_\nu (e^{rA_t})$ we can rewrite (by using directly the Definition \eqref{eq: def_P_t^h})
\begin{equation}
\label{eq: E_nu_rewriting}
    \F{E}_\nu(e^{rA_t}) =\sum_{x \in E}\nu_x (P_t^{rf}\textbf{1})(x)  =  \sum_{x \in E}\frac{\nu_x}{\pi_x} (P_t^{rf}\textbf{1})(x) \pi_x = \left \langle \frac{d \nu}{d \pi} ,P_t^{rf}\textbf{1} \right  \rangle ,
\end{equation}
where $\frac{d \nu}{d \pi}(x) = \frac{\nu_x}{\pi_x}$ denotes the Radon-Nikodym derivative of $\nu$ with respect to $\pi$ (note that $\frac{d \nu}{d \pi}$ is well defined, because by Theorem \ref{th: uniqueness_existence_invariant_measure} $\pi_x > 0$ for all $x \in E$).
Applying the Cauchy-Schwartz inequality to the above equality yields the bound
\begin{equation}
\label{eq: bound_E_nu}
    \F{E}_\nu(e^{rA_t}) \leq \norm{\frac{d\nu}{d\pi}}_2 \norm{P_t^{rf}}_2
\end{equation}
and consequently
\begin{equation}
\label{eq: bound_Psi_Z_prefinal}
    \Psi_{A_t}(r) \leq \log \norm{\frac{d\nu}{d\pi}}_2 + \log \norm{P_t^{rf}}_2.
\end{equation}
We will now derive the announced bound
\begin{equation*}
    \Phi(r) \geq \log \norm{\frac{d\nu}{d\pi}}_2 + \log \norm{P_t^{rf}}_2 \geq \Psi_{A_t}(r)
\end{equation*}
 by bounding  $\norm{P_t^{rf}}_2$. For that we will prove and use that $(P_t^h)_{ t \geq 0}$ is a continuous semigroup on $L^2(\pi)$ with infinitesimal generator $L + M_h$, where $L$ is the infinitesimal generator of the MJP and $M_h$ is the multiplication with $h$, i.e. $(M_hg)(x) = h(x)g(x)$  for $x \in E$, $g \in L^2(\pi)$. We have the following lemma:
\begin{lemma}[Feynman-Kac semigroup]
\label{lem: feynman_kac}
Let $(P_t^h)_{t \geq 0} $ be the family of operators defined as in \eqref{eq: def_P_t^h} and define $L_h = L + M_h$. Then $(P_t^h)_{t \geq 0} $ is a continuous semigroup on $L^2(\pi)$ with generator $L_h$, i.e.
\begin{equation}
    \label{eq: semigroup of tilted generator}
    P_t^h = \exp(tL_h).
\end{equation}
\end{lemma}
\begin{reference_paper}
\begin{enumerate}[$(a)$]
    \item  We decided to include here our own proof of Lemma \ref{lem: feynman_kac}, as \cite{wu} uses  this Lemma (see \cite[P.\,439, Case\,1]{wu}) but does not provide a proof.
    \item The above lemma can also be proved by using the proof idea of \cite[Lemma 2.1]{lezaud}: Using the Markov property it can be shown that  for any $g \in L^2(\pi)$, $k \in \F{N}$, $x \in E$ and any $t \geq 0$
    \begin{equation}
      \left( \left[\exp(\frac{t}{k}L)\exp(\frac{t}{k}M_h)\right]^k g \right)(x) = \F{E}_x \left ( e^{\sum_{i=1}^k h\left(X_{j\cdot \frac{t}{k}}\right) \frac{t}{k}} g(X_t)\right).
    \end{equation}
    Then, a combination of the Trotter-product formula (\cite[Theorem 2.11]{Hall2015})
    \begin{equation}
        \lim_{k \to \infty} \left[\exp(\frac{t}{k}L)\exp(\frac{t}{k}M_h)\right]^k = \exp(t(L+M_h))
    \end{equation}
    and the dominated convergence theorem yields Lemma \ref{lem: feynman_kac}
\end{enumerate}
\end{reference_paper}
\begin{remark}
\begin{enumerate}[$(a)$]
    \item In the context of the Feynman-Kac semigroup $(P_t^h)_{ t \geq 0}$ the operator $L_h = L + M_h$ is also frequently called  'tilted generator' (see e.g. \cite{Lapolla})
    \item As $M_{rf} = r M_f$, the generator of $(P_t^{rf})_{t \geq 0}$ is $L + rM_f$
    \item (Feynman-Kac formula) The above lemma implies the so called \textit{Feynman-Kac formula} (c.f \cite[P.\,118]{Liggett2010}), which states that the function 
    \begin{equation*}
        u(t,x) := \F{E}_x \left (e^{\int_0^th(X_s) ds}g(X_t) \right ) = (P_t^hg)(x)
    \end{equation*}
    is a probabilistic solution of the differential equation
    \begin{equation*}
        \frac{d}{dt}u(t,x) = Lu(t,x) + h(x)u(x), \, u(0,x) = g(x)
    \end{equation*}
     
\end{enumerate}

\end{remark}

\begin{proof}[Proof of Lemma \ref{lem: feynman_kac} ]
 As we are treating an MJP on a finite state space, $L^2(\pi)$ is finite dimensional and we can prove the claim by direct computation as follows.
In any finite dimensional vector space any two semigroups having the  same infinitesimal generator are identical (see Remark \ref{rem: semigroup_generator_finite_dim}). So to prove that $P_t^h = \exp(tL_h)$, we prove that $(P_t^h)_{t \geq 0}$ is a semigroup with generator $L_h$. Let $g \in L^2(\pi) $, $I_t = \int_0^th(X_s)ds$ and define $B_t : E^{[0,\infty)} \rightarrow \F{R}$ by 
\begin{equation}
\label{eq: def_A_t}
B_t((y_s)_{s \geq 0}) = \exp(\limsup_{k \to \infty }\sum_{i = 1}^k\frac{t}{k}h\left( y_{\frac{t i}{k}}\right))g(y_t).
\end{equation}
Note that we have $B_t((X_s)_{s \geq 0}) = e^{I_t}g(X_t)$ for all $t \geq 0$, because the Riemann sums in \eqref{eq: def_A_t} converge to $I_t$ by the right continuity of the paths of $\F{X}$ and boundedness of $h$.
Thus, for all $x \in E$ and $s,t \geq 0$ we have
\begin{align*}
  &(P_{t}^hP_{s}^hg)(x) = \F{E}_x \left [e^{I_t}(P_s^hg)(X_t) \right] = \F{E}_x \left [e^{I_t}\F{E}_{X_t}[e^{I_s}g(X_s)]\right] = \F{E}_x \left[ e^{I_t} \F{E}_{X_t} [ B_s((X_h)_{h \geq 0})] \right] \\
  &= \F{E}_x \left [e^{I_t}\F{E}_x[ B_s((X_{h+t})_{h \geq 0}) | (X_h)_{0 \leq h \leq t}] \right ] = \F{E}_x \left [\F{E}_x[ e^{I_t}B_s((X_{h+t})_{h \geq 0}) | (X_h)_{0 \leq h \leq t}] \right ] \\
  &= \F{E}_x\left [ e^{I_t}B_s((X_{h+t})_{h \geq 0})\right ] = \F{E}_x \left[e^{\int_0^th(X_u)du+\int_t^{t+s}h(X_u)du}g(X_{s+t})\right] = (P^h_{s+t}g)(x),
\end{align*}
where in the second line the Markov property \eqref{eq: Markov_property} of the MJP and the pullout-property of the conditional expectation were used.
Furthermore, as $\F{P}_x(X_0 = x) = 1$ and $S_0 \equiv 0$ we have $P_0^hg = g$, so $(P_t^h)_{t \geq 0}$ is a semigroup of operators. We now calculate the generator of $(P_t^h)_{t \geq 0}$. Let $x \in E$ and $ t > 0$, then
\begin{equation}
    \label{eq: lemma_tilted_generator}
    \frac{(P_t^hg)(x) - g(x)}{t} =  \F{E}_x \left ( \frac{e^{I_t}-1}{t} g(X_t) \right) + \frac{\F{E}_x( g(X_t)) - g(x)}{t}.
\end{equation}
We now consider the limit for $t \downarrow 0$ in the above expression.
The paths $ t \mapsto X_t$ are right continuous and $\F{P}_x(X_0 = x) = 1$ so $\F{P}_x$-a.s there is an $\varepsilon > 0$ such that $ I_t = t h(x)$ on $ [0, \varepsilon)$. It follows that $\F{P}_x$- a.s. 
\begin{equation*}
     h(x) = \lim \limits_{ t \downarrow 0 } \frac{e^{ t h(x)}-1}{t} = \lim \limits_{ t \downarrow 0 }\frac{e^{I_t}-1}{t}.
\end{equation*}
But as $\abs{I_t} \leq t \norm{h}_\infty$ by the monotonicity of the exponential function we have
\begin{equation*}
    \abs{\frac{e^{I_t}-1}{t}} \leq \max \left (\frac{e^{t \norm{h}_\infty}-1}{t},\frac{1-e^{-t\norm{h}_\infty}}{t} \right).
\end{equation*}
As the exponential function is differentiable at $0$, the right hand side of the above equation is bounded (for $t \to 0$). Furthermore, $g$ is bounded and $\lim_{t \downarrow 0}g(X_t) = x$ $\F{P}_x$-a.s., so using the dominated convergence theorem for the left term in \eqref{eq: lemma_tilted_generator} and the definition of $L$ for the right term in \eqref{eq: lemma_tilted_generator}, it follows that
\begin{equation*}
   \lim \limits_{t \downarrow 0}  \frac{(P_t^hg)(x) - g(x)}{t} = h(x)g(x) + (Lg)(x) = (M_hg)(x) + (Lg)(x)
\end{equation*}
for all $x \in E$. As pointwise convergence coincides with convergence in $L^2(\pi)$ (Remark \ref{rem: simplification_finite_state_space}) the claim follows.
\end{proof}

Using the above lemma we can bound $\norm{P_t^{rf}}_2$ in terms of the operators $L, M_{f}$. The following bound holds:

\begin{lemma}
\label{lem: lumer_phillips}
For all $r \in \F{R}$
\begin{equation}
\label{eq: bound_2_norm_feynman_kac}
    \norm{P_t^{rf}}_2 \leq \exp(t \lambda_0(r)),
\end{equation}
where $\lambda_0(r)$ is the largest eigenvalue of the selfadjoint operator $\tilde{L}(r) := \frac{L+L^*}{2} + r M_f$.
\end{lemma}
\begin{reference_paper}
\label{ref: lambda_0(r)_as_supremum}
\begin{enumerate}[$(a)$]
    \item The bound  \eqref{eq: bound_2_norm_feynman_kac} is also found in \cite[Eq.\,(8)]{wu}, \cite[Remark\,3$(a)$]{wu} and \cite[Lemma\,2.2]{lezaud}. In \cite[Eq.\,(8)]{wu} the bound is stated as 
    \begin{equation}
    \label{eq: wu_bound_feynman_kac_1}
        \norm{P_t^{V}}_2 \leq \exp(t \Lambda(V)),
    \end{equation}
    and in \cite[Remark 3(a)]{wu}
    \begin{equation}
    \label{eq: wu_bound_feynman_kac_2}
        \norm{P_t^{V}}_2 \leq \exp(t \Lambda_0(V)),
    \end{equation}
    where (see \cite[Eq.\,(9)]{wu})
    \begin{equation}
        \Lambda(V) = \sup \{- \mathcal{E}^{\sigma}(g,g) + \langle V, g^2 \rangle \, |\, \norm{g}_2 = 1, g \in D(\mathcal{E}^\sigma) \cap L^2(\abs{V}\cdot \mu)\},
    \end{equation}
    and (see \cite[Remark\,3(a)]{wu})
    \begin{equation}
        \Lambda_0(V) = \sup \{\langle Lg,g \rangle + \langle V, g^2 \rangle \, |\, \norm{g}_2 \leq  1, g \in D(L) \} .
    \end{equation}
    Here
    $(\mathcal{E}^\sigma, D(\mathcal{E}^\sigma))$ denotes the closure (see \cite[Ch.\,6.1.4]{Kato} for the definition of the closure of a bilinear form) of the symmetrized Dirichlet form $(\mathcal{E}^\sigma, D(L))$ ($D(L)$ denotes the domain of the $L^2$-infinitesimal generator)
    \begin{equation}
    \label{eq: symmetrized_dirichlet}
        \mathcal{E}^\sigma (g,h) = - \frac{1}{2} ( \langle Lg,h \rangle + \langle g, Lh \rangle ).
    \end{equation}
   In our setting $\mu = \pi$, $V = rf$, $D(L) = L^2(\pi)$, and $L$ is a bounded operator. Thus, $(\mathcal{E}^\sigma, D(L)) = (\mathcal{E}^\sigma, L^2(\pi))$ is (already) a closed symmetric form. Furthermore, $V$ is bounded (by finiteness of $E$), consequently $D(\mathcal{E}^\sigma) \cap L^2(\abs{V}\cdot \mu) = L^2(\mu)$ and
    \begin{align}
    \begin{split}
        \Lambda(V) &= \sup \{- \mathcal{E}^{\sigma}(g,g) + \langle V, g^2 \rangle \, |\, g \in L^2(\mu), \norm{g}_2 = 1 \} \\
        &= \sup \left \{\left \langle  \left (\frac{L+L^*}{2}+r M_f \right) g, g  \right \rangle \, \middle |\, g \in L^2(\mu), \norm{g}_2 = 1 \right \}  \\
        &= \lambda_0(r),
    \end{split}
    \end{align}
    where we used $\langle g, Lg \rangle = \langle L^*g, g \rangle$ and Lemma \ref{lem: biggest_eigenvalue}. Similarly one obtains
    \begin{equation}
        \Lambda_0(V) = \lambda_0(r)
    \end{equation}
    and consequently the bounds \eqref{eq: wu_bound_feynman_kac_1} and \eqref{eq: wu_bound_feynman_kac_2} coincide with the bound of Lemma \ref{lem: lumer_phillips}.
    \item In \cite{lezaud}, not the operator $\tilde{L}(r)$ is considered but instead the operator $\tilde{\Lambda}(r) = - \tilde{L}(r)$ (see \cite[Lemma 2.2]{lezaud}). Furthermore, $\lambda_0(r)$ is defined as the smallest eigenvalue of $\tilde{\Lambda}(r)$ (see \cite[Lemma 2.2]{lezaud}), consequently \cite{lezaud} obtains  \cite[Eq.(2.3)]{lezaud}
    \begin{equation}
        \norm{P_t(r)}_2 \leq \exp(-t\lambda_0(r)),
    \end{equation}
    where $P_t(r) = P_t^{rf}$ as \cite{lezaud} defines $P_t(r)$ as the semigroup generated by $L + rM_f$ (see \cite[P.\,187,Proof of Lemma 2.2]{lezaud}) 
    \item The above lemma follows also from the Lumer-Phillips theorem (following  \cite[P.\,439, Case\,1]{wu}). By writing $\lambda_0(r)$ as a supremum as Lemma \ref{lem: biggest_eigenvalue} and using the identity
    \begin{equation*}
        \langle Lg, g \rangle = \langle g, L^*g \rangle = \langle L^*g, g \rangle,
    \end{equation*}
    it is easy to see that
\begin{equation*}
    \langle (L+r M_f - \lambda_0(r)) g, g \rangle = \left \langle  \left (\frac{L+L^*}{2}+r M_f - \lambda_0(r) \right) g, g  \right \rangle \leq 0
\end{equation*}
for any $g \in L^2(\pi)$, so the operator $L + rM_f - \lambda_0(r)$ is dissipative (see \cite[P.\,81, Ch.\,II.3, Prop.\,3.23]{Engel2006} for a characterization of dissipativity). Furthermore, as the spectrum of any finite dimensional operator is bounded, for $\lambda > 0$ big enough the operator 
\begin{equation*}
    \lambda - (L + rM_f - \lambda_0(r))
\end{equation*}
is surjective
and consequently by the Lumer-Phillips theorem (\cite[P.\,76, Ch.\,II.3,Thrm.\,3.15]{Engel2006}), the semigroup generated by $L + rM_f - \lambda_0(r)$ is a contraction semigroup, i.e.
\begin{equation*}
   \norm{P_t^{rf}e^{-t\lambda_0(r)}}_2 = \norm{\exp [ t (L + rM_f - \lambda_0(r))]}_2 \leq 1,
\end{equation*}
which is the statement of Lemma \ref{lem: lumer_phillips}.
\end{enumerate}
\end{reference_paper}
\begin{remark}
\label{rem: lemma_bound_P_t^rf}
\begin{enumerate}[$(a)$]
    
    \item We will later (see Section \ref{subsubsec: concentration_via_perturbation}) justify the notation $\lambda_0(r)$ by proving that $\lambda_0(r)$ is the perturbation of the eigenvalue $0$ for the perturbed operator $\tilde{L}(r) = \frac{L+L^*}{2} + rM_f$ (recall by Lemma \ref{lem: properties_infinitesimal_generator}$(a)$ that $\lambda_0(0) = 0$ is an eigenvalue of the unperturbed operator $\Tilde{L}(0) = \frac{L+L^*}{2}$).
    \item  If $L$ is self adjoint, then $  L + r M_f$ is also selfadjoint (for $r \in \F{R})$ and it follows (e.g. by diagonalization of $\exp(t(L + rM_f))$ in an orthonormal basis) that
    \begin{equation*}
        \norm{P_t^{rf}}_2 = \norm{\exp(t(L + rM_f))}_2= \exp(t\lambda_0(r)),
    \end{equation*}
    so the bound is exact in the detailed balance case (by Theorem \ref{th: characterization_detailed_balance} $L$  is selfadjoint in the detailed balance case). 
    
\end{enumerate}
\end{remark}

\begin{proof}[Proof of Lemma \ref{lem: lumer_phillips}]
We follow the proof of \cite[Lemma 2.2]{lezaud} (but use our notation). Let $g  \in  L^2(\pi)$, define $\phi(t) = \norm{P_t^{rf}g}_2^2 = \langle P_t^{rf}g P_t^{rf}g\rangle $, and let $L(r) = L + rM_f$ be the generator of $(P_t^{rf})_{t \geq 0}$ . Using the usual product rule (for scalar products on finite vector spaces) yields
\begin{align}
\begin{split}
\label{eq: lemma2_derivative_phi}
    \frac{d}{dt}\phi(t) &=  \langle \frac{d}{dt}P_t^{rf}g, P_t^{rf}g\rangle + \langle P_t^{rf}g, \frac{d}{dt}P_t^{rf}g \rangle =  \langle L(r)P_t^{rf}g, P_t^{rf}g\rangle + \langle P_t^{rf}g, L(r)P_t^{rf}g \rangle \\
    &= \langle (L(r) + L(r)^*)P_t^{rf}g, P_t^{rf}g\rangle  =  2 \left \langle \left  (\frac{L+L^*}{2}+ rM_f \right )P_t^{rf}g, P_t^{rf}g \right \rangle  \\
    &= 2\langle \Tilde{L}(r)P_t^{rf}g, P_t^{rf}g \rangle,
\end{split}
\end{align}
where in the first line Lemma \ref{lem: feynman_kac} and in the second line the selfadjointness of $M_f$ were used. The operator $\Tilde{L}(r)$ is selfadjoint, so  by  Lemma \ref{lem: biggest_eigenvalue} 
\begin{equation}
\label{eq: lemma2_minimax}
    \lambda_0(r) = \sup \left \{ \frac{\langle \Tilde{L}(r)g, g \rangle }{\norm{g}_2^2} \middle | g \in L^2(\pi), g \neq 0 \right \}.
\end{equation}
Using \eqref{eq: lemma2_derivative_phi}  and \eqref{eq: lemma2_minimax} we get
\begin{equation*}
    \frac{d}{dt}\phi(t) =  2\langle \Tilde{L}(r)P_t^{rf}g, P_t^{rf}g \rangle \leq 2 \lambda_0(r) \norm{P_t^{rf}g}_2^2 = 2\lambda_0(r) \phi(t).
\end{equation*}
It follows that $\frac{d}{dt}(e^{-2\lambda_0(r)t}\phi(t)) \leq 0$, which means that $t \mapsto e^{-2\lambda_0(r)t}\phi(t) $ is nonincreasing and thus 
\begin{align*}
    &e^{-2\lambda_0(r)t}\phi(t) \leq \phi(0) = \norm{g}_2^2 \\
    &\Longleftrightarrow \phi(t) \leq e^{2\lambda_0(r)t}\norm{g}_2^2 \\
    &\Longleftrightarrow \norm{P_t^{rf}g}_2 \leq e^{\lambda_0(r)t}\norm{g}_2
\end{align*}
\end{proof}
Finally, by using the above lemma to bound $\norm{P_t^{rf}}_2$ and  recalling \eqref{eq: bound_Psi_Z_prefinal} we get the announced bound
\begin{equation}
\label{eq: final_bound_Phi}
    \Phi(r) := \log(\norm{\frac{d\nu}{d\pi}}_2) + t\lambda_0(r) \geq \Psi_{A_t}(r),
\end{equation}
which holds even for all $r \in \F{R}$. We can now take the Fenchel conjugate of $\Phi$ with respect to $\F{R}_{\geq 0}$ (as the Cramér transform is a Fenchel conjugate taken with respect to $\F{R}_{ \geq 0}$) to derive the
following  concentration inequality (c.f. \cite[Eq.\,(10), Thrm.\,1]{wu}).
\begin{theorem}
\label{th: main_conc_inequality}
In our setting the following concentration inequality holds.
For all $u \geq 0$
\begin{equation}
\label{eq: main_conc_inequality}
\F{P}_\nu \left (\frac{A_t}{t} \geq u \right ) \leq \norm{\frac{d\nu}{d \pi}}_2 e^{- t \lambda_0^*(u)},
\end{equation}
where 
\begin{equation*}
    \lambda_0^*(u)  = \sup_{ r \in \F{R}}(r u - \lambda_0(r)) = \sup_{r \geq 0}(r u - \lambda_0(r))
\end{equation*}
is the Fenchel conjugate of $\lambda_0$ \lb with respect to $\F{R}$\rb \space  for $u \geq 0$. Furthermore, define 
\begin{equation}
    I(u) := \inf \{ - \langle Lg, g \rangle \; | \, \norm{g}_2 = 1, \langle M_fg, g \rangle = u \}
\end{equation}
for $u \in \F{R}$. Then,  we have 
\begin{equation}
\label{eq: fenchel_dual_lambda_0(r)}
   \lambda_0^*(u) = \sup_{ r \in \F{R}}(r u - \lambda_0(r)) = I(u)
\end{equation}
for all $u \in \F{R}$.
\end{theorem}
\begin{reference_paper} 
The above theorem is a reformulation of \cite[Thrm.\,1]{wu} and \cite[Remark\,3]{wu}: Using the notation of \cite{wu} let $V = f$, $\mu = \pi $ and define (as in \cite{wu})
    \begin{equation}
        J_V(u) = \inf \{ \mathcal{E}^\sigma(g,g) \, | \, \norm{g}_2 = 1, g \in  D(\mathcal{E}^\sigma) \cap L^2(\abs{V}\cdot  \mu), \langle V, g^2 \rangle = u \},
    \end{equation}
    and
    \begin{equation}
        J_0(u) = \inf \{ \mathcal{E}^\sigma(g,g) \, | \, \norm{g}_2 = 1, g \in D(L), \langle V, g^2 \rangle = u \},
    \end{equation}
    where $(\mathcal{E}^\sigma, D(\mathcal{E}^\sigma))$ denotes the closure of the symmetrized Dirichlet form (defined in \eqref{eq: symmetrized_dirichlet}). Furthermore, let $I_V(u)$, respectively $I_0(u)$, denote the lower semi-continuous regularization of $J_V$, respectively $J_0$, defined in \cite[Eq.(3)]{wu}. Then, \cite[Thrm.\,1]{wu}, respectively \cite[Remark 3(a)]{wu} states that
    \begin{equation}
    \label{eq: wu_theorem1}
        \F{P}_\nu \left (\frac{A_t}{t} > u \right ) \leq \norm{\frac{d\nu}{d \mu}}_2 e^{- t I_V(u)}
    \end{equation}
    respectively
    \begin{equation}
        \F{P}_\nu \left (\frac{A_t}{t} > u \right ) \leq \norm{\frac{d\nu}{d \mu}}_2 e^{- t I_0(u)}
    \end{equation}
    However, as already explained in Remark \ref{rem: lemma_bound_P_t^rf}, in our setting  $D(L) = D(\mathcal{E}^\sigma) \cap L^2(\abs{V}\cdot  \mu) = L^2(\mu)$ and thus 
    \begin{equation}
        J_V(u) = J_0(u) = I(u) = \lambda_0^*(u),
    \end{equation}
    where in the last equality we used \eqref{eq: fenchel_dual_lambda_0(r)}. But $\lambda_0^* = J_V  = J_0$  is continuous on the interval $\{ \lambda_0^* < \infty \}$ (see the remark below) so $ I_V(u) = I_0(u) = J_V(u) = \lambda_0^*(u)$ (using the definition of the lower semi-continuous regularization; stated for example in \cite[Eq.(3)]{wu}).
\end{reference_paper}

\begin{remark}
\label{rem: th_main_concentration_inequality}
\begin{enumerate}[$(a)$]
    \item (Properties of $\lambda_0^*$) The Fenchel conjugate $\lambda_0^*$ has the following properties (see Appendix, Lemma \ref{lem: appendix_properties_lambda_0^*}) 
    \begin{enumerate}[1.]
        \item $\lambda_0^*\geq 0$
        \item $\{ \lambda_0^* < \infty \} = [\min_{x \in E} f(x), \max_{x \in E}f(x)]$
        \item $\lambda_0^*$ is  convex, continuous on $\{ \lambda_0^* < \infty \}$, and nondecreasing on $[0,\infty)$.
        
    \end{enumerate}
    Hereby, property 1. follows from $\lambda_0(0) = 0$, property 2. is a consequence of \eqref{eq: fenchel_dual_lambda_0(r)}, and property 1. follows from lower semicontinuity  of $\lambda_0^*$.
    \item  (Triviality of the inequality)
    The concentration inequality is trivial whenever 
    \begin{equation*}
        H(u) := \norm{\frac{d\nu}{d\pi}}_2e^{-t\lambda_0^*(u)} \geq 1.
    \end{equation*}
   Let $K := \sup \{ u \geq 0 \; | \; H(u) \geq 1 \}$. Then
    if $\nu \neq \pi$, we have $K  > 0$, $K \geq t^{-1}\F{E}_\nu(A_t)$ and the concentration inequality is trivial for $ u \in [0,K]$. This can be seen as follows. As $\nu \neq \pi$, $\frac{d\nu}{d\pi}$ is not equal to the constant 1-function $\mathbf{1}$ $\pi$- almost surely, so 
    \begin{equation}
        0 <  \text{Var}_\pi \left ( \frac{d \nu}{d \pi} \right )  =  \norm{\frac{d\nu}{d\pi}}_2^2 - 1,
    \end{equation}
    and thus $\norm{\frac{d\nu}{d\pi}}_2 > 1$. Furthermore, $\lambda_0^*(0) = 0$, which implies $H(0) > 1 $. But $\lambda_0^*$ is continuous on $\{ \lambda_0^* < \infty \}$  and nondecreasing on $[0,\infty)$, so $K > 0$ and $H(u) \geq 1$ for $ u \in [0,K]$. Furthermore, because  Theorem \ref{th: main_conc_inequality} is based on Chernoff's inequality (c.f. Lemma \ref{lem: more_explicit_concentration_inequality}), which is always trivial  for $u \leq t^{-1}\F{E}_\nu(A_t) $ by Remark \ref{rem: Chernoff_inequality}, we have $K \geq t^{-1}\F{E}_\nu(A_t)$.
    \item(Scale invariance of the bound)
    As $sM_f = M_{sf}$ the bound is invariant under the replacements $f \to sf$, $u \to su$ for $s > 0$. This is consistent with the invariance of $\F{P}_\nu \left (\frac{A_t}{t} \geq u \right )$ under those replacements. 
    \item (Bound for $\norm{\frac{d \nu}{d \pi}}_2)$ 
    We have 
    \begin{equation*}
        \max_{\nu \in \mathcal{M}_1(E)}\norm{\frac{d\nu}{d\pi}}_2 = \frac{1}{\sqrt{\min_{x \in E}\pi_x}},
    \end{equation*}
    where $\mathcal{M}_1(E)$ denotes the set of probability measures on $E$. The above inequality is checked by
    using $\norm{\frac{d \nu}{d \pi}}_2^2 = \sum_{x \in E} \frac{\nu_x^2}{\pi_x}$  and $\norm{\frac{d\delta_x}{d\pi}}_2 = \frac{1}{\sqrt{\pi_x}}$.
    \item (Optimality and Cramér's Theorem in the detailed balance case)
    In the case where $\pi$ satisfies the detailed balance condition, or equivalently $(\F{P}_\pi, \F{X})$ is reversible or $P_t$ is $\pi$-symmetric (see Theorem \ref{th: characterization_detailed_balance}), then \eqref{eq: main_conc_inequality} is asymptotically sharp; we have 
    \begin{equation}
    \label{eq: asymptotic_sharp}
        \lim_{ t \to \infty} \frac{ \log \F{P}_\nu  \left (\frac{A_t}{t} >  u \right )}{t} = - \lambda_0^*(u)
    \end{equation}
    for all $u \neq \max_{x \in E} f(x)$. The above result is a continuous time analogue of Cramérs theorem (c.f. \cite[Thrm.\,27.2]{Kallenberg2002}) for sums of i.i.d. random variables. The above equality  implies that if $\alpha : [0, \infty) \rightarrow \F{R}_{\geq 0}$ is a function such that a concentration inequality
    \begin{equation*}
        \F{P}_\nu  \left (\frac{A_t}{t} \geq u \right ) \leq \norm{\frac{d\nu}{d\pi}}_2 e^{-\alpha(u)t}
    \end{equation*}
    holds for all $u \geq 0$, then $\alpha(u) \leq \lambda^*(u)$ for all $u \geq 0$.  Thus, the concentration inequality \eqref{eq: main_conc_inequality} is optimal in this sense. The asymptotic sharpness \eqref{eq: asymptotic_sharp} is mentioned  in \cite[Eq.\,(4)]{wu} and \cite[P.\,13]{guillin}, and both works refer to  \cite[Thrm.\,5.3.10]{large_deviations_stroock}. However, as these works do not directly prove or explain how \cite[Thrm.\,5.3.10]{large_deviations_stroock} implies \eqref{eq: asymptotic_sharp} we included our own proof of \eqref{eq: asymptotic_sharp} (based on \cite[Thrm.\,5.3.10]{large_deviations_stroock}), see Lemma \ref{lem: appendix_asymptotic_sharp}. 
\end{enumerate}
\end{remark}

\begin{proof}[Proof of Theorem \ref{th: main_conc_inequality}]
We first show that 
\begin{equation}
    \sup_{r \geq 0} (ru- \lambda_0(r)) = \sup_{r \in \F{R}}( ru -\lambda_0(r))
\end{equation}
for $u \geq 0$. Let $\textbf{1}$ denote the constant 1-function on $E$.  Recall $0 = (L+L^*)\textbf{1}$, $\lambda_0(0) = 0$ (c.f. Lemma \ref{lem: properties_infinitesimal_generator}) and $\pi(f) = 0$ (by our assumption \eqref{eq: pi(f) = 0}). Consequently 
\begin{equation}
\label{eq: lambda_0^*_bigger_0}
    \lambda_0^*(u) = \sup_{ r \in \F{R}} ( ru - \lambda_0(r)) \geq 0 \cdot u - \lambda_0(0) = 0
\end{equation}
and
\begin{equation}
    \lambda_0(r) \geq r \langle M_f \textbf{1}, \textbf{1} \rangle + \frac{1}{2}\langle (L+L^*)\textbf{1}, \textbf{1} \rangle = r \pi(f) +  0 =  0
\end{equation}
for all $r \in \F{R}$, where the second inequality follows from Lemma \ref{lem: biggest_eigenvalue}. Thus,
\begin{equation*}
ru - \lambda_0(r) \leq 0
\end{equation*} 
for all $u \geq 0$ and $ r \leq 0$, which together with \eqref{eq: lambda_0^*_bigger_0} implies 
\begin{equation*}
    \sup_{r \geq 0} (ru- \lambda_0(r)) = \sup_{r \in \F{R}}( ru -\lambda_0(r)) = \lambda_0^*(u)
\end{equation*}
for $u \geq 0$.
As already explained in  Lemma \ref{lem: more_explicit_concentration_inequality}, the bound \eqref{eq: final_bound_Phi}
\begin{equation}
    \Phi(r) = \log(\norm{\frac{d\nu}{d\pi}}_2) + t\lambda_0(r) \geq \Psi_{A_t}(r),
\end{equation}
implies the concentration inequality 
\begin{equation*}
    \F{P}_\nu  \left (\frac{A_t}{t} \geq u \right ) \leq \exp (-\Phi^*(tu))  = \norm{\frac{d\nu}{d\pi}}_2e^{- t \sup_{r \geq 0}(ur - \lambda_0(r))} =  \norm{\frac{d\nu}{d\pi}}_2 e^{-t\lambda_0^*(u)},
\end{equation*}
proving \eqref{eq: main_conc_inequality}.
It remains to prove \eqref{eq: fenchel_dual_lambda_0(r)}. We follow the proof idea of \cite[Thrm.\,1]{wu} and add details. Note that 
\begin{equation}
    I(u) =  \min \{ - \langle Lg, g \rangle \; | \; \norm{g}_2 = 1, \langle M_fg, g \rangle = u \},
\end{equation}
i.e. the infimum is attained (we set $\min \emptyset = \inf \emptyset  = \infty$). Indeed, finite dimensionality of $L^2(\pi)$ implies that  the set $\{ g \in L^2(\pi) \; | \; \norm{g}_2 = 1, \langle M_fg,g \rangle = u \} $ is compact,  and $g \mapsto - \langle Lg,g \rangle$ is a continuous map. Consequently the infimum is attained as the infimum of a continuous function over a compact set is always attained. By definition of $\lambda_0(r)$, Lemma \ref{lem: biggest_eigenvalue} and the identity (which follows from the symmetry of the inner product)
\begin{equation}
    \left \langle \left ( \frac{L+L^*}{2} \right )g, g \right \rangle = \langle Lg, g \rangle
\end{equation}
we have
\begin{align}
\begin{split}
\label{eq: lambda_0(r)_is_Legendre_transform_of_I(u)}
    \lambda_0(r) &= \sup \{ r\langle M_f g, g \rangle +  \langle Lg, g \rangle \; | \; g \in L^2(\pi), \norm{g}_2 = 1 \}\\
    &= \sup \{  ru - (-  \langle Lg, g \rangle) \; | \; g \in L^2(\pi),u \in \F{R}, \norm{g}_2 = 1, \langle M_f g, g \rangle = u \} \\
    &= \sup \{ ru - \inf \{ -  \langle Lg, g \rangle \; | \; g \in L^2(\pi) , \norm{g}_2 =1 , \langle M_f g, g \rangle = u \} \; | \; u \in \F{R} \}\\
    &= I^*(r)
    \end{split}
\end{align}
for all $r \in \F{R}$
where $I^*(r) = \sup_{ u \in \F{R}} (ur - I(u)) = \sup_{u \in \F{R} : I(u)< \infty }(ur - I(u))$ is the Fenchel conjugate of $I$. To finish the proof we show now  that the conditions of the Fenchel-Biconjugation theorem (Lemma \ref{lem: properties_Legendre_transform}$(c)$) are satisfied, i.e. $\{ I < \infty \} = [a,b]$ for some $a < b$, $I$ is convex, and  $I$ is lower semicontinuous on $[a,b]$. Then, by the Fenchel-Biconjugation theorem
and \eqref{eq: lambda_0(r)_is_Legendre_transform_of_I(u)} 
\begin{equation*}
   I(u) =  I^{**}(u) = \sup_{ r \in \F{R}} (ru - \lambda_0(r) ) = \lambda_0^*(u)
\end{equation*}
for all $u \in \F{R}$. We first show $\{ I < \infty \} = [a,b]$. The sphere $S_1 := \{ g \in L^2(\pi) \; | \; \norm{g}_2 = 1 \} $ is compact and connected. Furthermore, the function $D : S_1 \rightarrow \F{R} $ defined via $D(g) := \langle M_f g, g \rangle$ is continuous  ($L^2(\pi)$ is finite dimensional), so the image $\text{Im}(D) \subset \F{R}$ is connected and compact, i.e.  $\text{Im}(D) = [a,b]$, where $a < b$  as $D$ is not constant (because $f$ is not constant). Thus, by definition of $I$ we have
\begin{equation}
    \label{eq: I_smaller_infty} 
    \{I < \infty \} = \text{Im}(D) = [a,b].
\end{equation}

 That $I$ is convex on $[a,b]$ can be seen as follows. Notice that for all $g_1, g_2 \in L^2(\pi)$, all $s_1, s_2 \geq 0$ and all $x,y \in E$ we have that
\begin{equation}
\label{eq: inequality1_for_convexity_of_I}
    s_1 g_1(y) g_1(x) + s_2 g_2(y) g_2(x) \leq \sqrt{s_1g_1(y)^2 + s_2 g_2(y)^2}\sqrt{s_1g_1(x)^2 + s_2 g_2(x)^2},
\end{equation}
which is easily checked by squaring the above inequality. Let $g := \sqrt{s_1g_1^2+s_2g_2^2}$, then  \eqref{eq: inequality1_for_convexity_of_I} implies (by setting $y = X_t$, $x = X_0$ and taking the expectation $\F{E}_x$) 
\begin{equation*}
    s_1 \F{E}_x(g_1(X_t))g_1(x) + s_2 \F{E}_x(g_2(X_t))g_2(x)  \leq \F{E}_x(g(X_t))g(x),
\end{equation*}
which  implies  (by multiplying with $\pi_x$ and summing over $x \in E$)
\begin{equation}
\label{eq: inequality2_for_convexity_of_I}
    s_1 \langle P_tg_1, g_1 \rangle + s_2 \langle P_tg_2, g_2 \rangle \leq \langle P_tg, g \rangle ,
\end{equation}
for all $t \geq 0$. Now we use the above equation to prove the convexity of $I$. For that
let $u_1, u_2 \in [a,b]$, $s_1 + s_2 = 1$ (with $s_1, s_2 \geq 0$), and choose $g_1,g_2 \in S_1$ such that $D(g_i) = u_i$ and $- \langle Lg_i, g_i \rangle = I(u_i)$ for $i = 1,2$ (this is possible as $I(u_i)$ is attained). For the convexity we have to show $I(s_1u_1 + s_2u_2) \leq s_1 I(u_1) + s_2 I(u_2)$. Define $g = \sqrt{s_1g_1^2+s_2g_2^2} $ as before. Then,  $g \in S_1$ and $D(g) = s_1u_1 + s_2u_2$. In particular  $I(s_1u_1+s_2u_2) \leq - \langle Lg, g \rangle$ as $I$ is the infimum. As both  sides of \eqref{eq: inequality2_for_convexity_of_I} equal $1$ at $t = 0$, the inequality holds also for the derivatives at $t = 0$ of both  sides, obtaining 
\begin{equation*}
    - s_1 I(u_1) - s_2 I(u_2) = s_1 \langle Lg_1 ,g_1\rangle + s_2 \langle Lg_2,g_2\rangle  \leq \langle Lg ,g \rangle \leq - I(s_1u_1+s_2u_2),
\end{equation*}
which proves convexity of $I$. Finally, the lower-semicontinuity of $I$ is checked as follows.  Let $ u \in [a,b]$ and $(u_n)_n$ be a sequence in $[a,b]$ converging to $u$. Furthermore let $(g_n)_n$ be a corresponding sequence in $S_1$ such that $D(g_n) = u_n$ and $ - \langle Lg_n ,g_n \rangle = I(u_n)$. Define $I_{-} := \lim \inf_{n \to \infty} I(u_n) $. We can choose a subsequence $(u_{n_k})_k$ such that $I(u_{n_k}) \to I_{-}$ as $k \to \infty $ and such that $g_{n_k}$ converges  to some $g \in S_1$ (as $k \to \infty $, by compactness of $S_1$). Continuity implies $D(g) = \lim_{k \to \infty} D(g_{n_k}) = \lim_{k \to \infty} u_{n_k} = u$ and $ - \langle Lg, g \rangle = \lim_{k \to \infty}- \langle Lg_{n_k}, g_{n_k} \rangle = \lim_{k \to \infty} I(u_{n_k}) = I_{-}$. But as $I(u)$ is defined as an infimum we have $I_{-} =  - \langle Lg, g \rangle  \geq I(u)$. Thus, $I$ is lower semicontinuous and all conditions for the Fenchel-Biconjugation theorem (Lemma \ref{lem: properties_Legendre_transform}$(c)$) are satisfied.

\end{proof}

\subsubsection{Further concentration Inequalities for MJPs}
\label{subsubsec: further_concentration_inequalities_MJP}
As we typically do not have the information about $L$ or $f$, we cannot compute the largest eigenvalue $\lambda_0(r)$ explicitly and thus Theorem \ref{th: main_conc_inequality}, which gives a bound containing
\begin{equation*}
   \lambda_0^*(u) =  \sup_{ r \geq 0} (r u - \lambda_0(r)) = \sup_{r \in \F{R}}(r u - \lambda_0(r)),
\end{equation*}
is not directly applicable.
In the following we will derive more explicit concentration inequalities based on Theorem \ref{th: main_conc_inequality} by finding a lower bound for $ \alpha(u) \leq \lambda_0^*(u) $, thus obtaining 
\begin{equation}
    \F{P}_\nu \left ( \frac{A_t}{t} \geq u  \right) \leq \norm{\frac{d\nu}{d\pi}}_2 e^{- t\alpha(u)}
\end{equation}
For that we present three different approaches following the works \cite{lezaud}, \cite{guillin}, \cite{bernstein}. All these works use implicitly or explicitly Theorem \ref{th: main_conc_inequality} to derive concentration inequalities. The three approaches we present are
\begin{enumerate}
    \item Perturbation theory: We follow \cite{lezaud} and arrive at our version of \cite[Thrm.\,2.4]{lezaud}; Theorem \ref{th: concentration_via_perturbation}.
    \item Functional Inequalities: Poincaré-Inequality, $F$-Sobolev inequality. We follow  \cite{guillin} and arrive at our versions of \cite[Prop.\,1.4]{guillin} and \cite[Thrm.\,2.3]{guillin}; Theorems \ref{th: concentration_via_poincare} and \ref{th: concentration_via_F_sobolev}.
    \item Information inequalities:
     We follow \cite{bernstein} and arrive at our versions of \cite[Thrm.\,2.2]{bernstein} and \cite[Thrm.\,1.2]{bernstein}; Theorems \ref{th: concentration_via_information} and \ref{th: extension_lezaud}. We also extend Theorem \ref{th: concentration_via_perturbation} (c.f. Theorem \ref{th: extension_lezaud})
\end{enumerate}
In approach 3 we derive a lower bound on $\lambda_0^*(u)$ by using the expression (Theorem \ref{th: main_conc_inequality})
\begin{equation}
    \lambda_0^*(u) = I(u) = \inf \{ - \langle Lg, g \rangle \; | \, \norm{g}_2 = 1, \langle M_fg, g \rangle = u \}.
\end{equation}
In approaches 1 and 2 we will make use of the following lemma:
\begin{lemma}[Concentration inequality by bounding $\lambda_0(r)$]
\label{lem: concentration_via_bound_on_lambda_0(r)}
Let $D \subset \F{R}$ and let $G$ be a function, defined on $D$ such that $G(r) \geq \lambda_0(r)$ for all $r \in D$. Then, for all $u \geq 0$
\begin{equation}
\label{eq: concentration_via_bound_on_lambda_0(r)}
  \F{P}_\nu \left (\frac{A_t}{t} \geq u \right ) \leq  \norm{\frac{d\nu }{ d \pi}}_2 \exp( - tG^*(u)), 
\end{equation}
where  $G^*(u) = \sup_{r\in D}(ru- G(r))$ denotes the Fenchel conjugate of $G$ \lb with respect to $D$\rb.
\end{lemma}

\begin{proof}[Proof of Lemma \ref{lem: concentration_via_bound_on_lambda_0(r)}]
As $D \subset \F{R}$ and $G(r) \geq \lambda_0(r)$ for all $r \in D$, by Lemma \ref{lem: properties_Legendre_transform}$(d)$ 
we have 
\begin{equation}
    G^*(u) \leq \lambda_0^*(u) 
\end{equation}
for all $u \geq 0$. Applying this inequality to the  concentration inequality \eqref{eq: main_conc_inequality} (Theorem \ref{th: main_conc_inequality}) yields 
\begin{equation}
    \F{P}_\nu \left (\frac{A_t}{t} \geq u \right ) \leq \norm{\frac{d\nu}{d \pi}}_2 e^{- t \lambda_0^*(u)} \leq \norm{\frac{d\nu}{d \pi}}_2 e^{-tG^*(u)}
\end{equation}
for all $u \geq 0$.
\end{proof}

\subsubsection{Concentration Inequalities via Perturbation Theory}
\label{subsubsec: concentration_via_perturbation}
In this section we use Lemma \ref{lem: concentration_via_bound_on_lambda_0(r)} and bound $\lambda_0(r)$ using perturbation theory; based on \cite{lezaud}. The main result of this section is Theorem \ref{th: concentration_via_perturbation}, which is our version of \cite[Thrm.\,2.4]{lezaud}. The most relevant notions and results of Section \ref{sec: preliminaries} are: Example \ref{ex: Legendre_transform_subgamma}, perturbation theory (Section \ref{subsubsec: perturbation_theory}); in particular reduced resolvent and Theorem \ref{th: perturbation_simple_eigenvalue}, and properties of the infinitesimal generator (Lemma \ref{lem: properties_infinitesimal_generator}).\\
We proceed as follows. The basic idea is to see $\lambda_0(r)$ as the perturbation of the eigenvalue $0$ of $L$, then use the formulae of Theorem \ref{th: perturbation_simple_eigenvalue} (which yields an expression for $\lambda_0(r)$ in terms of a perturbation series) to compute a bound $G(r) \geq \lambda_0(r)$ and apply Lemma \ref{lem: concentration_via_bound_on_lambda_0(r)}. \\
 Let us now explain more precisely in the following lemma how Theorem \ref{th: perturbation_simple_eigenvalue} applies in our setting (defined in Section \ref{subsubsec: setting_notation}). 
\begin{lemma}
\label{lem: preparation_application_perturbation}
Refering to our setting, let $\lambda_1 = \min_{ \lambda \in \sigma \left (\frac{L+L^*}{2} \right) } \abs{\lambda}$ be the spectral gap of $\frac{L+L^*}{2}$. Then, for all $r \in [0, \frac{\lambda_1}{2 \norm{f}_\infty})$ the ball $B_{\frac{\lambda_1}{2}}(0)$ contains exactly one simple eigenvalue of $\tilde{L}(r)$ and this eigenvalue is $\lambda_0(r)$, the largest eigenvalue of $\tilde{L}(r)$. Furthermore,
\begin{equation}
\label{eq: lambda_0(r)_perturbation_series}
    \lambda_0(r) = \sum_{n = 1}^{\infty}\lambda_0^{(n)}r^n
\end{equation}
with
\begin{equation}
  \lambda_0^{(n)} = \frac{(-1)^n}{n} \sum\limits_{\substack{k_1,..k_n \in \F{Z}_+ \\ k_1 + ... k_n = n-1}} \Tr( M_fS^{(k_1)} ... M_fS^{(k_n)})
  \label{eq: lambda_0^(n)}
\end{equation}
for all $ r \in  [0, \frac{\lambda_1}{2 \norm{f}_\infty}) $. Hereby,
\begin{equation}
\label{eq: S^(k)}
    S^{(0)} = - \textup{pr} \;, \; S^{(k)} = S^k \text{\space for $k \geq 1$,}
\end{equation}
where $S$ is the reduced resolvent of $\frac{L+L^*}{2}$ with respect to the eigenvalue $0$ and $\textup{pr}$ is the orthogonal projection onto the eigenspace corresponding to the eigenvalue 0, i.e. 
\begin{equation}
\label{eq: projection_P}
    \textup{pr}(g) = \langle g, \textup{\textbf{1}} \rangle \textup{\textbf{1}} = \pi(g) \textup{\textbf{1}}
\end{equation}
for all $g \in L^2(\pi)$.

\end{lemma}
\begin{reference_paper}
The above lemma is not explicitly stated as a  lemma in \cite{lezaud}, so we include our own proof. In particular, we include the proof of the statement that $\lambda_0(r)$ is the perturbation of the eigenvalue $0$, i.e.  the unique, simple eigenvalue of $\tilde{L}(r)$ contained in $B_{\frac{\lambda_1}{2}}(0)$.
\end{reference_paper}

\begin{remark}
\label{rem: lambda_0^(1)}
For $n = 1$ we get 
\begin{equation*}
    \lambda_0^{(1)} = \Tr(M_f\textup{pr}) = \sum_{x \in E}(M_f\text{pr}(e_x))(x) = \sum_{x \in E}(M_f\textbf{1})(x)\pi_x =  \sum_{x \in E} \pi_x f(x) =  \pi(f) = 0,
\end{equation*}
where the trace was evaluated in the basis $e_x(y) = \delta_{xy}$ and $\text{pr}(e_x) = \langle \textbf{1}, e_x \rangle \textbf{1} = \pi_x \textbf{1}$ was used.

\end{remark}

\begin{proof}[Proof of Lemma \ref{lem: preparation_application_perturbation}]
Using the notation of Theorem \ref{th: perturbation_simple_eigenvalue} set $T = \frac{L+L^*}{2}$ and $T' = M_f$. Note that $\norm{M_f}_2 = \norm{f}_\infty$. Furthermore, let 
\begin{equation}
    I = \left[0, \frac{\lambda_1}{2 \norm{M_f}_2}\right ) = \left[0, \frac{\lambda_1}{2 \norm{f}_\infty}\right ) \text{\space and \space } B = B_{\frac{\lambda_1}{2}}(0).
\end{equation}
Note that $0$ is a simple eigenvalue of $T = \frac{L+L^*}{2}$ (Lemma \ref{lem: properties_infinitesimal_generator}$(a)$) and $T = \frac{L+L^*}{2}$ is self adjoint, so the conditions of Theorem \ref{th: perturbation_simple_eigenvalue} are satisfied.  According to this theorem, for all $ r \in I $ the ball $B$ contains exactly one eigenvalue of $\tilde{L}(r)$ and this eigenvalue is simple. We now show that for $r \in I$, this eigenvalue contained in $B$ is exactly $\lambda_0(r)$, the largest eigenvalue of $\tilde{L}(r)$, making it possible to express $\lambda_0(r)$ using the perturbation series of Theorem \ref{th: perturbation_simple_eigenvalue}. Let $N = \#E = \dim L^2(\pi)$. By a combination of Lemma \ref{lem: properties_infinitesimal_generator}$(a),(d)$ and Lemma \ref{lem: eigenvalues_continuous}, there are continuous (real valued) functions $\mu_0,...,\mu_{N-1}$ defined on $I$, representing the repeated eigenvalues of $\tilde{L}(r)$ such that (after relabeling) $\max\{\mu_1(0),...,\mu_{N-1}(0)\} < \mu_0(0) = 0$. We have that $\mu_0(r) \in B$ for all $r \in I $. Indeed, suppose there is an $r \in I$ such that $\mu_0(r) \not \in B$. Let $s \in I$ be the infimum of all such $r \in I$ with $\mu_0(r) \not \in B$, so that $\mu_0(r) \in B$ for all $r \in [0,s)$. Then, as $\mu_0(0) \in B$ and $B$ is open, the continuity of $\mu_0$ implies $ s > 0 $ and $\mu_0(s) \not \in B$. But, as $B$ always contains a simple eigenvalue of $\tilde{L}(r)$ for $r \in I$, there is a $\mu_k$ (with $ k \neq 0$) such that $\mu_k(s) \in B$. Because of continuity of $\mu_k$ it follows that $\mu_0(s-\varepsilon), \mu_k(s-\varepsilon) \in B$ for small enough $\varepsilon > 0$, contradicting the fact that $B$ contains exactly one simple eigenvalue of $\tilde{L}(r)$ for $r \in I$. Thus, for all $r \in I$, $\mu_0(r)$ is the unique, simple eigenvalue of $\Tilde{L}(r)$ contained in $B$. But $\mu_0(r)$ is also the largest eigenvalue of $\tilde{L}(r)$, i.e. $\mu_0(r) = \lambda_0(r)$. Indeed, because the functions $\mu_1,...,\mu_{N-1}$ are continuous, $\max \{\mu_1(0),...,\mu_{N-1}(0)\} < \mu_0(0) = 0$ and $\mu_0(r)$ is the unique, simple eigenvalue contained in $B$, it follows that $\max \{\mu_1(r),...,\mu_{N-1}(r)\} < \mu_0(r)$  for all $r \in I$ and consequently $\mu_0(r)$ is also the largest eigenvalue of $\tilde{L}(r)$. Thus, the formulae \eqref{eq: lambda_0(r)_perturbation_series}, \eqref{eq: lambda_0^(n)} and \eqref{eq: S^(k)} follow from the formulae of Theorem \ref{th: perturbation_simple_eigenvalue}. Finally, \eqref{eq: projection_P} follows from Lemma \ref{lem: properties_infinitesimal_generator}$(a)$ and the computation of orthogonal projections in Hilbert spaces.
\end{proof}

As the above expression of $\lambda_0(r)$ as an infinite series in $r$ is not practical for computing $\lambda_0^*$, we derive an upper bound $G(r) \geq \lambda_0(r)$  using the above series representation of $\lambda_0(r)$ to obtain a explicit concentration inequality (invoking Lemma \ref{lem: concentration_via_bound_on_lambda_0(r)}). The derivation of the upper bound $G(r)$ we will involve algebraic manipulations with the reduced resolvent $S$. We will make use of the following lemma, which summarizes the most important properties of the reduced resolvent.
\begin{lemma}[Properties of the reduced resolvent]
\label{lem: properties_reduced_resolvent}
Let $L$ be the infinitesimal generator of an irreducible $MJP$. Let $S $ be the reduced resolvent of $\frac{L + L^*}{2}$ with respect to the eigenvalue $0$ (see Definition \ref{def: reduced_resolvent}). Let $\hat{S} := - S$ and let $\lambda_1 :=  \min \limits_{\lambda \in \sigma(\frac{L+L^*}{2}) \backslash \{0\}} \abs{\lambda }$ be the spectral gap. Then, the following statements hold:
\begin{enumerate}[\lb a\rb]
    \item $S$ is selfadjoint and negative semidefinite.
    \item $\norm{S}_2 = \frac{1}{\lambda_1}$ and $\textup{Ker}(S) = \textup{span}(\textup{\textbf{1}})$.
    \item For all $r \in \F{R}$, a unique selfadjoint operator $\hat{S}^r$ can be defined such that for any eigenvector $h \in L^2(\pi)$, with eigenvalue $\mu \geq 0$, of $\hat{S}$ we have 
    \begin{equation}
    \label{eq: definition_S_hat}
       \hat{S}^rh =  \begin{cases}
                        \mu^r \cdot h \; \mathrm{for} \; \mu > 0 \\
                        0 \; \mathrm{for} \; \mu = 0.
                      \end{cases}
    \end{equation}
    \item Furthermore,
    \begin{equation}
        \norm{\hat{S}^r}_2 = \frac{1}{\lambda_1^r},\text{\,} \hat{S}^{r_1 + r_2} = \hat{S}^{r_1}\hat{S}^{r_2},
    \end{equation}
    and for all $f \in \{\textup{\textbf{1}\}}^\perp $ we have $\hat{S}^{-r}\hat{S}^rf = f$.
\end{enumerate}

\end{lemma}
\begin{proof}
We prove all statements by diagonalizing $\frac{L+L^*}{2}$. Let  $(b_k)_{k = 1,...,N}$ be an orthonormal basis that diagonalizes $\frac{L+L^*}{2}$, where $N = \#E = \dim L^2(\pi)$. Furthermore, let $\mu_k$ denote the corresponding eigenvalues and let $\text{pr}_k$ denote the orthogonal projections onto $\text{span}(b_k)$, i.e. $\text{pr}_k(f) = \langle b_k, f \rangle b_k$ for all $f \in L^2(\pi)$. Using Lemma \ref{lem: properties_infinitesimal_generator}$(a)$ and $(d)$ we can assume  $\mu_1 = 0$ and $\mu_k < 0$ for $k \geq 2$, in particular 
\begin{equation}
    \frac{L+L^*}{2} = \sum_{k = 2}^N \mu_k \text{pr}_k
\end{equation}
and $\text{pr}_1$ is the projection onto the eigenspace corresponding to the eigenvalue $0$.
Using the expression of Remark \ref{rem: reduced_resolvent}  for the reduced resolvent $S$ we get 
\begin{equation}
    S = \sum_{k = 2}^N \frac{1}{\mu_k} \text{pr}_k.
\end{equation}
This equation implies immediately statements $(a)$ and $(b)$. Moreover, it also implies that 
\begin{equation}
    \hat{S} = \sum_{k = 2}^N \frac{1}{-\mu_k}\text{pr}_k.
\end{equation}
As $\frac{1}{-\mu_k} > 0$ for all $k \geq 2$ the operator $\hat{S}^r$ defined via
\begin{equation}
    \hat{S}^r = \sum_{k = 2}^N \frac{1}{(-\mu_k)^r} \textup{pr}_k
\end{equation}
is well defined, selfadjoint and satisfies \eqref{eq: definition_S_hat} by definition. Moreover, as $(b_k)_{k = 1,...,N}$ is a basis diagonalizing $\hat{S}$, \eqref{eq: definition_S_hat} uniquely determines $\hat{S}^r$ and thus statement $(c)$ is proved. The first two statements of $(d)$ follow directly from the above representation of $\hat{S}^r$, Finally by Lemma \ref{lem: properties_infinitesimal_generator}$(a)$ we have $\text{span}(b_1) = \text{span}(\textup{\textbf{1}})$, so $f \in \{\textbf{1}\}^\perp$ implies that $f = \sum_{k = 2}^N \text{pr}_k f$. The equality $\hat{S}^{-r}\hat{S}^rf = f$ then follows directly from the above representation of $\hat{S}^r$, respectively $\hat{S}^{-r}$. 
\end{proof}
\begin{reference_paper}
Notice that $Sf$ can also be written as (c.f. \cite[P.\,188]{lezaud}, \cite[P.\,363]{bernstein})
\begin{equation}
    Sf = \int_{0}^\infty \exp(t\frac{L+L^*}{2})f dt
\end{equation}
for any $f \in \{\textbf{1} \}^\perp$.
\end{reference_paper}
Using the above lemma we can now derive the following bound on $\lambda_0(r)$.

\begin{lemma}[Bound for $\lambda_0(r)$ via perturbation theory]
\label{lem: lezaud_bound_lambda_0(r)}
Define 
\begin{equation}
    \Phi(x) := \left ( \frac{1-x}{2} \right) \left( 1-  \sqrt{1-\frac{4x^2}{(1-x)^2}} \right)
\end{equation}
 for $ -1 \leq x \leq \frac{1}{3}$. Then, for all $ 0 \leq r \leq \frac{\lambda_1}{3 \norm{f}_\infty}$ 
\begin{equation}
\label{eq: lezaud_bound_lambda_0(r)}
    \lambda_0(r) \leq  \left( \frac{\hat{\sigma}_f^2\lambda_1^2}{2 \norm{f}_{\infty}^2} \right) \Phi\left (\frac{\norm{f}_{\infty}r}{\lambda_1} \right) \leq \frac{r^2\frac{\hat{\sigma}_f^2}{2}}{1-2\left(\frac{\norm{f}_{\infty}}{\lambda1}\right) r},
\end{equation}
 where  $S$ denotes the reduced resolvent of $\frac{L+L^*}{2}$ with respect to the eigenvalue $0$ and $\hat{\sigma}_f := - 2 \langle f, Sf \rangle $.
\end{lemma}
\begin{reference_paper}
 As \cite{lezaud} defines $\lambda_0(r)$ as the smallest eigenvalue of $- \frac{L+L^*}{2}$ the corresponding bound of \cite[Lemma 2.3]{lezaud} is stated as 
    \begin{equation}
       - \lambda_0(r) \geq - \left( \frac{\hat{\sigma}_f^2\lambda_1^2}{2 \norm{f}_{\infty}^2} \right) \Phi\left (\frac{\norm{f}_{\infty}r}{\lambda_1} \right) \geq  -\frac{r^2\frac{\hat{\sigma}_f^2}{2}}{1-2\left(\frac{\norm{f}_{\infty}}{\lambda1}\right) r}.
    \end{equation}
\end{reference_paper}

\begin{remark}
\label{rem: lemma_bound_lammba_(r)_via_perturbation_theory}
\begin{enumerate}[$(a)$]
    \item Although the presented proof below just works for \\ $0 \leq r \leq \frac{\lambda_1}{3 \norm{f}_\infty}$ (see Remark \ref{rem: lezauds_fault}) the bound 
    \begin{equation*}
    \lambda_0(r)  \leq \frac{r^2\frac{\hat{\sigma}_f^2}{2}}{1-2\left(\frac{\norm{f}_{\infty}}{\lambda1}\right) r}
    \end{equation*}
    actually holds also for $0 \leq r < \frac{\lambda_1}{2 \norm{f}}$ (see Theorem \ref{th: extension_lezaud}).
    \item If the   process starts with stationary initial conditions, i.e. $\nu = \pi$, then  using that this lemma holds also for $0 \leq r < \frac{\lambda_1}{2 \norm{f}}$ and the upper bound $\Phi(r) = t \lambda_0(r)$  on $\Psi_{A_t}$ (see \eqref{eq: final_bound_Phi}), it follows immediately that $A_t$ is sub-gamma (on the right tail) with variance $t \hat{\sigma}_f^2$ and scale parameter $\frac{2 \norm{f}_\infty}{\lambda_1}$. 
    
\end{enumerate}
\end{remark}
\begin{proof}[Proof of Lemma \ref{lem: lezaud_bound_lambda_0(r)}]
We follow the proof idea of \cite[Lemma 2.3]{lezaud} and add computational details. To simplify the notation and for enhanced readability of the proof we adopt the following notation for the product of operators. Let $m \geq l$, define for operators $A_{l},A_{l+1},...,A_{m}$ (defined on a common vector space)
\begin{equation}
     \prod_{i=l, \curvearrowright}^{m} A_i := A_lA_{l+1} ... A_m.
\end{equation}
If $m < l$ define 
\begin{equation}
    \prod_{i = l, \curvearrowright}^m A_i := 1.
\end{equation}
To get an upper bound for $\lambda_0(r)$ we bound each coefficient $\lambda_0^{(n)}$ in the series (c.f. Lemma \ref{lem: preparation_application_perturbation})
\begin{equation}
    \lambda_0(r) = \sum_{n = 1}^\infty \lambda_0^{(n)}.
\end{equation}
Recall  that $\lambda_0^{(1)} = 0$ (see Remark \ref{rem: lambda_0^(1)}), so we only have to consider $n \geq 2$. Fix  some $n \geq 2$. We rewrite the expression for $\lambda_0^{(n)}$ (c.f. Lemma \ref{lem: preparation_application_perturbation})
\begin{equation}
  \lambda_0^{(n)} = \frac{(-1)^n}{n} \sum\limits_{\substack{k_1,..k_n \in \F{Z}_+ \\ k_1 + ... k_n = n-1}} \Tr( M_fS^{(k_1)} ... M_fS^{(k_n)})
  \label{eq: lambda_0^(n)_proof}
\end{equation}
as follows. Call a permutation $\sigma : \{1,...,n \} \rightarrow \{1,...,n \} $ circular if it is a multiple of the permutation $(1 \; 2 \;... \; n)$ (in cycle notation). Then using the well known trace identity  for finite dimensional operators $A_1,A_2 $, i.e.  $\Tr(A_1A_2) = \Tr(A_2A_1)$, it follows that 
\begin{equation}
\label{eq: Tr_invariance_circular_permutation}
     \Tr( \prod_{i = 1,   \curvearrowright}^{k} A_i) = \Tr( \prod_{i = 1, \curvearrowright}^{k}A_{\sigma(i)})
\end{equation}
where $A_1,...A_k$ are arbitrary operators and $\sigma$ is some circular permutation. Rewriting \eqref{eq: lambda_0^(n)_proof} yields 
\begin{equation}
    \lambda_0^{(n)} = \frac{(-1)^n}{n} \sum\limits_{\substack{k_1,...,k_n \in \F{Z}_+ \\ k_1 + ... + k_n = n-1}} \Tr( \prod_{i=1,\curvearrowright}^nM_fS^{(k_i)}) .
\end{equation}

To sum over the  Tr-terms having the same value in the above expression of $\lambda_0^{(n)}$ (due to invariance with respect to circular permutations), we define an equivalence relation on the set of all $\F{Z}_+$ valued sequences $(k_1,...,k_n)$ with $k_1 + ... + k_n = n-1$ by $(k_1,...,k_n) \sim  (m_1,...,m_n)$ if and only if there is a circular permutation $\sigma $ such that $(m_1,..., m_n) = (k_{\sigma(1)},...,k_{\sigma(n)})$. It can be shown (see Lemma \ref{lem: appendix_equivalence_classes}) that all equivalence classes $[k_1,...,k_n]$ contain exactly $n$ elements. Furthermore, because $m_1 + ... +m_n = n-1$, at least one $m_i$ is zero, so the equivalence class $[m_1,...,m_n]$ of $(m_1,...,m_n)$ always has a representation $[k_1,..k_{n-1},0]$, where $k_1,...,k_{n-1} \in \F{Z}_+$ and $k_1 + ...+ k_{n-1} = n -1 $. Thus, combining the aforementioned facts with \eqref{eq: Tr_invariance_circular_permutation} yields  that  $\lambda_0^{(n)}$ can be rewritten as 
\begin{align}
\begin{split}
    \label{eq: lemma3_lambda_0(n)_2}
    \lambda_0^{(n)} &= (-1)^n \sum_{[k_1,...,k_{n-1},0]}\Tr(\left(\prod_{i = 1, \curvearrowright}^{n-1}M_fS^{(k_i)}\right)M_fS^{(0)}) \\
    &= (-1)^n\sum_{[k_1,...,k_{n-1},0]}\Tr(M_f\left(\prod_{i=1,\curvearrowright}^{n-1}S^{(k_i)}M_f\right)S^{(0)}),
\end{split}
\end{align}
where the sum is taken over all equivalence classes. To bound $\abs{\lambda_0^{(n)}}$ we will now bound the traces in \eqref{eq: lemma3_lambda_0(n)_2}. Define 
\begin{equation}
    \Tr([k_1,...,k_{n-1},0]) := \Tr(M_f\left(\prod_{i=1,\curvearrowright}^{n-1}S^{(k_i)}M_f\right)S^{(0)}).
\end{equation}
Note that only the equivalence classes $[k_1,...,k_{n-1},0]$ without two consecutive zeros lead to nonzero traces in \eqref{eq: lemma3_lambda_0(n)_2}. This can be seen as follows.  Let $(e_x)_{x \in E}$ be the basis of $L^2(\pi)$ defined via $e_x(y) := \delta_{xy}$ for $x,y \in E$. In this basis the trace $\Tr([k_1,...,k_{n-1},0])$ is given by
\begin{align}
\label{eq: lemma_3_lambda_0^(n)_3}
\begin{split}
    &\Tr([k_1,...,k_{n-1},0]) = 
    \sum_{x \in E}-\left (M_f \left(\prod_{i=1, \curvearrowright}^{n-1}S^{(k_i)}M_f \right) \text{pr}(e_x) \right )(x) \\
    &= \sum_{x\in E}-  \left(M_f \left(\prod_{i=1, \curvearrowright}^{n-1}S^{(k_i)}M_f \right) \textbf{1} \right)(x) \cdot \pi_x = \sum_{x\in E}-f(x) \cdot  \left( \left(\prod_{i=1, \curvearrowright}^{n-1}S^{(k_i)}M_f \right) \textbf{1} \right)(x) \cdot \pi_x\\
    &= - \left \langle f,  \left(\prod_{i=1, \curvearrowright}^{n-1}S^{(k_i)}M_f \right) \textbf{1}  \right\rangle ,
    \end{split}
\end{align}
where we used  the identities 
(c.f. Lemma \ref{lem: preparation_application_perturbation}) 
\begin{equation}
    S^{(0)}e_x = - \text{pr}(e_x) = -\langle e_x, \mathbf{1} \rangle \textbf{1}= -\pi_x \textbf{1}
\end{equation}
and 
\begin{equation}
   ( M_fg)(x) = f(x) g(x) .
\end{equation}
 If $[k_1,..k_{n-1},0]$ has two consecutive zeros, then by applying a circular permutation, we can assume that $k_{n-1} = 0$. But if $k_{n-1} = 0$, then using \eqref{eq: lemma_3_lambda_0^(n)_3} and $S^{(0)}M_f\textbf{1} = -PM_f\textbf{1} = -\pi(f) \mathbf{1} = 0$  yields $\Tr([k_1,...,k_{n-1},0]) =   0$. Thus it is sufficient to bound the traces in \eqref{eq: lemma3_lambda_0(n)_2} described by equivalence classes $[k_1,...,k_{n-1},0]$ with no adjacent zeros. \par In the following  let $[k_1,...k_{n-1},0]$ be such an equivalence class and let $m \geq 1$ be the number of (non adjacent) zeros in $[k_1,...,k_{n-1},0]$. We will now derive an upper bound for $\abs{\Tr([k_1,...,k_{n-1},0])}$, leading via \eqref{eq: lemma3_lambda_0(n)_2} to a upper bound for $\abs{\lambda_0^{(n)}}$. We first rewrite the last term in \eqref{eq: lemma_3_lambda_0^(n)_3}
 \begin{equation}
     - \left \langle f,  \left(\prod_{i=1, \curvearrowright}^{n-1}S^{(k_i)}M_f \right) \textbf{1}  \right\rangle
 \end{equation}
 as follows. Let $1 \leq j_1 < ... < j_m \leq n$ be the positions of the zeros in $(k_1,...,k_{n-1},0)$. In particular, $j_m = n$ and $k_{j_l} = 0$ for $l = 1,...,m$. Furthermore, because there are no consecutive zeros we have $k_1 \neq 0$ and $k_{j_l \pm 1} \neq 0$ for all $l = 1,...,m$. Then, using $j_1,...,j_m$ and $k_{j_l} = 0$
 we rewrite 
\begin{align}
\begin{split}
     &- \left \langle f,  \left(\prod_{i=1, \curvearrowright}^{n-1}S^{(k_i)}M_f \right) \textbf{1}  \right\rangle =  - \left \langle f, \left (\prod_{i = 1, \curvearrowright}^{j_1 -1}S^{(k_i)}M_f \right ) \left (\prod_{l = 1, \curvearrowright}^{m-1}  \left ( \prod_{i = j_{l}, \curvearrowright}^{j_{l+1}-1} S^{(k_i)}M_f \right ) \right ) \textbf{1} \right \rangle  \\
     &= - \left \langle f, \left (\prod_{i = 1, \curvearrowright}^{j_1 -1}S^{(k_i)}M_f \right ) \left (\prod_{l = 1, \curvearrowright}^{m-1}  S^{(0)}M_f \left ( \prod_{i = j_{l}+1, \curvearrowright}^{j_{l+1}-1} S^{(k_i)}M_f \right ) \right ) \textbf{1} \right \rangle .
\label{eq: scalar_product_with_f}
\end{split}
\end{align}
Now for all $m \in \F{N}$ and all operators $A_1,...,A_{m-1}$ (on $L^2(\pi)$) it holds that
\begin{equation}
\label{eq: iterated_product}
   \left ( \prod_{l = 1, \curvearrowright}^{m-1} S^{(0)}M_f A_l\right )\textbf{1} = (-1)^{m-1} \left (\prod_{l = 1}^{m-1} \langle f, A_l \textbf{1} \rangle \right ) \textbf{1},
\end{equation}
which follows by induction on $m$, using
\begin{equation*}
    S^{(0)}M_fg = - \text{pr}(fg) = - \langle \textbf{1},fg \rangle \textbf{1} = - \langle f, g \rangle \textbf{1}
\end{equation*}
for $g \in L^2(\pi)$.
 Applying \eqref{eq: iterated_product} to the product $\prod_{l = 1, \curvearrowright}^{m-1} ...$ in the last line of \eqref{eq: scalar_product_with_f}, setting $A_l = \left ( \prod_{i = j_{l}+1, \curvearrowright}^{j_{l+1}-1} S^{(k_i)}M_f \right )  $ and $j_0 := 0$, gives (continuing the chain of equalities in \eqref{eq: scalar_product_with_f})
\begin{align}
\begin{split}
\label{eq: lemma3_expression_Tr}
    & - \left \langle f,  \left(\prod_{i=1, \curvearrowright}^{n-1}S^{(k_i)}M_f \right) \textbf{1}  \right\rangle \\
    &=  (-1)^m\left \langle f , \left (\prod_{i = 1, \curvearrowright}^{j_1 -1}S^{(k_i)}M_f \right ) \left (\prod_{l = 1}^{m-1} \left \langle f, \left ( \prod_{i = j_{l}+1, \curvearrowright}^{j_{l+1}-1} S^{(k_i)}M_f \right) \textbf{1} \right \rangle \right ) \textbf{1} \right \rangle  \\
    &= (-1)^m \prod_{l = 1}^m \left \langle f, \left ( \prod_{i = j_{l-1}+1, \curvearrowright}^{j_l-1} S^{(k_i)}M_f \right) \textbf{1} \right \rangle 
      = (-1)^m \prod_{l = 1}^m \left \langle f, \left ( \prod_{i = j_{l-1}+1, \curvearrowright}^{j_l-1} S^{k_i}M_f \right) \textbf{1} \right \rangle,
\end{split}
\end{align}
where in the last equality  $S^{(k)} = S^k$ for $k \geq 1$ was used (by definition of $j_1,...j_m$ we have that $k_i \geq 1 $ for all $i \in \{1,...,n\} \backslash \{j_1,...,j_m \}$). Thus, we obtain (recalling \eqref{eq: lemma_3_lambda_0^(n)_3}) 
\begin{equation}
\label{eq: Tr[k_1,...,k_n]_absolute_value}
   \abs{ \Tr([k_1,...,k_{n-1},0])} =  \prod_{l = 1}^m \abs{\left \langle f, \left ( \prod_{i = j_{l-1}+1, \curvearrowright}^{j_l-1} S^{k_i}M_f \right) \textbf{1} \right \rangle}.
\end{equation}
The next objective is to bound each of the $m$ factors in the above equation to get a bound for the trace. We show the following bound for each factor: 
\begin{equation}
\label{eq: bound_for_each_factor}
    \abs{\left \langle f, \left ( \prod_{i = j_{l-1}+1, \curvearrowright}^{j_l-1} S^{k_i}M_f \right) \textbf{1} \right \rangle} \leq \frac{\hat{\sigma}_f^2 \lambda_1}{2} \left (\frac{1}{\lambda_1} \right)^{ \sum_{i = j_{l-1}+1}^{j_l -1}k_i} \norm{f}_\infty^{j_{l}-j_{l-1}-2}.
\end{equation}
To show this bound we will use results of Lemma \ref{lem: properties_reduced_resolvent}.
Notice that the product $\prod_{i = j_{l-1}+1, \curvearrowright}^{j_l-1}...$ in each factor is nonempty, i.e. $j_{l-1}+1 \leq j_l -1$ for all $l = 1,..., m$, because we consider equivalence classes $[k_1,...,k_{n-1},0]$ with no adjacent zeros.  Let $l \in \{1,...,m \}$ and define $\hat{S} = - S$ as in Lemma \ref{lem: properties_reduced_resolvent}. Then,

\begin{align}
\begin{split} 
&\abs{\left \langle f, \left ( \prod_{i = j_{l-1}+1, \curvearrowright}^{j_l-1} S^{k_i}M_f \right) \textbf{1} \right \rangle} = \abs{ \left \langle \hat{S}^{k_{j_{l-1}+1}}f , \left ( \prod_{i = j_{l-1} + 2, \curvearrowright}^{j_l -1}M_f S^{k_i} \right) f \right \rangle } \\
&= \abs{ \left \langle \hat{S}^{k_{j_{l-1}+1}-\frac{1}{2}}\hat{S}^{\frac{1}{2}}f , \left ( \prod_{i = j_{l-1} + 2, \curvearrowright}^{j_l -1}M_f S^{k_i} \right) \hat{S}^{-\frac{1}{2}}\hat{S}^{\frac{1}{2}} f \right \rangle }  \\
&\leq \norm{\hat{S}^{k_{j_{l-1}+1}-\frac{1}{2}}}_2 \norm{\hat{S}^{\frac{1}{2}}f}_2 \left (\prod_{i = j_{l-1}+2}^{j_l-1}\norm{M_f}_2 \norm{S^{k_i}}_2  \right) \norm{\hat{S}^{-\frac{1}{2}}}_2 \norm{\hat{S}^{\frac{1}{2}}f}_2 \\
&= \norm{\hat{S}^{\frac{1}{2}}f}_2^2\lambda_1  \left (\frac{1}{\lambda_1} \right)^{\sum_{i = j_{l-1}+1}^{j_l - 1}k_i} \norm{f}_\infty^{j_{l}-j_{l-1}-2}.
\end{split} 
\end{align}
In the first line we used selfadjointness of $S$ (Lemma \ref{lem: properties_reduced_resolvent}$(a)$), and $M_f\textbf{1} = f$. In the second line we used Lemma \ref{lem: properties_reduced_resolvent}$(d)$ (the conditions are satisfied because $\langle f, \textbf{1} \rangle = \pi(f) =  0$). Finally, in the last two lines we used the Cauchy-Schwartz inequality, submultiplicativity of the operator norm $\norm{\cdot}_2$, and Lemma \ref{lem: properties_reduced_resolvent}$(b),(d)$ (to replace $\norm{S^r}$ respectively $\norm{\hat{S}^r}$ by $\left (\frac{1}{\lambda_1}\right )^r$),  and $\norm{M_f}_2 = \norm{f}_\infty $. Because $\hat{S}^{\frac{1}{2}}$ is selfadjoint, $\hat{S}^{\frac{1}{2}}\hat{S}^{\frac{1}{2}} = \hat{S}$ (see Lemma \ref{lem: properties_reduced_resolvent}) and $\hat{\sigma}_f^2 = - 2 \langle f, Sf \rangle  $ (by definition) we have 
\begin{equation*}
    \norm{\hat{S}^{\frac{1}{2}}f}_2^2 = \langle \hat{S}^{\frac{1}{2}}f,\hat{S}^{\frac{1}{2}}f \rangle = \langle f, \hat{S}f \rangle = \frac{\hat{\sigma}_f^2}{2},
\end{equation*}
which proves the desired bound \eqref{eq: bound_for_each_factor}.
Multiplying these $m$ bounds in \eqref{eq: bound_for_each_factor} yields (recalling \eqref{eq: Tr[k_1,...,k_n]_absolute_value}) 
\begin{align}
\begin{split}
    &|\Tr([k_1, ...,k_{n-1},0])| \leq \left(\frac{\hat{\sigma}_f^2\lambda_1}{2}\right)^m\left(\frac{1}{\lambda_1}\right)^{\sum_{i = 1}^n k_i}\norm{f}_{\infty}^{\sum_{l = 1}^m( j_l - j_{l-1} -2)} \\
    &= \left(\frac{\hat{\sigma}_f^2\lambda_1}{2}\right)^{m-1}\left(\frac{1}{\lambda_1}\right)^{n-1}\norm{f}_{\infty}^{n-2m} \frac{\hat{\sigma}_f^2\lambda_1}{2} = \left(\frac{\hat{\sigma}_f^2 \lambda_1}{2 \norm{f}_\infty^2} \right )^{m-1}\left (\frac{\norm{f}_{\infty}^n}{\lambda_1^{n}} \right) \left( \frac{\hat{\sigma}_f^2\lambda_1^2}{2 \norm{f}_{\infty}^2} \right),
\end{split}
\end{align}
where in the first line we used $k_{j_l} = 0$ for all $l = 1,...,m $ and in the second line we used $j_0 = 0 $ and $j_m = n$.
But by using the Cauchy-Schwartz inequality and $\norm{S} = \frac{1}{\lambda_1}$ (Lemma \ref{lem: properties_reduced_resolvent}$(b)$) it is easily checked that $ \frac{\hat{\sigma}_f^2 \lambda_1}{2 \norm{f}_\infty^2}  \leq 1$, so the above inequality implies the bound
\begin{equation}
\label{eq: bound_Tr[k_1,...,k_n]}
    |\Tr([k_1,...,k_{n-1},0])| \leq \left (\frac{\norm{f}_{\infty}^n}{\lambda_1^{n}} \right) \left( \frac{\hat{\sigma}_f^2\lambda_1^2}{2 \norm{f}_{\infty}^2} \right),
\end{equation}
which is independent of $m$, the number of zeros in $[k_1,...,k_{n-1},0]$.
At last, using the above inequality we can now bound $\abs{\lambda_0^{(n)}}$ by bounding each trace in \eqref{eq: lemma3_lambda_0(n)_2}; 
\begin{equation}
    \lambda_0^{(n)} = (-1)^n\sum_{[k_1,...,k_{n-1},0]}\Tr([k_1,...,k_{n-1}])
\end{equation}

and counting the number of equivalence classes $[k_1,...k_{n-1},0]$ that contribute to the sum in \eqref{eq: lemma3_lambda_0(n)_2}. To count the number of contributing summands we group the equivalence classes with no adjacent zeros according to the number $m$ of zeros. Recall that these are the only equivalence classes that contribute to the sum and that necessarily $m\geq 1$. Let $\beta(n,m)$ be the number of equivalence classes $[k_1,...,k_n]$ with $m$ non-adjacent zeros. Let $\lfloor \cdot \rfloor $ denote the floor function. Then for $ m  > \lfloor \frac{n}{2} \rfloor $  there must exist two adjacent zeros and consequently $\beta(n,m) = 0$. For $1 \leq  m \leq \lfloor \frac{n}{2} \rfloor$ it can be shown  that (Lemma \ref{lem: appendix_beta(n,m)})
\begin{equation}
\label{eq: beta(n,m)}
   \beta(n,m) = \frac{1}{n-1} \binom{n-1}{m} \binom{n-1-m}{n-2m}.
\end{equation}

So the total number $\beta_n$ of summands that contribute in \eqref{eq: lemma3_lambda_0(n)_2} is given by
\begin{equation}
\label{eq: beta_n}
    \beta_n =\sum_{m = 1}^{\lfloor \frac{n}{2} \rfloor} \beta(n,m).
\end{equation}
Applying the bound \eqref{eq: bound_Tr[k_1,...,k_n]} on $\abs{\Tr([k_1,...,k_{n-1},0])}$ to \eqref{eq: lemma3_lambda_0(n)_2} yields
\begin{equation}
    \abs{\lambda_0^{(n)}} \leq  \beta_n\left (\frac{\norm{f}_{\infty}^n}{\lambda_1^{n}} \right) \left( \frac{\hat{\sigma}_f^2\lambda_1^2}{2 \norm{f}_{\infty}^2} \right).
\end{equation}
Let $x := \frac{r \norm{f}_\infty}{\lambda_1}$, then the above inequality implies that (by summing over $n \geq 2$, recall that $\lambda_0^{(1)} = 0$)
\begin{equation}
\label{eq: lezaud_final_bound_lambda_0(r)}
    \lambda_0(r) \leq  \left( \frac{\hat{\sigma}_f^2\lambda_1^2}{2 \norm{f}_{\infty}^2} \right) \sum_{n = 2}^\infty \beta_n x^n,
\end{equation}
which implies the first inequality of  \eqref{eq: lezaud_bound_lambda_0(r)},  because 
\begin{equation}
    \Phi(x) = \sum_{n = 2}^\infty \beta_n x^n
\end{equation}
 for $0 \leq x \leq \frac{1}{3}$ (Lemma \ref{lem: appendix_phi(x)}). Finally, the second inequality of \eqref{eq: lezaud_bound_lambda_0(r)} follows by the fact that
 \begin{equation}
     \Phi(x) \leq \frac{x^2}{1-2x}
 \end{equation}
for $0 \leq x \leq \frac{1}{3}$
which is easy to check by noting that $\Phi(0) = 0$ and calculating that $\Phi'(x) \leq \frac{d}{dx}(\frac{x^2}{1-2x})$ for $0 \leq x \leq \frac{1}{3}$.
\end{proof}
\begin{reference_paper}
\label{rem: lezauds_fault}
\begin{enumerate}[$(a)$]
    \item An elementary calculation using factorials shows that $\beta(n,m)$ (defined in \eqref{eq: beta(n,m)}) can also be written as
\begin{equation}
\label{eq: def_beta(n,m)_lezaud}
    \beta(n,m) = \frac{1}{m} \binom{n-m-1}{m-1} \binom{n-2}{n-m-1},
\end{equation}
which is the definition used by \cite{lezaud}. Thus, we have (c.f. \eqref{eq: beta_n})
\begin{equation}
\label{eq: def_beta_n_lezaud}
    \beta_n = \sum_{i = 1}^{\lfloor \frac{n}{2} \rfloor} \frac{1}{m} \binom{n-m-1}{m-1} \binom{n-2}{n-m-1},
\end{equation}
which is the definition in \cite[P.\,189]{lezaud}. However, in \cite[P.\,189]{lezaud} the sum ends at $\lceil \frac{n}{2} \rceil $ ($\lceil \cdot \rceil $ denotes the ceil function, which is not correct; the sum should end at $\lfloor \frac{n}{2} \rfloor$, because if the number $m$ of  zeros in $[k_1,...k_n]$ is bigger than $\lfloor \frac{n}{2} \rfloor$ then there must exist two adjacent zeros and consequently $\beta(n,m) = 0$.  
\item In contrast to  \cite[Lemma 2.3]{lezaud}, which states that the bound \eqref{eq: lezaud_bound_lambda_0(r)} even holds for $0 \leq r < \frac{\lambda_1}{2 \norm{f}}$, the proof presented in \cite{lezaud} just shows that the bound
    \eqref{eq: lezaud_bound_lambda_0(r)} holds for $0 \leq r < \frac{\lambda_1}{3 \norm{f}}$, because the series  $\sum_{n = 2} \beta_n x^n$ has convergence radius $\frac{1}{3}$ (Lemma \ref{lem: appendix_phi(x)}), and consequently the bound \eqref{eq: lezaud_final_bound_lambda_0(r)}
    \begin{equation}
    \lambda_0(r) \leq  \left( \frac{\hat{\sigma}_f^2\lambda_1^2}{2 \norm{f}_{\infty}^2} \right) \sum_{n = 2}^\infty \beta_n x^n,
\end{equation}
    is trivial for $ \frac{r \norm{f}_\infty}{\lambda_1} = x > \frac{1}{3} $.
\end{enumerate}

\end{reference_paper}

By computing the Fenchel conjugate of the above bound for $\lambda_0$ we get the following concentration inequalities.
\begin{theorem}[Concentration inequality via perturbation theory]
\label{th: concentration_via_perturbation}
Let $u \geq 0$, define 
\begin{equation}
    r_0 := \frac{\lambda_1}{2 \norm{f}_\infty} \left ( 1 -\left ( 1 + \frac{4 u \norm{f}_\infty
    }{\lambda_1 \hat{\sigma}_f^2} \right )^{-\frac{1}{2}} \right ).
\end{equation}
and 
\begin{equation}
    G(r) := \frac{r^2\frac{\hat{\sigma}_f^2}{2}}{1-2\left(\frac{\norm{f}_{\infty}}{\lambda1}\right) r}
\end{equation}
for $0 \leq r < \frac{\lambda_1 }{2 \norm{f}_\infty }$. The following concentration inequalities hold:
\begin{enumerate}[\lb a\rb]
    \item If $r_0 \leq \frac{\lambda_1}{3 \norm{f}_\infty} $, or equivalently $u \leq  \frac{2\hat{\sigma}_f^2\lambda_1}{\norm{f}_\infty}$, then 
    \begin{equation*}
      G^*(u) :=  \sup_{ r \in \left [0, \frac{\lambda_1}{3 \norm{f}_\infty} \right ]} (r u - G(r)) = r_0 u - G(r_0)
    \end{equation*}
    and consequently
    \begin{equation}
    \label{eq: concentration_via_perturbation_a}
        \F{P}_\nu \left (\frac{A_t}{t} \geq u \right ) \leq \norm{\frac{d \nu}{d \pi}}_2 \exp\left(-\frac{2tu^2}{\hat{\sigma}_f^2(1+ \sqrt{1+ \frac{4\norm{f}_\infty u}{\lambda_1 \hat{\sigma}_f^2}})^2} \right)
    \end{equation}
    \item If $r_0  > \frac{\lambda_1}{3 \norm{f}_\infty}$, or equivalently $u >  \frac{2\hat{\sigma}_f^2\lambda_1}{\norm{f}_\infty}$,  then
    \begin{equation*}
       G^*(u) =  \sup_{ r \in \left [0, \frac{\lambda_1}{3 \norm{f}_\infty} \right ]} (r u - G(r)) = \frac{\lambda_1}{3 \norm{f}_\infty} u - G\left (\frac{\lambda_1}{3 \norm{f}_\infty} \right )
    \end{equation*}
    and consequently
    \begin{equation} 
    \label{eq: concentration_via_perturbation_b}
        \F{P}_\nu \left (\frac{A_t}{t} \geq u \right ) \leq  \norm{\frac{d\nu }{ d \pi}}_2 \exp \left(- t \frac{\lambda_1}{3 \norm{f}_\infty} \left (u - \frac{\lambda_1 \hat{\sigma}_f^2}{2 \norm{f}_\infty} \right ) \right )
    \end{equation}
    
\end{enumerate}
\end{theorem}
\begin{reference_paper}
As the bound \eqref{eq: lezaud_bound_lambda_0(r)} of Lemma \ref{lem: lezaud_bound_lambda_0(r)} just was proven for  \\ $0 \leq r < \frac{\lambda_1 }{3 \norm{f}_\infty }$ (and the presented proof just works for this range) we had to make a case distinction (whether $r_0 \leq \frac{\lambda_1}{3 \norm{f}_\infty})$ and obtain a weaker statement (but with a legitimate proof) than stated in \cite[Thrm.\,1.1]{lezaud}.
\end{reference_paper}
\begin{remark}
\begin{enumerate}[$(a)$]
    \item For $\nu = \pi$ Theorem \ref{th: concentration_via_perturbation}$(a)$ is exactly  Bernstein's inequality (Lemma \ref{lem: bernstein_inequality}) for $A_t$, which is sub-gamma on the right tail with variance $t \hat{\sigma}_f^2$ and scale parameter $\frac{2 \norm{f}_\infty}{\lambda_1}$ (Remark \ref{rem: lemma_bound_lammba_(r)_via_perturbation_theory})
    \item We will later show in Theorem \ref{th: extension_lezaud} that Theorem \ref{th: concentration_via_information}$(a)$ actually holds for all $u \geq 0$ and obtain a sharper Bernstein-type  bound.
\end{enumerate}

\end{remark}

\begin{proof}[Proof of Theorem \ref{th: concentration_via_perturbation}]
An elementary calculation shows that we have $r_0 \leq \frac{\lambda_1}{3 \norm{f}_\infty}$ if and only if
    $u \leq  \frac{2\hat{\sigma}_f^2\lambda_1}{\norm{f}_\infty}$.
Recall that $G(r)$ is a bound for $\lambda_0(r)$ (Lemma \ref{lem: lezaud_bound_lambda_0(r)}), i.e. we have 
\begin{equation}
    G(r) \geq \lambda_0(r)
\end{equation}
 for all $r \leq \frac{\lambda_1}{3 \norm{f}_\infty }$. Consequently, by Lemma \ref{lem: concentration_via_bound_on_lambda_0(r)} the concentration inequality 
\begin{equation}
\label{eq: concentration_g^*(u)}
  \F{P}_\nu \left (\frac{A_t}{t} \geq u \right ) \leq  \norm{\frac{d\nu }{ d \pi}}_2 \exp( - G^*(u)).  
\end{equation}
holds for all $u \geq 0$.
To calculate $G^*(u) = \sup_{ r \in \left [0, \frac{\lambda_1}{3\norm{f}_\infty}\right]}(ur - G(r))$, we use Example \ref{ex: Legendre_transform_subgamma} with scale parameter $c = \frac{2 \norm{f}_\infty}{\lambda_1}$ and variance $v = \hat{\sigma}_f^2$. Let $u \geq 0$ and define
\begin{equation}
    H(r) = ur - G(r)
\end{equation}
as in Example \ref{ex: Legendre_transform_subgamma}. By Example \ref{ex: Legendre_transform_subgamma} we have
\begin{equation}
    \sup_{r \in \left [0, \frac{\lambda_1}{2 \norm{f}_\infty}\right)}H(r) = H(r_0).
\end{equation}
If $r_0 \leq \frac{\lambda_1}{3\norm{f}_\infty }$ then clearly
\begin{equation}
    H(r_0) = \sup_{r \in \left [0, \frac{\lambda_1}{2 \norm{f}_\infty}\right)}H(r) =  \sup_{r \in \left [0, \frac{\lambda_1}{3\norm{f}_\infty}\right ]} H(r) = G^*(u).
\end{equation}\
The desired concentration inequality \eqref{eq: concentration_via_perturbation_a} of part $(a)$  follows by  plugging  the formula for $H^*(u) = H(r_0)$ from Example \ref{ex: Legendre_transform_subgamma} (see \eqref{eq: Legendre_transform_subgamma}) into \eqref{eq: concentration_g^*(u)}. Thus, part $(a)$ is proved. Part $(b)$ can be proved as follows. We have $H'(0) = u \geq 0$ and by Example \ref{ex: Legendre_transform_subgamma}, the derivative $H'$ has a unique root in $[0, \frac{\lambda_1}{2 \norm{f}_\infty})$ located at $r_0$. Consequently $H' \geq 0$ in $[0,r_0]$ and $H$ is nondecreasing on $[0,r_0]$. So if $r_0 > \frac{\lambda_1}{3 \norm{f}_\infty}$, then $H$ is nondecreasing on $ \left [0 , \frac{\lambda_1}{3 \norm{f}_\infty} \right]$ and
\begin{equation}
    G^*(u) = \sup_{r \in \left [0 , \frac{\lambda_1}{3 \norm{f}_\infty}\right ]} H(r) = H \left( \frac{\lambda_1}{3 \norm{f}_\infty}\right).
\end{equation}
Finally, the concentration inequality \eqref{eq: concentration_via_perturbation_b} of part $(b)$ follows by explicitly computing $G^*(u) = H\left (\frac{\lambda_1}{3 \norm{f}_\infty} \right )$ and plugging it into \eqref{eq: concentration_g^*(u)}.

\end{proof} 
.

\subsubsection{Concentration Inequalities via functional Inequalities }
\label{subsubsec: concentration_via_functional_inequalities}
In this section we follow \cite{guillin} and bound $\lambda_0(r)$ using a Poincaré and $F$-Sobolev inequality  to obtain concentration inequalities (invoking Lemma \ref{lem: concentration_via_bound_on_lambda_0(r)}) . For a more general presentation on how functional inequalities (like the aforementioned) relate to the theory of Markov processes see e.g. \cite{functional_inequalities}. The concentration inequalities of this section are formulated  in Theorems \ref{th: concentration_via_poincare} and \ref{th: concentration_via_F_sobolev}, which are our versions of \cite[Prop.\,1.4]{guillin} and \cite[Thrm.\,2.3]{guillin}. The most relevant material of Section \ref{sec: preliminaries} is  Lemma \ref{lem: biggest_eigenvalue} and Example \ref{ex: Legendre_transform_subgamma}. \\ We now define the inequalities of interest:
\begin{definition}(Poincaré inequality)
We say that $L$ satisfies a Poincaré inequality with constant $C > 0$ if for all $g \in L^2(\pi)$ (c.f. \cite[Eq.\,(1.3)]{guillin})
\begin{equation}
\label{eq: poincare_inequality}
    \text{Var}_\pi(g)  \leq - C \langle Lg, g \rangle = - C \left \langle \frac{L+L^*}{2}g, g  \right \rangle .  
\end{equation}
Here $\text{Var}_\pi(g) = \pi(g^2) - \pi(g)^2$ denotes the variance of $g$ with respect to $\pi$.
\end{definition}
\begin{definition}($F$-Sobolev inequality)
Let $F : (0, \infty) \rightarrow \F{R}$ be a strictly increasing, continuous, concave function satisfying $F(1) = 0$ and $\lim_{t \to \infty } F(t) = \infty$. Then $F^{-1}$ exists and is defined on $(F(0), \infty)$, where $F(0) := \lim_{t \downarrow 0}F(t)$ (note that this limit exists in $[-\infty, 0)$). In addition, assume that
\begin{equation}
    \label{eq: F-sobolev_multiplicative_condition}
    F(st) \leq F(t) + F(s)
\end{equation}
for all $s, t > 0$. We say $L$ satisfies a $F$-Sobolev inequality (c.f. \cite[Def. 2.1]{guillin}) if 
\begin{equation}
    \label{eq: F-sobolev_inequality}
    \pi( g^2 F(g^2)) \leq  - \langle Lg, g \rangle = - \left \langle \frac{L+L^*}{2}g, g \right \rangle
\end{equation}
for all $g \in L^2(\pi)$ with $\norm{g}_2 = 1$, where $ 0 \cdot F(0) = 0$.

\end{definition}
\begin{remark}
\begin{enumerate}[$(a)$]
    \item (Poincaré inequality and spectral gap) By writing $g = h + a\textbf{1}$ with $h \in \{\textbf{1}\}^{\perp}$, using $\text{Var}_\pi(g) = \text{Var}_\pi(h) = \norm{h}_2^2$, and Lemma \ref{lem: properties_infinitesimal_generator}$(a)$ it is easily seen that the Poincaré inequality holds if and only if
\begin{equation}
\label{eq: poincare_1}
    \norm{h}_2^2 \leq - C \left \langle \frac{L+L^*}{2}h, h \right \rangle 
\end{equation}
for all $h \in \{\textbf{1}\}^{\perp}$. But by diagonalizing $\frac{L+L^*}{2}$ and using Lemma \ref{lem: properties_infinitesimal_generator}$(a),(c)$ it is straightforward to check that \eqref{eq: poincare_1} is equivalent to 
\begin{equation}
\label{eq: spectral_gap_inequality}
    C \geq \frac{1}{\lambda_1}
\end{equation}
where $\lambda_1 = \min \limits_{\lambda \in \sigma \left( \frac{L+L^*}{2} \right) \backslash \{0\}} \abs{\lambda}$ is the spectral gap. Thus, the Poincaré inequality with constant $C$ is equivalent to the above  inequality for the spectral gap.

\item (Log-Sobolev inequality) If $F = C \log$ for some $C > 0 $, the $F$-Sobolev inequality is called log-Sobolev inequality with constant $\frac{1}{C}$ (c.f. \cite[Ch.\,5.1, Eq.\,5.1.1]{functional_inequalities}).
\end{enumerate}
\end{remark}

Using  the above inequalities we can derive upper bounds for $\lambda_0(r)$ to then apply Lemma \ref{lem: concentration_via_bound_on_lambda_0(r)} and obtain concentration inequalities. The following two bounds hold:

\begin{lemma}[Bound for $\lambda_0(r)$ via Poincaré inequality] Denote by   $\lambda_1$ the spectral gap of $\frac{L+L^*}{2}$ and let $C \geq \frac{1}{\lambda_1}$, i.e. $L$ satisfies a Poincaré inequality with constant $C$. Then,
\label{lem: poincare_bound_lambda_0(r)}
for all $0 \leq r < \frac{1}{C \norm{f}_\infty}$ 
\begin{equation}
\label{eq: poincare_bound_lambda_0(r)}
    \lambda_0(r) \leq \frac{r^2\mathrm{Var}_\pi(f)C}{1 - r C\norm{f}_\infty}.
\end{equation}.
\end{lemma}
\begin{reference_paper}
The above bound is found in \cite[P.\,14]{guillin}. There this bound is stated as 
    \begin{equation}
    \label{eq: guillin_bound_lambda_0(r)}
        \Lambda(\lambda V) \leq \frac{\lambda^2\text{Var}_\mu(V)}{\frac{1}{C_P} - \lambda} = \frac{\lambda^2C_P\text{Var}_\mu(V)}{1 - C_P\lambda},
    \end{equation}
    where (see \cite[Thrm.\,1.1]{guillin}) $\mu$ denotes the invariant measure,
    \begin{equation}
        \Lambda(\lambda V) = \sup \{\langle Lg,g \rangle + \lambda \langle V, g^2 \rangle \, |\, \norm{g}_2 =  1, g \in D(L) \},
    \end{equation}
    and $C_P$ is the constant in the Poincaré inequality \cite[Eq.\,(1.3)]{guillin}. Furthermore, \cite{guillin} assumes that $\norm{V}_\infty = 1 $. In our case $V = f$, $\mu = \pi$, $\lambda = r$, $C_P = C$ and consequently (as $\norm{f}_\infty = 1$ is assumed)  \eqref{eq: guillin_bound_lambda_0(r)} is exactly the bound of Lemma \ref{lem: poincare_bound_lambda_0(r)}.
\end{reference_paper}
\begin{remark}
\label{rem: poincare_sharp_bound_lambda}
\begin{enumerate}[$(a)$]
    \item Note the similarity of this bound to the bound obtained using perturbation theory in Lemma \ref{lem: lezaud_bound_lambda_0(r)} . In particular, this bound implies (as in Remark \ref{rem: lemma_bound_lammba_(r)_via_perturbation_theory}) that for $\nu = \pi$ $A_t$ is sub-gamma (on the right tail) with variance $2t C \text{Var}_\pi(f)$ and scale parameter $\frac{1}{C\norm{f}_\infty}$.
    \item The best bound (i.e. the sharpest) is obtained for $C = \frac{1}{\lambda_1}$, consistent with the fact that  the Poincaré inequality is also the sharpest for $C = \frac{1}{\lambda_1}$. This follows because if $C_1 \leq C_2$, then
    \begin{equation}
        \frac{r^2\mathrm{Var}_\pi(f)C_1}{1 - r C_1\norm{f}_\infty} \leq \frac{r^2\mathrm{Var}_\pi(f)C_2}{1 - r C_2\norm{f}_\infty}
    \end{equation}
    for all $ 0 \leq r < \frac{1}{C_2 \norm{f}_\infty}$.
\end{enumerate}
\end{remark}
\begin{proof}[Proof of Lemma \ref{lem: poincare_bound_lambda_0(r)}]
To bound $\lambda_0(r)$
we follow the proof of \cite[Prop.\,1.4]{guillin} but modify it to make it applicable to a general $f \in L^2(\pi)$ with $\pi(f) = 0$ (\cite[Prop.\,1.4]{guillin} assumes that $\norm{f}_\infty = 1)$. Let $K := \frac{1}{C}$, then by applying the Poincaré inequality \eqref{eq: poincare_inequality} to Lemma \ref{lem: biggest_eigenvalue} we obtain
\begin{align}
\begin{split}
\label{eq: bound_lambda_0(r)_with_poincare}
    \lambda_0(r) &\leq \sup\{ \langle rM_fg,g \rangle - K \text{Var}_\pi(g) \;| \; g \in L^2(\pi) , \norm{g} = 1 \} \\
    &= \sup  \left \{ \frac{1}{1+\varepsilon^2} \langle rf, \textbf{1} + 2 \varepsilon h + \varepsilon^2 h^2 \rangle - K \frac{\varepsilon^2}{1+ \varepsilon^2}  \;| \; h \in \{ \textbf{1} \}^{\perp}, \norm{h} = 1, \varepsilon \geq 0 \right \} \\
    &= \sup  \left \{ \frac{\varepsilon}{1+\varepsilon^2}( 2 r \langle f, h \rangle + \varepsilon \langle rf - K \textbf{1}, h^2 \rangle ) \;| \; h \in \{ \textbf{1} \}^{\perp}, \norm{h} = 1, \varepsilon \geq 0 \right \} \\
    &\leq \sup \{ \varepsilon ( 2r \norm{f}_2 + \varepsilon ( r \norm{f}_\infty - K) \;| \; \varepsilon \geq 0 \}.
\end{split}
\end{align}
In the second line we wrote $g = \frac{\textbf{1}+ \varepsilon h}{\sqrt{1+\varepsilon^2}}$ with $h \in \{\textbf{1} \}^{\perp}, \norm{h}_2 = 1$ and used $\text{Var}_\pi(g) = \frac{\varepsilon^2}{1+ \varepsilon^2} $. In the third line we used $\langle f, \textbf{1} \rangle = 0$ and $ \langle \textbf{1}, h^2 \rangle = 1$. Finally, in the last line we used the Cauchy-Schwartz inequality  $\langle f, h \rangle \leq \norm{f}_2 \cdot 1$ and the estimates $\frac{1}{1+\varepsilon^2} \leq 1$, and
\begin{equation*}
    \langle rf - K \textbf{1}, h^2 \rangle  = \pi((rf - K \textbf{1})h^2) \leq \pi((r\norm{f}_\infty - K)h^2) = r\norm{f}_\infty - K.
\end{equation*}
Notice that if $r \geq \frac{K}{\norm{f}_\infty}$, then \eqref{eq: bound_lambda_0(r)_with_poincare} is the trivial estimate $\lambda_0(r) \leq \infty $. So assume that $r < \frac{K}{\norm{f}_\infty}$. Then, to calculate the supremum in \eqref{eq: bound_lambda_0(r)_with_poincare}  note that (by an elementary calculation)  the quadratic polynomial 
\begin{equation*}
   P(\varepsilon) = \varepsilon^2(r \norm{f}_\infty -K) + \varepsilon 2 r \norm{f}_2
\end{equation*}
reaches a maximum at $\varepsilon_0 = \frac{r \norm{f}_2}{K- r \norm{f}_\infty }$. Finally, the desired bound on $\lambda_0(r)$ \eqref{eq: poincare_bound_lambda_0(r)} follows by \eqref{eq: bound_lambda_0(r)_with_poincare}, as 
\begin{equation}
\label{eq: maximum_quadratic_polynomial_poincare}
   \lambda_0(r) \leq P(\varepsilon_0) =  P \left( \frac{r \norm{f}_2}{K- r \norm{f}_\infty }\right) = \frac{r^2 \norm{f}_2^2}{K- r \norm{f}_\infty } = \frac{r^2\frac{\text{Var}_\pi(f)}{K}}{1 - r \frac{\norm{f}_\infty}{K}}.
\end{equation}

\end{proof}
\begin{lemma}[Bound for $\lambda_0(r)$ via $F$-sobolev inequality]
\label{lem: bound_lambda_0(r)_F-sobolev}
Suppose $L$ satisfies an F-Sobolev inequality.
Let $r_f := \frac{F(0)}{\min_{x \in E}f(x)}$, so that $F^{-1}(rf)$ is defined for all $0 \leq r < r_f $. Then for all $0 \leq r < r_f$ we have
\begin{equation}
\label{eq: bound_lambda_0(r)_F-sobolev}
    \lambda_0(r) \leq F( \pi( F^{-1}( rf))). 
\end{equation}

\end{lemma}

\begin{reference_paper}
The above bound is stated in \cite[P.\,17]{guillin} as 
    \begin{equation}
    \label{eq: guillin_bound_F_sobolev}
        \Lambda(\lambda V) \leq F \left ( \int F^{-1}(\lambda V) d\mu \right ).
    \end{equation}
    Recall (see Remark \ref{rem: poincare_sharp_bound_lambda}) $ V = f, \lambda = r,  \Lambda(\lambda V) = \lambda_0(r)$ so \label{eq: guillin_bound_F_sobolev} this is exactly the bound of Lemma \ref{lem: bound_lambda_0(r)_F-sobolev}.
\end{reference_paper}
\begin{remark}
     Notice that $r_f > 0$ because $F(0) < 0 $ and $ \min_{x \in E}f(x) < 0$. Indeed, $F$ is strictly increasing and $F(1) = 0$, so $F(0) = \lim_{ t \downarrow 0}F(t)  < 0$. Furthermore, $\sum_{x \in E }f(x) \pi_x = 0$ and $f$ is not constant, so $\min_{x \in E}f(x) < 0$.
\end{remark}
\begin{proof}[Proof of Lemma \ref{lem: bound_lambda_0(r)_F-sobolev}]
We follow the proof of \cite[Thrm.\,2.3]{guillin} and add computational details. Use  Lemma \ref{lem: biggest_eigenvalue} and the $F$-Sobolev inequality \eqref{eq: F-sobolev_inequality} to obtain  
\begin{align}
\begin{split}
\label{eq: bound_lambda_0(r)_F_sobolev_prototype}
    & \lambda_0(r) \leq \sup \{ \langle rM_fg, g \rangle - \pi( g^2 F(g^2)) \; | \; \norm{g}_2 = 1 \}\\
    &= \sup \{ \pi((rf - F(g^2))g^2) \; | \; \norm{g}_2 = 1 \}.
\end{split}
\end{align}
The desired bound on $\lambda_0(r)$ \eqref{eq: bound_lambda_0(r)_F-sobolev} follows now from the above equation and the following estimate:
If $\norm{g}_2 = 1$ and $0 \leq r < r_f$, then
\begin{align}
    \begin{split}
    \label{eq: estimate_F_sobolev}
       & \pi((rf- F(g^2))g^21_{\{ g \neq 0 \}}) \leq \pi \left ( F \left (\frac{F^{-1}(rf)}{g^2} \right) g^2 1_{\{ g \neq 0 \}} \right )  \\
       &\leq F \left ( \pi \left ( \frac{F^{-1}(rf)}{g^2} g^2 1_{\{ g \neq 0 \}} \right) \right ) = F ( \pi( F^{-1}(rf) 1_{\{ g \neq 0 \}})) \\
       &\leq F( \pi( F^{-1}(rf)),
    \end{split}
\end{align}
where the first line we used the inequality 
\begin{equation*}
    (rf - F(g^2))1_{\{ g \neq 0 \}} \leq F \left (\frac{F^{-1}(rf)}{g^2} \right) 1_{\{ g \neq 0 \}},
\end{equation*}
which follows from \eqref{eq: F-sobolev_multiplicative_condition}
\begin{equation}
    F(st) \leq F(s) + F(t)
\end{equation}
with $t = \frac{F^{-1}(rf)}{g^2} $ and $s = g^2$. In the second line we used Jensen's inequality with the concave function $F$ and with respect to integration with the probability measure $g^2 \pi$. In the last line the monotonicity of $F$ and $F^{-1}(rf) > 0$ were used. We conclude by plugging the estimate \eqref{eq: estimate_F_sobolev} into \eqref{eq: bound_lambda_0(r)_F_sobolev_prototype}.
\end{proof}

Similarly as in Section \ref{subsubsec: concentration_via_perturbation} we use the above bounds for $\lambda_0(r)$ to obtain the following concentration inequalities.

\begin{theorem}[Concentration inequality via Poincaré inequality]
\label{th: concentration_via_poincare}
Let $\tilde{\sigma}_f^2 := 2 \frac{Var_\pi(f)}{\lambda_1 }$ In our setting the following concentration inequality holds:
\begin{equation}
\label{eq: concentration_via_poincare}
        \F{P}_\nu \left (\frac{A_t}{t} \geq u \right ) \leq \norm{\frac{d \nu}{d \pi}}_2 \exp\left(-\frac{2tu^2}{\tilde{\sigma}_f^2(1+ \sqrt{1+ \frac{2\norm{f}_\infty}{\lambda_1 \tilde{\sigma}_f^2}})^2} \right).
    \end{equation}
\end{theorem}
\begin{reference_paper}
The above result is a generalization and reformulation of \cite[Prop.\,1.4]{guillin}: The proposition \cite[Prop.\,1.4]{guillin} states that  (\cite{guillin} used the notation $V = f$, $R = u$, $\mu = \pi$, $C_P = C$)
    \begin{equation}
    \label{eq: guillin_concentration_inequality_poincare}
        \F{P}_\nu \left (\frac{A_t}{t} \geq u \right ) \leq \norm{\frac{d\nu}{d\pi}}_2 e^{-tA(u)}
    \end{equation}
    for $\norm{f}_\infty = 1$ and $0 < u \leq 1$,
    where 
    \begin{equation}
        A(u) = \frac{1}{C} \left ( 1 - \sqrt{1- \frac{u}{u+ \text{Var}_\pi(f)}} \right ) \left( u - \frac{\text{Var}_\pi(f) \left (1- \sqrt{1- \frac{u}{u + \text{Var}_\pi(f)}} \right)}{\sqrt{1- \frac{u}{u+ \text{Var}_\pi(f)}}} \right)
    \end{equation}
    and  $C = \frac{1}{\lambda_1}$ (we choose $C$ to be the optimal constant for the Poincaré inequality). As the above bound \eqref{eq: guillin_concentration_inequality_poincare} is quite intricate it is not directly clear how Theorem \ref{th: concentration_via_poincare} implies \cite[Prop.\,1.4]{guillin} as a special case. However, this follows from the proof (found in \cite[P.\,14]{guillin}) of \cite[Prop.\,1.4]{guillin}: We have  \\(see \cite[P.\,14]{guillin}: 'The supremum is attained for ... and is equal to $A$') 
    \begin{equation}
        A(u) = \sup_{r \in [0, \frac{1}{C})} \left (ur - \frac{r^2 \text{Var}_\pi(f)}{\frac{1}{C}- r}\right ) = \sup_{r \in [0, \frac{1}{C})}\left ( ur - \frac{r^2 C\text{Var}_\pi(f)}{1- Cr} \right) ,
    \end{equation}
    where $C = \frac{1}{\lambda_1}$. This implies (by using Example \ref{ex: Legendre_transform_subgamma} with variance $v = \tilde{\sigma}_f^2$ and scale parameter $c =  \frac{1}{\lambda_1}$) that
    \begin{equation}
        A(u) = -\frac{2u^2}{\tilde{\sigma}_f^2(1+ \sqrt{1+ \frac{2}{\lambda_1 \tilde{\sigma}_f^2}})^2}.
    \end{equation}
    Thus, \cite[Prop.\,1.4]{guillin} is a special case of Theorem \ref{th: concentration_via_poincare}.
\end{reference_paper}
\begin{remark}
\label{rem : concentration_via_poincare}
\begin{enumerate}[$(a)$]
    \item For $\nu = \pi$ Theorem \ref{th: concentration_via_poincare} is exactly  Bernstein's inequality (c.f. Lemma \ref{lem: bernstein_inequality}) for $A_t$, which is sub-gamma on the right tail with variance $t \Tilde{\sigma}_f^2$ and scale parameter $\frac{ \norm{f}_\infty}{\lambda_1}$ (Remark \ref{rem: poincare_sharp_bound_lambda})
    \item We will later obtain in Theorem \ref{th: extension_lezaud} a sharper Bernstein-type  bound.
    \item The used constant for the Poincaré inequality is $C = \frac{1}{\lambda_1}$, as this yields the sharpest bound (c.f. Remark \ref{rem: poincare_sharp_bound_lambda})
    
\end{enumerate}

\end{remark}
\begin{proof}[Proof of Theorem \ref{th: concentration_via_poincare}]
The proof is similar to the proof of Theorem \ref{th: concentration_via_perturbation}. Let
\begin{equation}
   G(r) := \frac{r^2\frac{\text{Var}_\pi(f)}{\lambda_1}}{1 - r \frac{\norm{f}_\infty}{\lambda_1}}   =\frac{r^2\tilde{\sigma}_f^2}{2(1 - r \frac{\norm{f}_\infty}{\lambda_1})}  
\end{equation}
 for $r \in \left [ 0, \frac{\lambda_1}{\norm{f}_\infty} \right)$ be the bound for $\lambda_0(r)$ of Lemma \ref{lem: poincare_bound_lambda_0(r)} for $C = \frac{1}{\lambda_1}$. Then, invoking Lemma \ref{lem: concentration_via_bound_on_lambda_0(r)} yields 
 \begin{equation}
      \F{P}_\nu \left (\frac{A_t}{t} \geq u \right ) \leq  \norm{\frac{d\nu }{ d \pi}}_2 \exp( - tG^*(u)),
 \end{equation}
 where the Fenchel conjugate is taken with respect to $\left [0, \frac{\lambda_1}{\norm{f}_\infty}\right)$.
The claim now follows by computing $G^*(u)  $ using Example \ref{ex: Legendre_transform_subgamma} with variance $v = \Tilde{\sigma}_f^2$ and scale factor $c = \frac{\norm{f}_\infty}{\lambda_1}$.
\end{proof}

\begin{theorem}[Concentration inequality via $F$-Sobolev inequality]
\label{th: concentration_via_F_sobolev}
Suppose that the MJP satisfies an $F$-Sobolev inequality. Then,
\begin{equation}
\F{P}_\nu \left (\frac{A_t}{t} \geq u \right ) \leq \norm{\frac{d \nu}{d \pi}}_2 \exp( - t \sup_{ r \in [0, r_f)} (ru -F(\pi( F^{-1}(rf))))) 
\end{equation}

\end{theorem}
\begin{reference_paper}
The corresponding Theorem in \cite{guillin} is the second part of \cite[Thrm.\,2.3]{guillin}. Using the notation of \cite{guillin} we have
\begin{equation}
    H_c^*(a) = \sup_{ r \in [0, r_f)} (ra -F(\pi( F^{-1}(rf)))).
\end{equation}
\end{reference_paper}
\begin{remark}
\begin{enumerate}[$(a)$]
    \item Unlike for the Poincaré inequality, which is satisfied in our setting for $C \geq \frac{1}{\lambda_1}$ , for a general $F$ the $F$-Sobolev inequality may not be satisfied, thus we need the extra assumption.
    \item If $\nu = \pi$ and $F = C \log $ (for $C > 0$), we get a continuous time analogue of Chernoff's inequality for sum of i.i.d random variables (see Corollary \ref{cor: cramer_chernoff_sum_iid}). Indeed, an elementary calculation shows that
    \begin{align}
    \begin{split}
        &\sup_{ r \in [0, r_f)} (ru -F(\pi( F^{-1}(rf)))) = C \sup_{r \geq 0}(ru - \log \pi(e^{rf})) \\
        &=  C \sup_{r \geq 0}(ru - \log \F{E}_\pi(e^{rf(X_0)})) =  C \Psi_{f(X_0)}^*(u),
        \end{split}
    \end{align}
    consequently 
    \begin{equation}
        \F{P}_\nu \left (\frac{A_t}{t} \geq u \right ) \leq e^{-tC \Psi_{f(X_0)}^*(u) }
    \end{equation}
\end{enumerate}

\end{remark}
\begin{proof}[Proof of Theorem \ref{th: concentration_via_F_sobolev}]
Follows immediately by recalling (see Lemma \ref{lem: bound_lambda_0(r)_F-sobolev}). 
\begin{equation}
    \lambda_0(r) \leq F(\pi( F^{-1}(rf)))
\end{equation}
for all $ r \in [0, r_f)$ and then applying Lemma \ref{lem: concentration_via_bound_on_lambda_0(r)} with $G(r) = F(\pi( F^{-1}(rf))) $.

\end{proof}

\subsubsection{Concentration Inequalities via Information Inequalities}
\label{subsubsec: concentration_via_information}
In this section we follow \cite{bernstein} and use a inequality for the so called \textit{Donsker-Varadhan information} to  derive a concentration inequality. Finally, we use this inequality to extend Theorem \ref{th: concentration_via_perturbation}, Lemma \ref{lem: lezaud_bound_lambda_0(r)}, and sharpening the Bernstein-type bounds Theorem \ref{th: concentration_via_perturbation}$(a)$ and Theorem \ref{th: concentration_via_poincare}. The main results of this section are Theorems \ref{th: concentration_via_information} and  \ref{th: extension_lezaud}, which are our versions of \cite[Thrm.\,2.2]{bernstein} and \cite[Thrm.\,1.2]{bernstein}. The most relevant background knowledge for this section is Theorem \ref{th: main_conc_inequality}, Example \ref{ex: Legendre_transform_subgamma} and the notion of reduced resolvent (Section \ref{subsubsec: perturbation_theory}). \\ 
Denote by $\mathcal{M}_1(E)$ the set of all probability measures on $E$. The  Donsker-Varadhan information $ I(\beta | \mu)$ for $\beta,\mu \in \mathcal{M}_1(E)$ is defined as (c.f. \cite[Eq.\,(2.2)]{bernstein}) 

\begin{equation}
\label{eq: def_donsker_varadhan_info}
    I(\beta | \mu) := 
    \begin{cases}
    - \left \langle L \left(\frac{d\beta}{d\mu} \right)^{\frac{1}{2}}, \left (\frac{d\beta}{d\mu} \right)^{\frac{1}{2}} \right\rangle \; &; \; \beta << \mu \\
     \infty \; &; \; \text{otherwise}, \\
    \end{cases}
\end{equation}
where $\beta << \mu$ means that $\beta$ is absolutely continuous with respect to $\mu$.
\begin{remark}
In our case, where $\pi$ is the invariant measure of an irreducible MJP
 we always have $I(\nu | \pi) < \infty $ for all $\nu \in \mathcal{M}_1(E)$, because the Radon-Nikodym derivative $\frac{d\nu}{d\pi}(x) = \frac{\nu_x}{\pi_x}$ always exists as $\pi_x > 0$ for all $x \in E$ (by Theorem \ref{th: uniqueness_existence_invariant_measure}).
\end{remark}
We start by proving
\begin{lemma}
\label{lem: information_Legendre_transform_lambda_0(r)}
Let $I$ be defined as in Theorem \ref{th: main_conc_inequality}, i.e. 
\begin{equation}
    I(u) = \inf \{ - \langle Lg, g \rangle \; | \, \norm{g}_2 = 1, \langle M_fg, g \rangle = u \}.
\end{equation}
Then, 
\begin{equation}
    I(u) = \inf \{ I(\beta | \pi) \; | \; \beta \in \mathcal{M}_1(E), \beta(f) = u \}
\end{equation}
for all $u \in \F{R}$.
\end{lemma}
\begin{reference_paper}
The above lemma is not directly stated in \cite{bernstein} but implicitly used as follows. To prove \cite[Thrm.\,2.2]{bernstein}, \cite{bernstein} uses \cite[Thrm.\,1]{wu} (which is Theorem \ref{th: main_conc_inequality}), which states (in our case)
\begin{equation}
\label{eq: repetition_wu_theorem1}
 \F{P}_\nu \left (\frac{A_t}{t} > u \right ) \leq \norm{\frac{d\nu}{d \pi}}_2 e^{- t I(u)}.
 \end{equation}
However, this theorem is formulated in \cite[Thrm.\,2.1]{bernstein} as (\cite{bernstein} used the notation $g = f, \beta = \nu , \mu = \pi, r = u$)
\begin{equation}
\label{eq: bernstein_theorem_2.1}
    \F{P}_\nu \left (\frac{A_t}{t} > u \right ) \leq \norm{\frac{d\nu}{d \pi}}_2 e^{- t \tilde{I}(u-)} 
\end{equation}
for $ t,u > 0$, where  $\tilde{I}(u) = \inf \{ I(\beta | \pi) \; | \; \beta \in \mathcal{M}_1(E), \beta(f) = u \}$ and $\tilde{I}(r-) = \lim_{\varepsilon \downarrow 0}I(r-\varepsilon) $. The authors do not explain further how \eqref{eq: repetition_wu_theorem1} implies \eqref{eq: bernstein_theorem_2.1} and it is not directly clear that $I(u) = \tilde{I}(u-)$. Thus, we decided to formulate and prove the above lemma. 
\end{reference_paper}

\begin{proof}[Proof of Lemma \ref{lem: information_Legendre_transform_lambda_0(r)}]
Our starting point is \eqref{eq: fenchel_dual_lambda_0(r)} (see Theorem \ref{th: main_conc_inequality}): 
\begin{equation}
\label{eq: infimum_over_smaller_set_1}
    \lambda_0^*(u) = \sup_{ r \in \F{R}}(r u - \lambda_0(r)) = \inf \{ - \langle Lg, g \rangle \; | \, \norm{g}_2 = 1, \langle M_fg, g \rangle = u \}.
\end{equation}
Furthermore, the following inequality holds for all $g \in L^2(\pi)$:
\begin{equation}
\label{eq: infimum_over_smaller_set_2}
    - \langle L g, g \rangle \geq -  \langle L \abs{g}, \abs{g} \rangle.
\end{equation}
This can be seen as follows.
Using the definition of the contraction semigroup $(P_t)_{t \geq 0}$ of the MJP (see \eqref{eq: contraction_semigroup_Markov_process}), it follows that 
\begin{equation}
\label{eq: inequality_semigroup}
    P_t\abs{g} \geq \abs{P_tg}
\end{equation}
for all $t \geq 0$.
Consequently, for all $ t \geq 0$ 
\begin{equation*}
    \langle P_tg, g \rangle  = \pi(P_tg g) \leq  \pi(P_t\abs{g}\abs{g}) = \langle P_t\abs{g}, \abs{g} \rangle,
\end{equation*}
with equality for $t = 0$. Thus, the above inequality \eqref{eq: inequality_semigroup} extends to the derivative at  $t = 0$, i.e.
\begin{equation*}
    \langle Lg, g \rangle \leq \langle L\abs{g}, \abs{g} \rangle,
\end{equation*}
which implies \eqref{eq: infimum_over_smaller_set_2}. Notice that $\norm{g}_2 = \norm{\abs{g}}_2$ and $\langle M_fg,g \rangle = \langle M_f \abs{g}, \abs{g} \rangle $. Combining these facts with
\eqref{eq: infimum_over_smaller_set_1} and \eqref{eq: infimum_over_smaller_set_2} yields
\begin{align}
\begin{split}
     \lambda_0^*(u) &= \inf \{ - \langle Lg, g \rangle \; | \, \norm{g}_2 = 1, \langle M_fg, g \rangle = u \} \\
    &= \inf \{ - \langle L g, g \rangle \; | \; \norm{g}_2 = 1, \langle M_fg,g \rangle = u, g \geq 0 \} .
\end{split}
\end{align}
To complete the proof note that there is a one to one correspondence of \\ $\{ \beta \in \mathcal{M}_1(E) \; | \; \beta(f) = u \}$ with $ \{ g \in L^2(\pi) \; | \; \norm{g}_2 = 1, \langle M_fg,g \rangle = u, g \geq 0 \} $ via the two functions
\begin{equation*}
 \mathcal{M}_1(E)\rightarrow L^2(\pi); \; \beta \mapsto \left ( \frac{d\beta}{d\pi} \right)^{\frac{1}{2}} \; \text{and} \;  L^2(\pi) \rightarrow \mathcal{M}_1(E); \; g \mapsto g^2\pi  .
\end{equation*}
Furthermore, by the definition of the Donsker-Varadhan information  for the corresponding $\beta$ and $g$ by we have
\begin{equation*}
    I(\beta | \pi) = - \langle L g, g \rangle.
\end{equation*}
Finally, using  this correspondence we get
\begin{align}
\begin{split}
   & \inf \{ I(\beta | \pi) \; | \; \beta \in \mathcal{M}_1(E), \beta(f) = u \} \\
   &= \inf \{ - \langle L g, g \rangle \; | \; \norm{g}_2 = 1, \langle M_fg,g \rangle = u, g \geq 0 \} \\
    &= \lambda_0^*(u).
\end{split}
\end{align}
\end{proof}
Now, using the above expression for $\lambda_0^*(u)$ 
\begin{equation}
     \lambda_0^*(u) = \inf \{ I(\beta | \pi) \; | \; \beta \in \mathcal{M}_1(E), \beta(f) = u \},
\end{equation}
and assuming additionally an inequality involving the Donsker-Varadhan information one
can derive concentration inequalities using the following theorem.
\begin{theorem}[Concentration inequality via information inequalities]
\label{th: concentration_via_information}
Let $\alpha : [0, \infty) \rightarrow [0, \infty]$ be some function satisfying 
\begin{equation}
\label{eq: condition_concentration_via_information}
    \alpha(\beta(f)) \leq I (\beta | \pi)
\end{equation}
for all $\beta \in \mathcal{M}_1(E)$ such that $\beta(f) \geq 0$.
Then,
\begin{equation}
\label{eq: concenration_inequality_information}
    \F{P}_\nu \left( \frac{A_t}{t} \geq u \right ) \leq \norm{\frac{d\nu}{d\pi}}_2 e^{-t \alpha(u)}.
\end{equation}
Furthermore,
\begin{equation}
\label{eq: bound_lambda_0(r)_information}
    \lambda_0(r) \leq \alpha^*(r)
\end{equation}
for all $r \geq 0$, where $\alpha^*$ denotes the Fenchel conjugate of $\alpha$ with respect to $\F{R}_{ \geq 0}$.
\end{theorem}
\begin{reference_paper}
A similar theorem as Theorem \ref{th: concentration_via_information} is stated  in \cite[Thrm.\,2.2]{bernstein}. Hereby, statements \cite[Thrm.\,2.2]{bernstein}$(a),(c),(e)$ correspond to \eqref{eq: condition_concentration_via_information}, \eqref{eq: concenration_inequality_information},\eqref{eq: bound_lambda_0(r)_information}. However, we reformulated \cite[Thrm.\,2.2]{bernstein} as follows:  \cite[Thrm.\,2.2]{bernstein} considers a symmetric Markov process, assumes left continuity and convexity of $\alpha$, and states equivalence of \eqref{eq: condition_concentration_via_information}, \eqref{eq: concenration_inequality_information} and \eqref{eq: bound_lambda_0(r)_information}. For our purposes just the implications $\eqref{eq: condition_concentration_via_information} \Rightarrow \\ \eqref{eq: concenration_inequality_information}, \eqref{eq: bound_lambda_0(r)_information}$ are of interest as we use Theorem \ref{th: concentration_via_information} to obtain a Bernstein-type  bound (c.f. Theorem \ref{th: extension_lezaud}). Furthermore, the proof of the implication $\eqref{eq: concenration_inequality_information} \Rightarrow \eqref{eq: condition_concentration_via_information}$ uses large deviation theory for symmetric Markov processes (c.f. \cite[Proof of Thrm.\,2.2]{Guillin2009}; cited in \cite[Thrm.\,2.2]{bernstein}), so we can not generalize directly this proof to the general, non-symmetric case. Furthermore, for the implications of interests we do not  need to assume that $\alpha$ is convex and left continuous. 
\end{reference_paper}
\begin{proof}[Proof of Theorem \ref{th: concentration_via_information}]
Combining the assumption \eqref{eq: condition_concentration_via_information} 
\begin{equation}
    \alpha(\beta(f)) \leq I (\beta | \pi)
\end{equation}
for all $\beta \in \mathcal{M}_1(E)$ such that $\beta(f) \geq 0$
with Lemma \ref{lem: information_Legendre_transform_lambda_0(r)}, which states that
\begin{equation}
    \lambda_0^*(u) = \inf \{ I(\beta | \pi) \; | \; \beta \in \mathcal{M}_1(E), \beta(f) = u \},
\end{equation}
yields  $\alpha(u) \leq \lambda_0^*(u)$ for all $u \geq 0$. Thus, the concentration inequality (Theorem \ref{th: main_conc_inequality})
\begin{equation}
    \F{P}_\nu \left (\frac{A_t}{t} \geq u \right ) \leq \norm{\frac{d\nu}{d \pi}}_2 e^{- t \lambda_0^*(u)}.
\end{equation}
implies 
\begin{equation}
     \F{P}_\nu \left( \frac{A_t}{t} \geq u \right ) \leq \norm{\frac{d\nu}{d\pi}}_2 e^{-t \alpha(u)}
\end{equation}
for all $u \geq 0 $.
To prove the second statement; $\lambda_0(r) \leq \alpha^*(r)$ for all $r \geq 0$, recall that (Remark \ref{rem: th_main_concentration_inequality}$(a)$)
 $\lambda_0^*(u) \geq 0$ for all $u \in \F{R}$. Furthermore, by the proof of Theorem \ref{th: main_conc_inequality} (see \eqref{eq: lambda_0(r)_is_Legendre_transform_of_I(u)} and use $I(u) = \lambda^*_0(u)$) we have 
\begin{equation*}
    \lambda_0(r) = \sup_{u \in \F{R}} (ur - \lambda_0^*(u))
\end{equation*}
for all $r \in \F{R}$. Now, if $r \geq 0$, then, because $\lambda_0^*(u) \geq 0$, it follows that the above supremum can be taken over $u \in \F{R}_{\geq 0}$, i.e.
\begin{equation}
    \lambda_0(r) = \sup_{u \in \F{R}_{\geq 0}} (ur - \lambda_0^*(u))
\end{equation}
for $r \geq 0$. Thus, the inequality $\alpha(u) \leq \lambda_0^*(u)$ for $u \geq 0 $ implies   (c.f. Lemma \ref{lem: properties_Legendre_transform}$(d)$)
\begin{equation}
    \lambda_0(r) \leq \alpha^*(r)
\end{equation}
for all $r \geq 0$.
\end{proof}
Now, using  the above theorem  we can extend Theorem \ref{th: concentration_via_perturbation}, Lemma \ref{lem: lezaud_bound_lambda_0(r)} and Theorem \ref{th: concentration_via_poincare} by choosing 
\begin{equation*}
    \alpha(u) = -\frac{2u^2}{\hat{\sigma}_f^2 \left (1+ \sqrt{1+ \frac{2\norm{f^+}_\infty u}{\lambda_1\hat{\sigma}_f^2}}\right)^2},
\end{equation*}
where $\Hat{\sigma}_f^2 = - \langle Sf, f \rangle$ and $\lambda_1 = \min_{\lambda \in \sigma \left ( \frac{L+L^*}{2} \right)\backslash \{0 \} }$ are defined as in Theorem \ref{th: concentration_via_perturbation}, and $f^+ = \max \{0, f \}$ denotes the nonnegative part of $f$.

\begin{theorem}[A general Bernstein-type bound]

\label{th: extension_lezaud}
The following  concentration inequality holds. For all $u \geq 0$
\begin{equation}
    \F{P}_\nu \left (\frac{A_t}{t} \geq u \right ) \leq \norm{\frac{d \nu}{d \pi}}_2 \exp\left(-\frac{2tu^2}{\hat{\sigma}_f^2 \left (1+ \sqrt{1+ \frac{2 \norm{f_+}_\infty u}{\lambda_1 \hat{\sigma}_f^2}} \right )^2} \right),
\end{equation}
where $\Hat{\sigma}_f^2 = - \langle Sf, f \rangle$ and $\lambda_1 = \min_{\lambda \in \sigma \left ( \frac{L+L^*}{2} \right)\backslash \{0 \} } \abs{\lambda} $ are defined as in Theorem \ref{th: concentration_via_perturbation} and $f^+ = \max \{0, f \}$ denotes the nonnegative part of $f$. Furthermore,
\begin{equation}
\label{eq: extension_lezaud_bound_lambda_0(r)}
   \lambda_0(r)  \leq \frac{r^2\frac{\hat{\sigma}_f^2}{2}}{1-\left(\frac{\norm{f_+}_{\infty}}{\lambda1}\right) r} 
\end{equation}
for all $0 \leq r < \frac{\lambda_1}{ \norm{f_+}_\infty}$.
\end{theorem}
\begin{remark}
\begin{enumerate}[$(a)$]
    \item The above theorem implies that Lemma \ref{lem: lezaud_bound_lambda_0(r)} holds for $0\leq r < \frac{\lambda_1}{2 \norm{f}_\infty }$.
    \item The above Bernstein-type concentration inequality is sharper than the previous Bernstein-type concentration inequalities; Theorem \ref{th: concentration_via_perturbation}$(a)$ and Theorem \ref{th: concentration_via_poincare}. Indeed, this follows by direct comparison using $\norm{f^+}_\infty \leq \norm{f}_\infty$ and $\hat{\sigma}_f^2 \leq \Tilde{\sigma}_f^2$, where the latter inequality follows from the Cauchy-Schwartz inequality.
\end{enumerate}
\end{remark}

\begin{proof}[Proof of Theorem \ref{th: extension_lezaud}]
We follow the proof idea of \cite[Theorem\,1.2, P.\,363-364]{bernstein}, add computational details and generalize it for the non-symmetric case. Let
\begin{equation*}
    \alpha(u) = \frac{2u^2}{\hat{\sigma}_f^2 \left (1+ \sqrt{1+ \frac{2\norm{f^+}_\infty u}{\lambda_1\hat{\sigma}_f^2}}\right )^2}.
\end{equation*}
The claim follows by Theorem \ref{th: concentration_via_information} if we 
 show that $\alpha $
satisfies the condition \eqref{eq: condition_concentration_via_information} of Theorem \ref{th: concentration_via_information}, i.e.
\begin{equation}
\label{eq: condition_alpha}
    \alpha(\beta(f)) \leq I (\beta | \pi)
\end{equation}
for all $\beta \in \mathcal{M}_1(E)$ with $\beta(f) \geq 0$. An elementary calculation shows that $\alpha $ is strictly increasing and its inverse is given by (c.f. \cite[P.\,363]{bernstein})
\begin{equation*}
    \alpha^{-1}(s) = \sqrt{2 \hat{\sigma}_f^2 s} + \frac{\norm{f^+}_\infty}{\lambda_1} s 
\end{equation*}
for $s \geq 0$. Thus, the condition \eqref{eq: condition_alpha} is equivalent to 
\begin{equation}
\label{eq: conditition_beta(f)}
    \beta(f) \leq \sqrt{2 \hat{\sigma}_f^2 I (\beta | \pi)} + \frac{\norm{f^+}_\infty}{\lambda_1} I (\beta | \pi)
\end{equation}
for all $\beta \in \mathcal{M}_1(E)$ with $\beta(f) \geq 0$.
Let $\beta \in \mathcal{M}_1(E)$ and let $g := (\frac{d\beta}{d\pi})^{\frac{1}{2}}$. We use $\pi(f) = 0$ to  rewrite
\begin{equation}
\label{eq: beta(f)}
    \beta(f) = \pi(fg^2) = \pi(f [(g - \pi(g))^2 + 2 g \pi(g)]) = 2 \pi(g) \langle f, g \rangle + \pi( f(g - \pi(g))^2) .
\end{equation}
We show \eqref{eq: conditition_beta(f)}  by bounding both summands of the right hand side of the above equation. Let $K_1 := 2 \pi(g) \langle f, g \rangle$ and $K_2 := \pi( f(g - \pi(g))^2)$. 
The first summand $K_1$ can be bounded by the two inequalities
\begin{equation}
\label{eq: first_inequality}
    \pi(g) \leq \pi(g^2)^{\frac{1}{2}} = 1,
\end{equation}
and 
\begin{equation}
\label{eq: second_inequality}
    \langle f, g\rangle  \leq \sqrt{\frac{\Hat{\sigma}_f^2}{2}I (\beta | \pi)},
\end{equation}
yielding 
\begin{equation}
    K_1 \leq \sqrt{2\Hat{\sigma}_f^2I (\beta | \pi)}.
\end{equation}
Hereby \eqref{eq: first_inequality}
follows from Jensens inequality and  \eqref{eq: second_inequality} can be the derived as follows. As $- (L + L^*)$ is positive semidefinite (Lemma \ref{lem: properties_infinitesimal_generator}$(d)$),
\begin{equation*}
    b(h,g) := -  \left \langle h, \frac{L+ L^*}{2}g \right \rangle 
\end{equation*}
defines a symmetric, positive semidefinite, bilinear form on $L^2(\pi)$. Furthermore, we have  $f = \frac{L+L^*}{2}Sf = S\frac{L+L^*}{2}f  $ (by Definition \ref{def: reduced_resolvent}). Consequently, using the Cauchy-Schwartz inequality for $b$ yields
\begin{align*}
\begin{split}
    \langle f, g\rangle &= \left \langle -S f,-\frac{L+L^*}{2} g \right \rangle  = b(-Sf, g) \\
    &\leq \sqrt{\langle -Sf, f \rangle \left   \langle -\frac{L+L^*}{2}g, g \right \rangle } = \sqrt{\frac{\Hat{\sigma}_f^2}{2}I (\beta | \pi)},
\end{split}
\end{align*}
which proves \eqref{eq: second_inequality}.
Finally, the second term $K_2$ can be bounded using the Poincaré-inequality \eqref{eq: poincare_inequality} with constant $C = \frac{1}{\lambda_1}$,  yielding
\begin{align*}
    K_1 &=  \pi( f(g - \pi(g))^2) \leq \pi( f^+(g - \pi(g))^2)  \leq   \norm{f^+}_\infty \mathrm{Var}_\pi(g) \\
    &\leq \frac{\norm{f^+}_\infty}{\lambda_1} \langle - Lg,g \rangle = \frac{\norm{f^+}_\infty}{\lambda_1} I (\beta | \pi).
\end{align*}
Combining the bounds for both summands $K_1$ and $K_2$ of the right hand side of \eqref{eq: beta(f)} yields the desired inequality \eqref{eq: conditition_beta(f)}.
\end{proof}
\begin{reference_paper}
We extended the proof of \cite[Thrm\,1.2]{bernstein} (found on \cite[P.\,363]{bernstein}) to the nonsymmetric case. The central 'trick' was to replace $L$ by $\frac{L+L^*}{2}$ in scalar products (which is easily possible because of the finite dimensionality of $L^2(\pi)$) and work with the reduced resolvent of $\frac{L+L^*}{2}$. 
\end{reference_paper}

%% file: Generalization_infinite_state_space.tex
\section{Summary, Further Theory and Applications in Physics}
\label{sec: schluss}
\subsection{Summary}
\label{subsec: Summary}
In this thesis we derived bounds in Theorems \ref{th: main_conc_inequality}, \ref{th: concentration_via_perturbation}, \ref{th: concentration_via_poincare}, \ref{th: concentration_via_F_sobolev}, \ref{th: concentration_via_information}, \ref{th: extension_lezaud}  of the form
\begin{equation}
    \F{P}_\nu \left ( \frac{A_t}{t} \geq u \right) \leq \norm{\frac{d\nu}{d\pi}}_2 e^{-t \alpha(u)}
\end{equation}
based on the Cramér-Chernoff method. First, we introduced the general Cramér-Chernoff method, which can be applied to arbitrary random variables $Z$ to obtain bounds  (c.f. Section \ref{subsubsec: general_cramer_chernoff_method})
\begin{equation}
\label{eq: general_chernoff}
    \F{P}( Z \geq u) \leq e^{-\Psi_Z^*(u)}.
\end{equation}
Moreover, as a special case of the above Chernoff inequality we obtained Bernstein's inequality for sub-gamma random variables. Then, we applied the Cramér-Chernoff method to to functionals of MJPs, i.e. $ Z = \int_0^t f(X_s) ds$. More precisely, we considered the setting: 
\begin{setting}
Consider an irreducible MJP $(\F{X}, (\F{P}_x)_{x \in E})$ on a finite state space $E$ with invariant distribution $\pi$ and infinitesimal generator $L$, some arbitrary, nonconstant $f \in \mathcal{B}(E) = \F{R}^E$ with $\pi(f) = 0$, some arbitrary initial distribution $\nu \in \mathcal{M}_1(E)$, and let $\lambda_0(r)$ (for $r \in \F{R})$ be the largest eigenvalue of the selfadjoint operator $\frac{L+L^*}{2} + r M_f$ (defined on $L^2(\pi)$), where $M_f$ is the multiplication with $f$.  Furthermore, let $A_t = \int_0^t. f(X_s)ds$, $\Psi_{A_t}(r) = \log \F{E}_\nu (e^{rA_t})$, and $\Psi_{A_t}^*(u) = \sup_{ r \geq 0} (ru - \Psi_{A_t}(r))$.
\end{setting}
First we followed Wu \cite{wu}, and starting from the general Chernoff inequality \eqref{eq: general_chernoff} for $Z = A_t$ and using a Feynman Kac semigroup $P_t^{rf} = \exp(t(L+rM_f))$ to bound $\Psi_{A_t}(r)$ we obtained a general concentration inequality
\begin{equation}
    \F{P}_\nu \left ( \frac{A_t}{t} \geq u \right ) \leq e^{ - \Psi_{A_t}^*(tu)} \leq     \norm{\frac{d\nu}{d\pi}}_2 e^{-t \lambda_0^*(u)},
\end{equation}
depending on $\lambda_0^*$, the Fenchel conjugate of $\lambda_0$. We noted that for $\nu \neq \pi$ the above bound is trivial for $u \in [0, K]$ (where $K > 0$), and that if $\pi$ obeys the detailed balance condition, then this bound is asymptotically sharp (see Remark \ref{rem: th_main_concentration_inequality} for details). Afterwards, based on this (general) concentration inequality we used three different approaches to get more explicit bounds $ \alpha(u) \leq \lambda_0^*(u) $, thus obtaining
\begin{equation}
\label{eq: summary_concentration_inequality}
    \F{P}_\nu \left ( \frac{A_t}{t} \geq u \right ) \leq e^{ - \Psi_{A_t}^*(tu)} \leq     \norm{\frac{d\nu}{d\pi}}_2 e^{-t \lambda_0^*(u)} \leq \norm{\frac{d\nu}{d\pi}}_2 e^{-t \alpha(u)}.
\end{equation}
In the first (Section \ref{subsubsec: concentration_via_perturbation}) approach we followed Lezaud \cite{lezaud} and applied perturbation theory to express $\lambda_0(r)$ as a perturbation series and derived (after a lengthy computation, see proof of Lemma \ref{lem: lezaud_bound_lambda_0(r)}) a sub-gamma type bound (c.f. Definition \ref{def: sub-gamma} and Lemma \ref{lem: lezaud_bound_lambda_0(r)})

\begin{equation}
\label{eq: summary_sub_gamma_bound}
    G(r) =  \frac{r^2v}{2(1-\frac{2r}{\epsilon})} \geq \lambda_0(r),
\end{equation}
which implied by computing the Fenchel conjugate $G^*(u) \geq \lambda_0^*(u)$ the concentration inequality of Theorem \ref{th: concentration_via_perturbation}. Hereby, we corrected the claim of \cite[Lemma\,2.3]{lezaud} as we noted that the proof (presented in \cite{lezaud}) of the sub-gamma type bound \eqref{eq: summary_sub_gamma_bound}  just proves this bound  for  $r \in [0, \frac{\epsilon}{3}]$ and not for $r \in [0, \frac{\epsilon}{2}] $ (as claimed in \cite[Lemma\,2.3]{lezaud}). As a result of this correction, the above sub-gamma type bound just implied a Bernstein-type concentration inequality for small $u$ (Theorem \ref{th: concentration_via_perturbation}(a)) and a weaker inequality for larger $u$ (Theorem \ref{th: concentration_via_perturbation}(b)).  In the second approach (Section \ref{subsubsec: concentration_via_functional_inequalities}) we followed Guillin \cite{guillin} and used functional inequalities; the Poincaré inequality and the F-Sobolev inequality, which give bounds for $\langle g, L g \rangle$. Hereby, in our setting, the Poincaré inequality 'automatically' holds whereas a general $F$-Sobolev inequality may not hold in general. Using these inequalities we derived again bounds $G(r) \geq \lambda_0(r)$, and by computing $G^*(u) \leq \lambda_0^*(u) $ we obtained Theorems \ref{th: concentration_via_poincare} and \ref{th: concentration_via_F_sobolev}. Herewith, we generalized, reformulated \cite[Prop.\,1.4]{guillin} and noted that the Poincaré inequality implies a Bernstein-type concentration inequality; Theorem \ref{th: concentration_via_poincare}. Furthermore, we noted that if an $\log$-Sobolev inequality holds, one obtains a continuous time analogue of Chernoff's inequality for sums of i.i.d. random variables. Finally, in the third approach (Section \ref{subsubsec: concentration_via_information}) we followed Gao \cite{bernstein} and used the Donsker-Varadhan information. First, we expressed $\lambda_0^*(u)$ in terms of the Donsker-Varadhan information. Then, using this expression we arrived at Theorem \ref{th: concentration_via_information}; a concentration inequality  \eqref{eq: summary_concentration_inequality} holds for a function $\alpha : [0,\infty) \rightarrow [0,\infty]$ if
\begin{equation*}
    \alpha(\beta(f)) \leq I( \beta | \pi)
\end{equation*}
 for all $\beta \in \mathcal{M}_1(E)$ with $\beta(f) \geq 0$. Afterwards, we showed that for a Bernstein-type 
 \begin{equation*}
     \alpha(u) = \frac{2u^2}{v(1+ \sqrt{1 + \frac{2uc}{v}})^2}
 \end{equation*}the above condition is automatically satisfied in our setting, thus obtaining a general Bernstein-type bound; Theorem \ref{th: extension_lezaud}, which is a strengthening and extension of the previous obtained Bernstein-type bounds; Theorem \ref{th: concentration_via_perturbation}(a) and Theorem \ref{th: concentration_via_poincare}. Herewith, we proved the expression of $\lambda_0^*(u)$ in terms of the Donsker-Varadhan information (a fact which is just stated in \cite{bernstein}) and extended the results of \cite{bernstein} to the general, non-symmetric case. \\
Thus, summarizing the (central, nonredundant) results, we get: Assume the above \textbf{Setting}. Then, the following general concentration inequality holds
\begin{customthm}{\ref{th: main_conc_inequality}}[A general inequality]
For any $u \geq 0$
\begin{equation*}
 \F{P}_\nu \left (\frac{A_t}{t} \geq u \right ) \leq \norm{\frac{d\nu}{d \pi}}_2 e^{- t \lambda_0^*(u)},  
\end{equation*}
where $\lambda_0^*(u)$ denotes the Fenchel conjugate of $\lambda_0(r)$ and we also have
\begin{equation*}
    \lambda_0^*(u) =  \inf \{ - \langle Lg, g \rangle \; | \, \norm{g}_2 = 1, \langle M_fg, g \rangle = u \}
\end{equation*}
\end{customthm}
Based on this inequality  more explicit concentration inequalities follow (by using bounds on $\lambda_0$ or $\lambda_0^*$): The following Bernstein-type bound holds without further assumptions (than made in the above setting)
\begin{customthm}{\ref{th: extension_lezaud}}[A general Bernstein-type bound]
Let $S$ denote the reduced resolvent of $\frac{L+L^*}{2}$ \lb with respect to the eigenvalue $0$\rb , $\hat{\sigma}_f^2 = - 2 \langle f, Sf \rangle$, $\lambda_1 = \min_{\lambda \in \sigma \left ( \frac{L+L^*}{2} \right)\backslash \{0 \} } \abs{\lambda} $ the spectral gap, and $f^+ = \max \{ 0, f \}$ the nonnegative part of $f$. For all $u \geq 0$
\begin{equation*}
    \F{P}_\nu \left (\frac{A_t}{t} \geq u \right ) \leq \norm{\frac{d \nu}{d \pi}}_2 \exp\left(-\frac{2tu^2}{\hat{\sigma}_f^2 \left (1+ \sqrt{1+ \frac{2 \norm{f_+}_\infty u}{\lambda_1 \hat{\sigma}_f^2}} \right )^2} \right).
\end{equation*}
\end{customthm}
Furthermore, if one assumes further conditions one gets:
\begin{customthm}{\ref{th: concentration_via_F_sobolev}}[A bound assuming an $F$-Sobolev inequality] Assume additionally that an $F$-Sobolev inequality holds and let $r_f = \frac{F(0)}{\min_{x \in E}f(x)}$. Then, for all $ u \geq 0$
\begin{equation*}
    \F{P}_\nu \left (\frac{A_t}{t} \geq u \right ) \leq \norm{\frac{d \nu}{d \pi}}_2 \exp( - t \sup_{ r \in [0, r_f)} (ru -F(\pi( F^{-1}(rf))))) .
\end{equation*}
\end{customthm}
and 
\begin{customthm}{\ref{th: concentration_via_information}}[A bound assuming an information inequality] Let $I (\cdot | \pi)$ denote the Donsker-Varadhan information, and assume additionally that $\alpha : [0, \infty) \rightarrow [0, \infty]$ is some function satisfying 
\begin{equation*}
    \alpha(\beta(f)) \leq I (\beta | \pi)
\end{equation*}
for all probability measures $\beta$ on $E$ such that $\beta(f) \geq 0$.
Then, for all $u \geq 0$
\begin{equation*}
    \F{P}_\nu \left( \frac{A_t}{t} \geq u \right ) \leq \norm{\frac{d\nu}{d\pi}}_2 e^{-t \alpha(u)}.
\end{equation*}
\end{customthm}
Furthermore, it should be remarked that although we focused on upper tail probabilities, by replacing $f$ by $-f$ one obtains bounds for the lower tail probabilities, i.e. 
\begin{equation}
    \F{P}_\nu \left (\frac{A_t}{t} \leq u \right) \leq \norm{\frac{d\nu}{d\pi}}_2 e^{-t\alpha(-u)}.
\end{equation}
Finally, by subbaditivity one obtains bounds of the form
\begin{equation}
    \F{P}_\nu \left (\abs{\frac{A_t}{t}} \geq u \right) \leq \norm{\frac{d\nu}{d\pi}}_2 ( e^{-t\alpha_1(u)} +   e^{-t\alpha_2(-u)}).
\end{equation}

\subsection{Outlook: Further Theory}
Let us present a brief outlook into further theory related to concentration inequalities for MJPs, which was not covered in this thesis. \\
In this work, we applied the Cramér-Chernoff method to functionals of irreducible MJPs on finite state spaces and obtained bounds
\begin{equation*}
    \F{P}_\nu \left( \frac{A_t}{t} \geq u \right ) \leq \norm{\frac{d\nu}{d\pi}}_2 e^{-t \alpha(u)}.
\end{equation*}
A question that was not covered (in detail), is the question on the sharpness of these bounds. It is not clear that these bounds are automatically  sharp; we even have seen that for $\nu \neq \pi$ the above bound is trivial for small $u$ (see Remark \ref{rem: th_main_concentration_inequality}$(b)$). However, the question of sharpness is particularly of interest for the application of such concentration inequalities, where one is interested to approximate $\pi(f) $ by $t^{-1} \int_0^t f(X_s)ds$, and consequently sharp estimates for the deviation probabilities $\F{P}_\nu \left( \frac{A_t}{t} \geq u \right )$ are desired. A possible approach to discuss the sharpness is large deviation theory; the study of the asymptotic limit
\begin{equation}
    \limsup_{t \to \infty} \frac{\log(\F{P}_\nu \left ( \frac{A_t}{t} \geq u \right))}{t}.
\end{equation}
One then could compare the above limit to the asymptotic limit 
\begin{equation}
    -\alpha(u) = \limsup_{t \to \infty} \frac{\log(\norm{\frac{d\nu}{d\pi}}_2 e^{-t \alpha(u)})}{t} 
\end{equation}
of the bound. This comparison would yield a first impression of the (asymptotic) sharpness of the concentration inequality. We have already seen such a result briefly in Remark \ref{rem: th_main_concentration_inequality}$(e)$, for an extensive presentation of large deviation theory for Markov processes see \cite{large_deviations_stroock}. Another aspect which limits the sharpness of the concentration inequalities of this work is the sharpness of Chernoff's inequality: All concentration inequalities in this work are based on this inequality, and are not sharper than it. So, to analyze the sharpness of the concentration inequalities of this work, it may be also of interest to have a look at (general) theory concerning the sharpness of the Chernoff inequality. Furthermore, recall that our general concentration inequality (Theorem \ref{th: main_conc_inequality}) was based on the Cauchy Schwartz inequality (c.f. Section \ref{subsubsec: general_concentration_inequality})
\begin{equation}
   \F{E}_\nu( e^{rA_t}) = \langle \frac{d\nu}{d\pi}, P_t^{rf}\textbf{1} \rangle \leq \norm{\frac{d\nu}{d\pi}}_2 \norm{P_t^{rf}}_2,
\end{equation}
which is in general not an equality. Another possible approach to bound $\F{E}_\nu(e^{rA_t})$ may be by using a Dyson identity to get an exact series representation of $P_t^{rf}\textbf{1} = \exp(L+rM_f)\textbf{1}$ and then using this series representation to obtain a bound for $\F{E}_\nu( e^{rA_t})$, c.f. for example \cite{Lapolla}.
Finally, a more pragmatic, direct approach to study the sharpness, is to consider concrete examples of MJPs, where one can calculate (analytically or computationally) $\F{P}_\nu \left( \frac{A_t}{t} \geq u \right )$ and then to directly compare this (exact) result to the bound given by the concentration inequality. \\
Another aspect that was not covered here, is the generalization of the present results to more general Markov processes on (uncountably) infinite state spaces. One can extend the presented approach to more general Markov Processes on (uncountably) infinite state spaces, in fact the works \cite{wu}, \cite{lezaud}, \cite{guillin}, \cite{bernstein} on which this thesis is based on, work with more general Markov processes and the qualitative form of these results (and the corresponding proofs) are quite similar. Formally, the presented approach, i.e. the application of the Cramér-Chernoff method,  can be directly transfered to more general Markov processes. However, a rigorous generalization of the results of this work leads to technical subtleties due to the infiniteness of the state space. Let us discuss briefly some changes and technicalities that have to be considered for a (rigorous) generalization of the results. \\ For a general Markov process the existence of an invariant measure is not guaranteed and even if such a measure exists, asymptotic properties of the Markov process may not be given by the invariant measure. Furthermore, one still needs some form of regularity of the paths $t \to X_t$, otherwise the time average $\int_0^t f(X_s) ds$ is not well defined. Thus, in more general settings (considered in \cite{wu}, \cite{lezaud}, \cite{guillin}, \cite{bernstein}) one considers a càdlàg Markov process $(\F{X}, (\F{P}_x)_{x \in E})$  on  a polish space $E$, where càdlàg (continue à droite, limite à gauche) means that it's paths $t \mapsto X_t)$ are right continuous and have everywhere existing left limits. Furthermore, one assumes the existence of an invariant measure $\pi$ such that $\pi$ is ergodic with respect to the Markov semigroup $(P_t)_{t \geq 0}$ (defined as in \eqref{eq: contraction_semigroup_Markov_process}). Here ergodic means that $\mathscr{L}_\nu(X_t)$ converges to 'fast enough' $\pi$ for all initial distributions. Moreover, note that the function $f$ over which the time average is taken can not be as general as in our case (where $f \in \F{R}^E$ was arbitrary); as $\int f d\pi$ has to be defined one must have at least $f \in L^1(\pi)$. Furthermore, for unbounded $f$ many of our results and proofs can not be applied;  all computations involving $\norm{f}_\infty$ are invalid and also the Feynman-Kac semigroup $P_t^f$ must be treated differently; e.g. by approximating $f$ with bounded functions (c.f. \cite[Proof of Thrm.\,1]{wu}). Also, in a general setting not for all $\nu \in \mathcal{M}_1(E)$ a density $\frac{d\nu}{d\pi} \in L^2(\pi)$ might exist (whereas in our setting a density always exists); for example if $\nu$ is a delta distribution and $\pi$ is a measure with density with respect to the Lebesgue measure. However, the basic starting point stays the same; one considers  the Markov semigroup $(P_t)_{t \geq 0}$ on $L^2(\pi)$ (c.f. Remark \ref{rem: extension_semigroup_to_L^p}), its infinitesimal generator $L: D(L) \rightarrow L^2(\pi)$ with domain $D(L)$, a Feynman Kac semigroup $(P_t^{rf})_{ t \geq 0}$ and bounds the cumulant generating function $\Psi_{A_t}$ by $ \Psi_{A_t}(r) \leq \log \norm{\frac{d\nu}{d\pi}}_2 + \log \norm{P_t^{rf}}_2 $.  Based on this bound one derives concentration inequalities by bounding $\norm{P_t^{rf}}_2$ (c.f. \cite{wu}). Here  it has to be made sure, that $L$ still has the properties used in the proofs, like the simplicity of the eigenvalue $0$, a spectral gap, the Poincaré-inequality or other properties of Lemma \ref{lem: properties_infinitesimal_generator}, which may be in general not given. For example \cite{lezaud} assumes a spectral gap to be able to generalize the applied perturbation theory (see \cite[P.\,195]{lezaud}). Furthermore, in expressions involving the infinitesimal generator, such as $\langle g_1, Lg_2 \rangle, \langle g_1, \frac{L+L^*}{2}g_2 \rangle$ or $\langle g_1 , Sg_2 \rangle$ the functions $g_1,g_2$ must always be contained in the domain of the operators, which does not hold for all $g_1,g_2 \in L^2(\pi)$ as operators such as $L, L^*$ are just densely defined. Finally, $\lambda_0(r)$ may not be defined as the largest eigenvalue, but defined as a supremum (c.f. Lemma \ref{lem: biggest_eigenvalue}, Reference \ref{ref: lambda_0(r)_as_supremum}).

\subsection{Digression: Markov  Processes and Concentration Inequalities in Physics}
\label{subsec: conc_inequalities_physics}
The present section may be of interest for the reader who is interested in application of concentration inequalities in a physical context. First, we  briefly explain how stochastic processes and Markov (jump) processes arise in physical systems and then explain how concentration inequalities for MJPs may be applied in a physical context. For a more detailed  discussion of how stochastic processes arise in physics see \cite[Ch.\,III.2]{KAMPEN}. \\ Often one considers a physical system (with many degrees of freedom) that is observed in time, for example a colloidal particle in water  or a protein in a biological system (these two systems will be used as  examples throughout this section). One  is interested in the time evolution of the the state of some specific physical object or observable. Refering to the two aforementioned examples, this could be the position of  the colloidal particle or the conformation of the protein. Mathematically, the state (of the physical object of interest) at time $t$ is described as some $x_t \in E$, where $E$ is a set, which contains all possible states of the object of interest. In the case of the colloidal particle we would have (in three dimensions) $E = \F{R}^3$, in the case of the protein, $E$ could be some discrete finite set, where each $x \in E$ would represent some concrete conformation of the protein. In many particle systems (e.g. a biological system) the state space $E$ does not describe the whole physical system (classically the whole system would be completely described by the momentum and position of all particles). Consequently, as the physical object (described by $E$) of interest is coupled to the rest of the system (e.g. the colloidal particle collides with the surrounding water molecules), whose specific state is not known, the observed trajectory $(x_t)_{t \geq 0} \in E^{[0,\infty)}$, appears to be 'irregular' and 'non-deterministic'. Thus, the observed 'randomness' of the trajectory comes from the 'ignorance' of the rest of the physical system. \par For example,  in the case of the colloidal particle  the momenta and positions of the water molecules are unknown, consequently one cannot predict the movement of the colloidal particle (as the particle collides with the water molecules) and the observed motion of the particle seems random and irregular. This suggest, that for the description of the time evolution, the observed trajectory $(x_t)_{t \geq 0}$ should be considered as a 'random' trajectory, which is mathematically exactly a stochastic process. Thus, one transitions from a description considering a single trajectory $(x_t)_{t \geq 0}$ to an 'ensemble' $\{(X_t)_{t \geq 0}(\omega) |  \omega \in \Omega \}$ of trajectories, described by a stochastic process $(X_t)_{t \geq 0}$ on $E$, defined on some underlying probability space $(\Omega, \mathcal{F}, \F{P})$. Intuitively, $\omega \in \Omega$ can be thought of parametrizing the 'uncertainty' (or 'randomness'), due to the 'random' interaction of the rest of the (unknown) physical system with the physical object of interest. Physically, a concrete realization $X_t(\omega)$ corresponds to  one concrete observation (i.e. performing an experiment) of the physical object and doing $n$ observations would correspond to drawing $\omega_1,...,\omega_n$ from $\Omega$, independently according to $\F{P}$ (the observations would correspond to $(X_t)_{t \geq 0}(\omega_1),..., (X_t)_{t \geq 0}(\omega_n)$). The  distribution of the stochastic process depends on the physical system and the description of it. One may impose certain (physically or practically justified) conditions like Markovianity or continuity of paths. In our example of a protein with certain conformations, if one assumes Markovianity and right continuity, then $(X_t)_{t \geq 0}$ is just exactly an MJP (examples where biomolecules are modeled by an MJP are found in \cite{schuette}, \cite{Schuette2015}). Often also the 'random' influences of the surrounding system are modeled (e.g. the force on the colloidal particle resulting from collisions with water molecules may be modeled as white noise) or they are obtained by projection of the high dimensional many particle dynamics (including the object of interest and the surrounding system) onto the dynamics of the object of interest (see e.g. \cite{Chorin2002}), and in doing this one obtains stochastic differential equations (e.g. the Langevin equation, see  \cite[Ch.\,1]{Zwanzig2001}), which then determine the distribution of the stochastic process. \par 
As already mentioned, MJPs (on a finite state space) arise in physics always directly when one assumes Markovianity of the system. Markovianity is  justified when one has a time scale separation between the time scale of the dynamics of the physical object of interest and the hidden dynamics of the surrounding system, more precisely when the time scale governing the surrounding system is much smaller than of the physical object. The intuitive argument on why this implies Markovianity is as follows. \par Without the (local) interaction of the object with the surrounding system, the system is at an equilibrium state. Now, when the physical object interacts with its surrounding, the surrounding system gets perturbed out of equilibrium. However, because of the much smaller time scale, the surrounding system quickly again reaches the equilibrium state, whereas the state of the object does not change significantly, and thus, the surrounding system 'forgets' its past interaction with the object. Consequently, given the present state of the object, the future interaction of the object with its surrounding, determining the future time evolution of the state of the object, is independent of the past, which corresponds exactly to the Markov property. Moreover, MJPs also arise in the context  when one examines the time evolution of so called 'site populations' (see \cite{moro}).\\
Let us now discuss the possible applications of concentration inequalities for MJPs. Nowadays it is experimentally possible to probe individual trajectories $ t \mapsto X_t(\omega)$ (e.g. single particle tracking \cite{SP1}, \cite{SP2}, single molecule spectroscopy \cite{SM1}, \cite{SM2}) and  one is interested to deduce properties  of the physical system by the observed individual trajectories. A possible approach is the study of time averages 
\begin{equation}
   \frac{A_t}{t} =  \frac{1}{t}\int_0^t f(X_s) ds,
\end{equation}
discussed in this work. Functionals of this form arise  for example in the context of chemical inference  \cite{berg_purcell}, time average statistical mechanics \cite{Lapolla} and stochastic thermodynamics \cite{Seifert_2012}.  In the analysis of these time averages it is of particular interest to establish relations between the fluctuations (of the time average) and physical properties and quantities (c.f. \cite{Lapolla}, \cite{tur_1}, \cite{horowitz}). Although the asymptotic analysis of those time averages (and fluctuations) by application of asymptotic results, like the ergodic theorem or large deviation theory is already well established in the physics literature (see e.g. \cite{Maes_2008}, \cite{kaiser}), a general approach
to the non-asymptotic study of the statistics of the above time averages remains elusive (c.f. \cite{Lapolla}). Moreover, in single molecule experiments ergodic time scales (i.e. time scales where $\mathscr{L}(X_t)$ is close to $\pi$ and $t^{-1}A_t$ is close to $\pi(f)$) frequently cannot be reached  (see e.g. \cite{non_ergodic1}, \cite{non_ergodic2}), thus the
correspondence between time and ensemble averages breaks down and the typical behavior of time-averaged
observables is frequently found to be very different from ensemble-averages. \par Furthermore, from a practical point of view, the quantitative analysis of the rate of convergence of $t^{-1}A_t \to \pi(f)$ is of interest, as this yields quantitative results on the time scales when the time average $t^{-1}A_t$ (obtained in experiments) is a 'good' approximation for the  ensemble average $\pi(f)$.  Consequently, the study of concentration inequalities is natural, as they provide some insight to the fluctuations of time averages and  a quantitative bound for the time scales, where time averages become ensemble averages. Let us illustrate this idea in an example (c.f. \cite{free_energy}). One may be interested in establishing results about the energies $H(x)$ of the states $x \in E$, which are related (for systems which obey detailed balance) to $\pi$ by the Boltzmann distribution $\pi_x \propto \exp(\frac{-H(x)}{k_B T})$. Thus, to estimate the energies, one approximates $\pi(f)$ by  $t^{-1}A_t$ (which can be measured). However, the ergodic theorem (Theorem \ref{th: ergodic_theorem}) just yields the asymptotic result $t^{-1}A_t \to \pi(f)$ and consequently it is not clear that this strategy gives a 'good' approximation. Concentration inequalities 
\begin{equation}
    \F{P}(t^{-1}A_t \geq u) \leq B(u)
\end{equation}
provide quantitative results on time scales when '$A_t$ is close to $\pi(f)$ with a high probability' and thus yield sufficient conditions on when the above strategy (to estimate energies) yields a 'good approximation with a high probability'. Furthermore, if one applies the concentration inequalities of this work, which are exponential concentration inequalities ($B(u) = \exp(-\Phi(u))$ derived by the Cramér-Chernoff method, the concentration inequalities for the observation of a single trajectory (i.e. measuring once), directly generalize to concentration inequalities for $n$ independent observations by using Corollary \ref{cor: cramer_chernoff_sum_iid} to obtain 
\begin{equation}
\label{eq: conc_ineq_sum}
    \F{P} \left ( \frac{1}{n}\sum_{i = 1}^n t^{-1}A_t^{(i)} \geq u \right) \leq \exp(-n \Phi(u)),
\end{equation}
where $t^{-1}A_t^{(i)}$ are the independent time averages, obtained by independent MJPs. Practically, \eqref{eq: conc_ineq_sum} gives an sufficient condition on the number of times an experiment has to be conducted to obtain 'good approximation for $\pi(f)$ with  high probability'.

%% file: Appendix.tex
\newpage
\appendix
\section{Appendix}

\subsection{Perturbation Theory}
\label{subsec: app_perturbation_theory}
This section is devoted to proof Lemma \ref{lem: eigenvalues_continuous} and Theorem \ref{th: perturbation_simple_eigenvalue} for complex vector spaces and transfer them to real vector spaces (see Remarks \ref{rem: app_red_resolvent_real_case}, \ref{rem: app_eigenvalues_continuous}, \ref{rem: app_perturbation_simple_eigenvalue}). For proving the results in the complex case ; Lemma \ref{lem: app_eigenvalues_continuous} and Theorem \ref{th: app_perturbation_simple_eigenvalue}, we combine results of \cite{Kato}.  \\
To rigorously transfer the results to real vector spaces, we need 
\begin{lemma}[Complexification]
\label{lem: app_complexification}
Let $V$ be a real vector space, and $\langle \cdot, \cdot \rangle$ an inner product on $V$. Then, the complexification $V^{\F{C}} = \{ v_1 + iv_2 \, | \, v_1, v_2 \in V \}$ has the following properties:
\begin{enumerate}[\lb a\rb]
    \item For every operator $T : V \rightarrow V$, there is a unique complexification $T^{C}$, i.e. a unique \lb $\F{C}$ - linear\rb \space operator $T^{\F{C}}: V^{\F{C}} \rightarrow V^{\F{C}}$ that extends $T$, i.e. $T = T^\F{C} |_{V}$. For all $v_1, v_2 \in V$ we have
    \begin{equation}
    \label{eq: app_T^C}
        T^\F{C}(v_1 + iv_2) = Tv_1 + i Tv_2.
    \end{equation}
    \item There is a unique \lb complex\rb \space  inner product $\langle \cdot , \cdot \rangle^\F{C}$ on $V^\F{C}$ that extends $\langle \cdot , \cdot \rangle$. We have 
    \begin{equation}
    \label{eq: app_complex_inner_product}
        \langle (v_1 + iv_2), (v_1' + iv_2') \rangle^\F{C} = \langle v_1,v_1' \rangle + \langle v_2, v_2' \rangle + i ( \langle v_1, v_2' \rangle - \langle v_2, v_1' \rangle  ) 
    \end{equation}
    \item We have $(T_1T_2)^\F{C} = T_1^\F{C} T_2^\F{C}$ and the operation $T \mapsto T^\F{C}$ preserves properties of $T$, we have\begin{enumerate}[1.]
        \item $T$ and $T^\F{C}$ have the same characteristic polynomials   and $\Tr (T) = \Tr(T^\F{C})$, in particular $T$ and $T^\F{C}$ have the same eigenvalues and algebraic multiplicities of eigenvalues.
        \item If $\textup{pr}$ is a(n) (orthogonal) projection, $\textup{pr}^\F{C}$ is also a(n) (orthogonal) projection
        \item If $T$ is selfadjoint with respect to $\langle \cdot , \cdot \rangle$, $T^\F{C}$ is selfadjoint with respect to $\langle \cdot, \cdot \rangle^\F{C}$
        \item If $T$ is diagonalizable \lb with real eigenvalues\rb, i.e. $T = \sum_{\lambda \in \sigma(T)} \lambda \textup{pr}_\lambda$, where $\textup{pr}_\lambda$ is the projection onto $\textup{Ker}(T - \lambda)$ according to the decomposition $V = \bigoplus_{\lambda \in \sigma(T)} \textup{Ker}(T-\lambda)$, then $T^\F{C}$ is diagonalizable with $T^\F{C} = \sum_{\lambda \in \sigma(T)} \lambda \textup{pr}_\lambda^\F{C}$ and $\textup{pr}_\lambda^\F{C}$ is the projection onto $ \textup{Ker}(T - \lambda)^\F{C} = \textup{Ker}(T^\F{C} - \lambda)$ according to the decomposition $V^\F{C} = \bigoplus_{\lambda \in \sigma(T)} \textup{Ker}(T-\lambda)^\F{C} $ 
    \end{enumerate} 
\end{enumerate}

\end{lemma}
\begin{proof}
The proofs are elementary, so we omit the computations. Statement $(a)$ and  $(b)$ follow directly from the $\F{C}$ linearity for operators and  the sesquilinearity for complex inner products. Statement $(c)$.1 follows  by using that if $(v_i)_{i = 1,..., \dim V}$ is a basis of $V$, then it is also a basis of $V^{C}$ and that the matrix representations of $T$ and $T^\F{C}$ under this basis are the same. Statements $(c)$.2 - $(c)$.4 follow by direct computation using \eqref{eq: app_T^C} and \eqref{eq: app_complex_inner_product}.
\end{proof}

In the following (unless otherwise stated) let $V$ be a complex vector space, $H$ a complex Hilbert space and let $\sigma(T)$ denote the spectrum of an operator $T: V  \rightarrow V$.  We first present some basic definitions and concepts required for  proving Theorem \ref{th: app_perturbation_simple_eigenvalue}.

\begin{definition}(Resolvent)
Let $T : V \rightarrow V$ be an operator. The operator-valued function $R$ defined on $\F{C} \backslash \sigma(T)$ via 
\begin{equation*}
    R(\zeta) = (T - \zeta )^{-1}
\end{equation*}
is called the resolvent of $T$.
\end{definition}

For $\lambda \in \sigma(T)$ it can be shown that the Laurent series expansion of $R(\zeta)$ at $\lambda$ takes the following form \cite[P.\,39-40]{Kato}
\begin{equation}
\label{eq: resolvent_laurent_series}
    R(\zeta) = - (\zeta - \lambda)^{-1}\text{pr}_\lambda - \sum_{n = 1}^{m_\lambda-1}(\zeta - \lambda)^{-n-1}D_\lambda^n + S_\lambda(\zeta),
\end{equation}
where $\text{pr}_\lambda$ is a projection, i.e. $\text{pr}_\lambda^2 = \text{pr}_\lambda$, $D_\lambda$ is a nilpotent operator with $D_\lambda^{m_\lambda} = 0$, $m_\lambda = \mathrm{dim }\mathrm{Im}(\text{pr}_\lambda)$, and $S_\lambda(\zeta)$ is a holomorphic operator-valued function. Kato \cite{Kato} defines the projection $\text{pr}_\lambda$ as the \textit{eigenprojection} for the eigenvalue $\lambda$, the integer $m_\lambda$ as  the \textit{algebraic multiplicity} of $\lambda$, and the holomorphic operator valued function $S_\lambda(\zeta)$ as  the \textit{reduced resolvent} of $T$ with respect to the eigenvalue $\lambda$ \cite[P.\,40-41]{Kato}. An eigenvalue $\lambda$ is called \textit{simple} if $m_\lambda = 1$ \cite[P.\,41]{Kato}. A straightforward application of the residue theorem shows that for any positively oriented circle $\Gamma $ in $\F{C}$ containing exactly one eigenvalue $\lambda \in \sigma(T)$ we have
\begin{equation}
 \label{eq: eigenprojection_integral_definition}
    \text{pr}_\lambda = \frac{-1}{2 \pi i}\int_\Gamma R(\zeta) d\zeta.
\end{equation}
Furthermore, the eigenprojections satisfy \cite[P.\,40]{Kato}
\begin{align}
\label{eq: properties_pr_lambda}
\begin{split}
    &\sum_{\lambda \in \sigma(T)} \text{pr}_\lambda = \text{Id}_V, \\
    & \text{pr}_\lambda \text{pr}_\mu = \delta_{\lambda \mu} \text{pr}_\lambda,\\
    & \text{pr}_\lambda T = T \text{pr}_\lambda  ,
\end{split}
\end{align}
where $\delta_{\lambda \mu}$ denotes the Kroenecker delta. Thus, the eigenprojections define a decomposition $V = \bigoplus_{\lambda \in \sigma(T)} M_\lambda$, with $M_\lambda = \mathrm{Im}(\text{pr}_\lambda)$ and $TM_\lambda \subset M_\lambda$. To clarify Kato's notions we state and prove 
\begin{lemma}
\label{lem: katos_notions}
Kato's notions of eigenprojection, algebraic multiplicity , and the decomposition $V = \bigoplus_{\lambda \in \sigma(T)}M_\lambda $ defined as above by the Laurent series expansion \eqref{eq: resolvent_laurent_series} coincide  with  the 'usual' notions of eigenprojection and algebraic multiplicity, defined by the Jordan decomposition and the characteristic polynomial. In other words, if $m(\lambda)$ denotes the multiplicity of the root $\lambda$ in the characteristic polynomial of $T$, and
\begin{equation}
\label{eq: jordan_decomposition}
    T = \sum_{\lambda \in \sigma(T)} \lambda \hat{\textup{pr}}_\lambda + \hat{D}
\end{equation}
denotes the Jordan decomposition, where $\hat{\textup{pr}}_\lambda$  denotes the projection onto the generalized eigenspace $V_\lambda = \mathrm{Ker}(T - \lambda)^{m(\lambda)}$ according to the decomposition $V = \bigoplus_{\lambda \in \sigma(T)}V_\lambda$ and $\hat{D}$ denotes the nilpotent operator belonging to the Jordan decomposition,  then $\textup{pr}_\lambda = \hat{\textup{pr}}_\lambda$, $m_\lambda = m(\lambda)$ and $M_\lambda = V_\lambda $
\end{lemma}
And as a result we obtain immediately 
\begin{corollary}
\label{cor: selfadjoint case}
In particular,  if $T$ is diagonalizable, i.e. $T = \sum_{ \lambda \in \sigma(T)} \lambda \textup{pr}_\lambda$, then $\textup{pr}_\lambda$ is the projection onto the eigenspace $\mathrm{Ker}( T -\lambda)$ according to the decomposition $V = \bigoplus_{ \lambda \in \sigma(T)} \mathrm{Ker}( T -\lambda) $, and if $T $ is a self adjoint operator on a Hilbert space $H$, then $\textup{pr}_\lambda = \hat{\textup{pr}}_\lambda $ is the orthogonal projection onto the eigenspace $\mathrm{Ker}(T-\lambda)$ 
\end{corollary}
\begin{remark}
The 'usual' definitions and the Jordan decomposition that we refer to are found for example in \cite[Ch.\,14]{Kersten2} or \cite[Ch.\,4]{Fischer2020}. 
\end{remark}
\begin{proof}[Proof of Lemma \ref{lem: katos_notions}]
It can be shown that \cite[P.\,41]{Kato}
\begin{equation*}
    T = \sum_{\lambda \in \sigma(T)}\lambda \text{pr}_\lambda + D,
\end{equation*}
where $D = \sum_{\lambda \in \sigma(T)} D_\lambda $ is a nilpotent operator that commutes with $\sum_{\lambda \in \sigma(T)}\lambda \text{pr}_\lambda$. Furthermore, the representation $T = S  +  D$, where $S$ is diagonalizable and $D$ is nilpotent and commutes with D is unique \cite[P.\,41-42]{Kato} and the Jordan decomposition \eqref{eq: jordan_decomposition} is also such a representation of $T$ \cite[Ch.\,14.4]{Kersten2}. Thus, by the uniqueness of the representation $T = S + D$ it follows that $\hat{\textup{pr}}_\lambda = \text{pr}_\lambda$, consequently   $ m_\lambda = \dim \text{Im}(\text{pr}_\lambda) = m(\lambda)$ and $M_\lambda = \text{Im}(\text{pr}_\lambda) = V_\lambda$.
\end{proof}
In this work  we will only work with $S_\lambda = S_\lambda(\lambda)$ and we  simply refer to $S_\lambda$ as the reduced resolvent. We can give a more explicit representation of $S_\lambda$, we have 
\begin{lemma}[Alternative definition of the reduced resolvent]
\label{lem: app_reduced_resolvent_other_def}
The reduced resolvent $S_\lambda = S_\lambda(\lambda)$ is given by
 \begin{equation}
 \label{eq: app_reduced_resolvent}
     S_\lambda v = 
 \begin{cases}
   (T - \lambda)|_{ \mathrm{Im}(1- \textup{pr}_\lambda)}^{-1}v  \;; \; v \in \mathrm{Im}(1 - \textup{pr}_\lambda) \\
   0 \;;\; v \in \mathrm{Im}(\textup{pr}_\lambda)
 \end{cases}
 \end{equation}
\end{lemma}
\begin{proof}
The reduced resolvent has the following properties  \cite[P.\,40]{Kato}
\begin{align}
    &S_\lambda \text{pr}_\lambda = \text{pr}_\lambda S_\lambda = 0  \label{eq: reduced_res_property_1}\\
    & (T - \lambda )S_\lambda = S_\lambda(T-\lambda) = 1 - \text{pr}_\lambda \label{eq: reduced_res_property_2}.
\end{align}
Furthermore, note that $T(1- \text{pr}_\lambda) =(1- \text{pr}_\lambda) T$ (follows from \eqref{eq: properties_pr_lambda}), thus 
\begin{equation}
\label{eq: T_preserves_Im(1-pr_lambda)}
    \text{Im}((T- \lambda)|_{\text{Im}(1-\text{pr}_\lambda)}) \subset \text{Im}(1-\text{pr}_\lambda)
\end{equation}
Moreover, \eqref{eq: reduced_res_property_1} implies
\begin{equation}
\label{eq: S_preserves_Im}
    \text{Im}(S_\lambda|_{\text{Im}(1-\text{pr}_\lambda)}) \subset \text{Im}(1-\text{pr}_\lambda)
\end{equation}
Combining \eqref{eq: T_preserves_Im(1-pr_lambda)} and \eqref{eq: S_preserves_Im} with the restriction of \eqref{eq: reduced_res_property_2} to $\text{Im}(1 - \text{pr}_\lambda)$ yields 
\begin{equation}
    S_\lambda(\lambda)|_{\text{Im}(1-\text{pr}_\lambda)} = (T - \lambda)|_{ \mathrm{Im}(1- \text{pr}_\lambda)}^{-1}.
\end{equation}
Finally, \eqref{eq: reduced_res_property_1} implies that $S_\lambda |_{\text{Im}(\text{pr}_\lambda)} = 0 $.
\end{proof}

\begin{remark}(Real reduced resolvent)
\label{rem: app_red_resolvent_real_case}
If $T = \sum_{\lambda \in \sigma(T)} \lambda \text{pr}_\lambda$ is a diagonalizable operator on a real vector space (with real eigenvalues). And $S_\lambda  $ denotes its reduced resolvent  (with respect to $\lambda$) in the sense of Definition \ref{def: reduced_resolvent}, then the reduced resolvent of $T^\F{C}$ with respect to $\lambda$ (in the sense of Lemma \ref{lem: app_reduced_resolvent_other_def}) is just the complexification $S_\lambda^\F{C}$ of $S_\lambda$. Indeed, this is easily checked by using Lemma \ref{lem: app_complexification}$(c)$.
\end{remark}
After we have now given and clarified the basic definitions and notions that are used in \cite{Kato} to analyze the perturbed operator $T(\chi) = T + \chi T',$ we present two results that will be needed for our later purposes.
 \begin{lemma}[Continuous Dependence of Eigenvalues]
 \label{lem: app_eigenvalues_continuous}
 Let $T(r)$ be a continuous operator-valued function \lb on $V$\rb \space  defined on some interval $I \subset \F{R}$. Furthermore, let $N = \dim V$. Then, there are continuous functions $\lambda_k : I \rightarrow \F{C}$, $k = 1,...,n$ such that the $N$-tuple 
 \begin{equation*}
    (\lambda_1(r), ... , \lambda_N(r)).
 \end{equation*}
 represents the eigenvalues of $T(r)$, where the eigenvalues are repeated according to their algebraic multiplicity.
\end{lemma}
 \begin{proof}
 Follows by  a combination of \cite[Ch.\,2.5.1,Thrm.\,5.1]{Kato} and  \cite[Ch.\,2.5.2,Thrm.\,5.2]{Kato}.
 \end{proof}
 \begin{remark}
 As $V$ is finite dimensional the notion of continuity of $T(r)$ does not depend on the underlying operator norm.
 \end{remark}
 \begin{remark}(Real case, Lemma \ref{lem: eigenvalues_continuous})
\label{rem: app_eigenvalues_continuous}
Consider the setting of Lemma \ref{lem: eigenvalues_continuous}. As $r \mapsto T(r)$ is continuous $r \mapsto T(r)^{\F{C}}$ is also continuous. Furthermore as $T(r)$ and $T(r)^\F{C}$ have the same spectral properties (Lemma \ref{lem: app_complexification}$(c)$) the statement of Lemma \ref{lem: eigenvalues_continuous} follows directly.
\end{remark}

\begin{theorem}
\label{th: app_perturbation_simple_eigenvalue}
Let $T : H \rightarrow H$ be a self adjoint operator. Assume that $0$ is a simple eigenvalue of $T$. Let $\lambda_1 := \min_{\lambda \in \sigma(T) \backslash \{0\}} \abs{\lambda}$ be the spectral gap, $\textup{pr}$ the orthogonal projection onto the eigenspace with eigenvalue $0$ and $S$ the corresponding reduced resolvent. Furthermore, let $ T' : H \rightarrow H$ be some operator and define $ T(\chi) := T + \chi T'$ for $\chi \in \F{C}$. Then for all $\abs{\chi} < \frac{\lambda_1}{2 \norm{T}}$ the ball $B_{\frac{\lambda_1}{2}}(0)$ contains exactly one simple eigenvalue $\mu_0(\chi)$ of $T(\chi)$ and we have
\begin{equation}
    \label{eq: theorem_perturbation_mu_0}
    \mu_0(\chi) = \sum_{n = 1}^\infty \mu_0^{(n)}\chi^n,
\end{equation}
with 
\begin{equation}
\label{eq: theorem_perturbation_mu_0_2}
 \mu_0^{(n)} =  \frac{(-1)^n}{n} \sum\limits_{\substack{k_1,..k_n \in \F{Z}_+ \\ k_1 + ... k_n = n-1}} \Tr( T'S^{(k_1)} ... T'S^{(k_n)}),
\end{equation}
where
\begin{equation}
\label{eq: theorem_perturbation_S^(n)}
    S^{(0)} = - \textup{pr} \quad \text{and} \quad S^{(k)} = S^k
\end{equation}
for $ k \geq 1$.
\end{theorem}

As the result is specifically formulated for our later use and \cite{Kato} considers a more general setting we give a short argument for the above theorem by  combining results of \cite[Ch.\,2.1-2.3]{Kato}.
\begin{proof}
Kato treats the more general case of a perturbed operator of the form $T(\chi) = \sum_{n=0}^\infty T^{(n)} \chi^n$ ,  in our case $T^{(0)} = T, T^{(1)} = T'$ and $T^{(k)} = 0$ for $k \geq 2$. In the following let  $\norm{\cdot} = \norm{\cdot}_2 $ denote the operator norm induced by the norm on the Hilbert space $H$. The formula for $\lambda_0(\chi)$ is based on the following formula for the resolvent $R(\zeta, \chi) = (T(\chi)-\zeta)^{-1}$ of $T(\chi)$. Let $R(\zeta)$ denote the resolvent of $T$, and $\zeta \in \F{C} \backslash \sigma(T)$ such that $\norm{\chi T' R(\zeta)} < 1$ (this condition is found in \cite[P.\,88]{Kato}). Then, $R(\zeta, \chi) $ exists and is given by \cite[P.~66-67]{Kato}
\begin{equation}
\label{eq: perturbation_resolvent}
    R(\zeta, \chi) = R(\zeta) \sum_{n=0}^\infty (- \chi T' R(\zeta))^n.
\end{equation}
Furthermore,   using an orthonormal basis that diagonalizes $T$, it is easy to see that $\norm{R(\zeta)} =  \norm{R(\zeta)}_2 = \max _{ \lambda \in \sigma(T)} \abs{\lambda - \zeta}^{-1}$. Thus, invoking submultiplicativity of $\norm{\cdot}$  we get for any $ \abs{\zeta} = \frac{\lambda_1}{2}$
\begin{equation*}
    \norm{T' R(\zeta)}\leq \norm{T'} \frac{2}{\lambda_1}
\end{equation*}
Consequently  for all $\abs{\chi}  < \frac{\lambda_1}{2 \norm{T'}}$ and $\abs{\zeta} = \frac{\lambda_1}{2}$ we have $\norm{\chi T' R(\zeta)} < 1$,  and $R(\chi,\zeta)$ is given by \eqref{eq: perturbation_resolvent}. Let $\Gamma$ be the circle with radius $\frac{\lambda_1}{2}$ around 0 , then for $\abs{\chi} < \frac{\lambda_1}{2\norm{T}}$ the operator
\begin{equation} 
    \label{eq: perturbation_projection}
    \text{pr}(\chi) := \frac{-1}{2 \pi i} \int_\Gamma R(\zeta, \chi) d\zeta 
\end{equation}
depends holomorphically on $\chi$ as the series \eqref{eq: perturbation_resolvent} converges uniformly on $\Gamma$. So by \cite[Ch.1,Lemma 4.10]{Kato}  we have $\mathrm{dim}\mathrm{Im}(\text{pr}(0)) = \mathrm{dim}\mathrm{Im}(\text{pr}(\chi)) = 1 $ for all $\abs{\chi} < \frac{\lambda_1}{2 \norm{T'}}$. Furthermore, $\text{pr}(\chi)$ is the sum of all eigenprojections  for all eigenvalues of $T(\chi)$ lying in $B_{\frac{\lambda_1}{2}}(0)$ (c.f. \cite[P.67]{Kato}), so $1 = \dim \text{pr}(\chi)$  implies that the ball $B_\frac{\lambda_1}{2}(0)$ contains exactly one simple eigenvalue $\lambda_0(\chi)$ of $T(\chi)$ for all  $\abs{\chi} < \frac{\lambda_1}{2\norm{T}}$. Finally, Eqs.\eqref{eq: theorem_perturbation_mu_0}, \eqref{eq: theorem_perturbation_mu_0_2} and \eqref{eq: theorem_perturbation_S^(n)} can be proved by noticing that hat $\lambda_0(\chi) = \Tr(T(\chi)\text{pr}(\chi))$, and using equations \eqref{eq: perturbation_resolvent}, \eqref{eq: perturbation_projection} to calculate $\Tr(T(\chi)\text{pr}(\chi))$. For details of this computation see \cite[Ch.\,2.2]{Kato}. Here, formula \eqref{eq: theorem_perturbation_mu_0} is stated in \cite[Ch.\,2.2.2,P.\,78,Eq.\,(2.21)]{Kato}, formula \eqref{eq: theorem_perturbation_mu_0_2} is obtained by using \cite[Ch.\,2.2.2,P.\,79,Eq.\,(2.31)]{Kato} and $T^{(k)} = 0 $ for $k \geq 2$ (we only treat a linear perturbation $T(\chi) = T + \chi T'$) and formula \eqref{eq: theorem_perturbation_S^(n)} is stated in \cite[Ch.\,2.2.1,P.\,76,Eq.\,(2.10)]{Kato}, where we also use $\text{pr}(0) = \text{pr}$ (Corollary \ref{cor: selfadjoint case}).

\end{proof}

\begin{remark}(Transfer to real operators, Theorem \ref{th: perturbation_simple_eigenvalue})
\label{rem: app_perturbation_simple_eigenvalue}
Consider the setting and notation of Theorem \ref{th: perturbation_simple_eigenvalue}. Transfer the result of Theorem \ref{th: app_perturbation_simple_eigenvalue} to the real case by using the complexification $T^\F{C}(r) = T^\F{C} + r T'^{\F{C}}$: Because of the properties of the complexification (Lemma \ref{lem: app_complexification}$(c)$), $T^\F{C}$ satisfies the requirements of  Theorem \ref{th: app_perturbation_simple_eigenvalue}. Furthermore, because the spectral properties of $T$ and $T^\F{C}$ are the same (Lemma \ref{lem: app_complexification}$(c)$), it does not matter whether we consider the eigenvalues of $T(r)^\F{C}$ or $T(r)$. Thus, Theorem \ref{th: app_perturbation_simple_eigenvalue} implies for all $\abs{r} < \frac{\lambda_1}{2 \norm{T'}}$ the ball  $B_{\frac{\lambda_1}{2}}(0) \subset \F{C}$  contains exactly one simple eigenvalue $\mu_0(r)$ of $T(r)$ and \eqref{eq: mu_0(r)} with  
\begin{equation}
    \mu_0^{(n)} =  \frac{(-1)^n}{n} \sum\limits_{\substack{k_1,..k_n \in \F{Z}_+ \\ k_1 + ... k_n = n-1}} \Tr( (T')^\F{C}\tilde{S}^{(k_1)} ... (T')^\F{C}\tilde{S}^{(k_n)}),
\end{equation}
where
\begin{equation}
    \tilde{S}^{(0)} = - \widetilde{\text{pr}} \quad \text{and} \quad \tilde{S}^{(k)} = \tilde{S}^k
\end{equation}
for $k \geq 1$.
Here $\widetilde{\textup{pr}}$ denotes the (complex) orthogonal projection onto $\text{Ker}( T^\F{C})$ and $\tilde{S}$ denotes the (complex) reduced resolvent  (see Lemma \ref{lem: app_reduced_resolvent_other_def}) of $T^\F{C}$ with respect to the eigenvalue $0$. But as the complexification respects orthogonal projections and diagonalizations (Lemma \ref{lem: app_complexification}$(c)$), and also reduced resolvents (Remark \ref{rem: app_red_resolvent_real_case}) we have
\begin{equation}
     \widetilde{\text{pr}} = \text{pr}^\F{C} \quad \text{and} \quad \tilde{S} = S^{\F{C}}.
\end{equation}
Consequently, $\tilde{S}^{(k)} = (S^{(k)})^\F{C}$ and
.\begin{align}
\begin{split}
    \mu_0^{(n)} &=  \frac{(-1)^n}{n} \sum\limits_{\substack{k_1,..k_n \in \F{Z}_+ \\ k_1 + ... k_n = n-1}} \Tr( (T')^\F{C}(S^{(k_1)})^{\F{C}} ... (T')^\F{C}(S^{(k_n)})^{\F{C}}) \\
    &= \frac{(-1)^n}{n} \sum\limits_{\substack{k_1,..k_n \in \F{Z}_+ \\ k_1 + ... k_n = n-1}} \Tr(T'S^{(k_1)} ... T'^{(k_n)}),
\end{split}
\end{align}
where we used Lemma \ref{lem: app_complexification}$(c)$. Thus, Theorem \ref{th: perturbation_simple_eigenvalue} is a consequence of the corresponding complex version Theorem \ref{th: app_perturbation_simple_eigenvalue}.

\end{remark}

\subsection{Remark \ref{rem: th_main_concentration_inequality}}
\begin{lemma}[Properties of $\lambda_0^*$]
\label{lem: appendix_properties_lambda_0^*}
Consider the setting of Theorem \ref{th: main_conc_inequality}. The Fenchel conjugate \\ $\lambda_0^*(u) = \sup_{ r \in \F{R}} (ru - \lambda_0(r))$ has the following properties :
\begin{enumerate}[$\lb a\rb$]
        \item $\lambda_0^* \geq 0$
        \item $\{ \lambda_0^* < \infty \} = [\min_{x \in E} f(x), \max_{x \in E}f(x)]$
        \item $\lambda_0^*$ is  convex, bounded and continuous on $\{ \lambda_0^* < \infty \}$, and nondecreasing on $[0,\infty)$.
        
    \end{enumerate}
\end{lemma}
\begin{proof}
Part $(a)$ follows directly from  $\lambda_0(0) = 0$ (Lemma \ref{lem: properties_infinitesimal_generator}) and $\lambda_0^*(u) \geq 0 \cdot u - \lambda_0(0) = 0$. For part $(b)$ recall that by Theorem \ref{th: main_conc_inequality} we have \begin{equation}
    \lambda_0^*(u) = I(u) =  \inf \{ - \langle Lg, g \rangle \; | \, \norm{g}_2 = 1, \langle M_fg, g \rangle = u \}.
\end{equation}
Furthermore, recall the proof of Theorem \ref{th: main_conc_inequality}: Let $D : S_1 \rightarrow \F{R}$ be defined by
\begin{equation}
    D(g) = \langle M_fg, g \rangle = \sum_{y \in E} g(y)^2 f(y) \pi_y,
\end{equation}
 where $S_1 = \{ g \in L^2(\pi) \, | \, \norm{g}_2 = 1 \}$. By \eqref{eq: I_smaller_infty} we have
\begin{equation}
    \{ \lambda_0^* < \infty \} = \{ I < \infty \}  = \text{Im}(D) = [a,b],
\end{equation}
for some $a < b $. Note that for $g \in S_1$ we have
\begin{align}
\begin{split}
\min_{x \in E} f(x) &= \sum_{y \in E} \min_{x \in E} f(x) g(y)^2 \pi_y \leq \sum_{y \in E} f(y) g(y)^2 \pi_y  \\
 &\leq \sum_{y \in E}  \max_{ x \in E}f(x) g(y)^2 \pi_y = \max_{x \in E} f(x).
\end{split}
\end{align}
Furthermore, if $g(x) = \delta_{yx} \frac{1}{\sqrt{\pi_y}}$ for some $y \in E$, then $f(y) = \langle g, M_f g \rangle$, so
\begin{equation}
\text{Im}(D) = [\min_{x \in E}f(x) , \max_{x \in E}f(x)],
\end{equation}
which proves part $(b)$. As a Fenchel conjugate, $\lambda_0^*$ is convex (Lemma \ref{lem: properties_Legendre_transform}$(a)$). Moreover, $\lambda_0^* = I$ is bounded on $\{ \lambda_0^* < \infty \}$ because by the Cauchy Schwartz inequality $I(u) \leq \norm{L}_2$ for any $u \in \{ \lambda_0^* < \infty \} $. To show the continuity of $\lambda_0^*$ on $ \{ \lambda_0^* < \infty \} = [a,b]$ let $u \in [a,b]$ and let $(u_n)_{n \in \F{N}}$ be a sequence in $[a,b]$ converging to $u$. Recall from the proof of Theorem \ref{th: main_conc_inequality} that $I = \lambda_0^*$ is lower semicontinuous, so 
\begin{equation}
\label{eq: appendix_lower_semicont}
    \lambda_0^*(u)  \leq \liminf_{n \to \infty } \lambda_0^*(u_n) \leq \limsup_{n \to \infty } \lambda_0^*(u_n).
\end{equation}
 Choose a subsequence $(u_{n_k})_{k \in \F{N}}$ such  that $\lim_{k \to \infty }\lambda_0^*(u_{n_k}) = \limsup_{n \to \infty} \lambda_0^*(u_n)$ and  $(u_{n_k})_{k \in \F{N}}$ is monotonous. Then, every $u_{n_k}$ can be expressed as $u_{n_k} = (1-s_k) u + s_ku_{n_1}$ for some $s_k \in [0,1]$. As $u_{n_k} \to u$, we have $s_k \to 0$. Consequently convexity and $\lambda_0^*(u_{n_1}) < \infty $ implies
 \begin{equation}
    \limsup_{n \to \infty } \lambda_0^*(u_n) =  \lim_{k \to \infty} \lambda_0^*(u_{n_k}) \leq \lim_{k \to \infty } (1-s_k) \lambda_0^*(u) + s_k \lambda_0^*(u_{n_1})  = \lambda_0^*(u).
 \end{equation}
Combining this with \eqref{eq: appendix_lower_semicont} implies $\lambda_0^*(u) = \lim_{n \to \infty} \lambda_0^*(u_n)$, which proves continuity. Finally, we show that $\lambda_0^*$ is nondecreasing on $[0,\infty)$. Let $0 \geq u_1 < u_2$ and let $\varepsilon > 0 $. By Theorem \ref{th: main_conc_inequality} we have
\begin{equation}
    \lambda_0^*(u_1) = \sup_{ r \in \F{R}_{\geq 0}} (ru_1 - \lambda_0(r)),
\end{equation}
so there is some $r \geq 0$ with $ (ru_1 - \lambda_0(r)) \geq \lambda_0^*(u_1) - \varepsilon$. This implies
\begin{equation}
    \lambda_0^*(u_2) \geq r(u_2-u_1) + ru_1 - \lambda_0(r) \geq  \lambda_0^*(u_1) - \varepsilon.
\end{equation}
 As the above inequality holds for all $\varepsilon > 0$, we have $\lambda_0^*(u_2) \geq \lambda_0^*(u_1)$ and the proof is complete.
\end{proof}

\begin{lemma}[Cramérs theorem for symmetric MJPs]
\label{lem: appendix_asymptotic_sharp}
Assume the setting of Section \ref{subsec: application_cramer_chernoff}. Furthermore, assume that $\pi$ satisfies the detailed balance condition. Then,

\begin{equation}
\label{eq: appendix_exact_asymptotic}
  \lim_{ t \to \infty} \frac{ \log \F{P}_\nu  \left (\frac{A_t}{t}  > u \right )}{t} = - \lambda_0^*(u) 
\end{equation}
for all $u \geq 0$ and $u \neq \max_{x \in E} f(x)$.
\end{lemma}
\begin{proof}
If \hfill $u > \max_{x \in E} f(x)$, \hfill then \hfill the \hfill claim follows \hfill directly \hfill by \hfill the  \hfill estimate \\ $A_t \leq t \max_{x \in E} f(x)$ and $\lambda_0^*(u) = \infty $ (Lemma \ref{lem: appendix_properties_lambda_0^*}). So assume $0 \leq u < \max_{x \in E } f(x) $. We apply \cite[Ch.\,5.3., Thrm.\,5.3.10, Eq.\,(5.3.12)]{large_deviations_stroock}. The context and notation of this theorem are the following: We endow the set of probability measures  $\mathcal{M}_1(E)$ on $E$ with the so called $\tau$-topology (see \cite[Ch.\,3.2, P.\,64]{large_deviations_stroock} for a definition), i.e. the topology generated by the maps 
\begin{equation}
    \nu \mapsto \nu(g) 
\end{equation}
for $g \in \mathcal{B}(E)$ and $\nu \in \mathcal{M}_1(E)$. Furthermore, $L_t = t^{-1}\int_0^t \delta_{X_s} ds = \sum_{x \in E} \delta_x \cdot  t^{-1} \int_0^t 1_{\{ X_s = x\}}ds$ denotes the 'empirical measure', $m$ denotes the invariant measure $\pi$ and $J_{\mathcal{E}}(\nu)  $ denotes the Donsker-Varadhan information $I(\nu | \pi)$ defined in \eqref{eq: def_donsker_varadhan_info}. Notice that $L_t(f) = t^{-1}A_t$ and $J_{\mathcal{E}}(\nu) =  - \left \langle L \left(\frac{d\nu}{d\pi} \right)^{\frac{1}{2}}, \left (\frac{d\nu}{d\pi} \right)^{\frac{1}{2}} \right\rangle$ (as we always have $\nu << \pi$). Furthermore, set $\Gamma = \{ \nu \in \mathcal{M}_1(E) \; | \; \nu(f) >  u \}$. Now, we prove that all requirements of \cite[Ch.\,5.3, Thrm.\,5.3.10, Eq.\,(5.3.12)]{large_deviations_stroock} are satisfied, i.e. we show that $J_\mathcal{E}$ is a good rate function (in the sense of \cite[Ch.\,2.2, P.\,33]{large_deviations_stroock}), $J_\mathcal{E}(\nu) = 0$ if and only if $\nu = m = \pi$, and $\nu$ is not singular to $m$. By using the canonical identifications
\begin{equation}
    \mathcal{M}_1(E) \Tilde{=} \left \{ \nu \in \F{R}^E \; | \; \nu_x \in [0,1] \text{ for all }x \in E, \sum_{x \in E} \nu_x = 1 \right \}  =: M
\end{equation}
and 
\begin{equation}
    \mathcal{B}(E) = \F{R}^E
\end{equation}
it is straightforward to see that the $\mathcal{M}_1(E)$ endowed with the $\tau$-topology is homeomorphic to $M \subset \F{R}^E$, endowed with the subspace topology of the standard topology on $\F{R}^E \Tilde{=} \F{R}^N$, where $N = \# E$. Using this homeomorphism  and $\frac{d\nu}{d\pi}(x) = \frac{\nu_x}{\pi_x}$, it is straightforward to check that $\nu \mapsto J_{\mathcal{E}}(\nu)$ is continuous and $\mathcal{M}_1(E)$ is compact. Thus, $J_{\mathcal{E}} $ is a good rate function (see \cite[Ch.\,2.2, P.\,32-33]{large_deviations_stroock} for a definition). Furthermore, using that $-L$ is positive semidefinite (Lemma \ref{lem: properties_infinitesimal_generator}) and selfadjoint (Theorem \ref{th: characterization_detailed_balance}) with $\text{Ker}(L) = \text{span}(\textbf{1})$ (Lemma \ref{lem: properties_infinitesimal_generator}), by diagonalizing $-L$, it is straightforward to check  that $J_{\mathcal{E}}(\nu) = 0$ if and only if $\nu = m = \pi$. Finally, $\nu$ is not singular to $m$ as $\nu << m$ and consequently all requirements are satisfied. As $u < \max_{x \in E} f(x)$, $\Gamma$ is open (with respect to the $\tau$-topology) and non-empty. By the continuity of $J_{\mathcal{E}}$ we have
\begin{equation}
\label{eq: appendix_same_infima}
    \inf_{\nu \in \Gamma } J_{\mathcal{E}}(\nu) =  \inf_{\nu \in \overline{\Gamma} } J_{\mathcal{E}}(\nu).
\end{equation}
 Using \cite[Ch.\,5.3, Thrm.\,5.3.10, Eq.\,(5.3.12)]{large_deviations_stroock} 
 \begin{equation}
     - \inf_{\nu \in \Gamma} I(\nu|\pi) \leq    \liminf_{ t \to \infty} \frac{ \log \F{P}_\nu  \left (\frac{A_t}{t} \geq u \right )}{t} \leq \limsup_{ t \to \infty} \frac{ \log \F{P}_\nu  \left (\frac{A_t}{t} \geq u \right )}{t} \leq  - \inf_{\nu \in \overline{\Gamma}} I(\nu|\pi)
 \end{equation}
 and \eqref{eq: appendix_same_infima} it follows that 
\begin{equation}
    - \inf_{\nu \in \Gamma} I(\nu|\pi) =   \lim_{ t \to \infty} \frac{ \log \F{P}_\nu  \left (\frac{A_t}{t} \geq u \right )}{t} = - \inf_{\nu \in \overline{\Gamma }} I(\nu|\pi).
\end{equation}
To finish the proof, note that  
\begin{equation}
\label{eq: appendix_sandwich_prep}
  - \inf_{\nu \in \overline{\Gamma} } I(\nu|\pi) =   \lim_{ t \to \infty} \frac{ \log \F{P}_\nu  \left (\frac{A_t}{t} \geq u \right )}{t} \leq - \lambda_0^*(u) = -\inf \{  I(\nu|\pi) \; | \; \nu(f) = u \},
\end{equation}
where we used Theorem \ref{th: main_conc_inequality} and Lemma \ref{lem: information_Legendre_transform_lambda_0(r)}.
Finally, note that (by definition) $\overline{\Gamma} = \{ \nu \in \mathcal{M}_1(E) \,| \nu(f) \geq u \}$, thus  $-\inf \{  I(\nu|\pi) \; | \; \nu(f) = u \} \leq  - \inf_{\nu \in \overline{\Gamma} } I(\nu|\pi)$ (the infimum over a bigger set is smaller), which together with \eqref{eq: appendix_sandwich_prep} implies the desired result \eqref{eq: appendix_exact_asymptotic}.
\end{proof}

\subsection{Proof of Lemma \ref{lem: lezaud_bound_lambda_0(r)}} 
\begin{lemma}
\label{lem: appendix_equivalence_classes}
Let $n ,k \in \F{N}$, $M = \{ (k_1,...,k_n) \in \F{Z}_+^n \;|\; k_1 + ... + k_n = k \}$ and let $G = \{ (1 \; 2 \;... \; n),..., (1 \; 2 \;... \; n)^{n-1}, \text{Id} \} $ be the \lb abelian\rb \space  group of all circular permutations of $n$ elements \lb we used cycle notation\rb. Furthermore, define a  \lb right\rb \space  group action via $(k_1,...,k_n) \cdot \sigma  = (k_{\sigma(1)},...,k_{\sigma(n)})$ and let $G_{(k_1,...,k_n)}$ denote the stabilizer of $(k_1,...,k_n)$, i.e.
\begin{equation}
    G_{(k_1,...,k_n)} = \{ g \in G \; |  \ (k_1,...,k_n) \cdot g = (k_1,...,k_n) \} 
\end{equation}
Then, $\# G_{(k_1,...,k_n)} \leq \textup{GCD}(n,k) $, where $\textup{GCD}(n,k) $ denotes the greatest common divisor. Moreover, if $ \textup{GCD}(n,k) = 1$ then all the equivalence classes \lb defined by the orbits of the group action\rb  \space  $[k_1,...,k_n] = (k_1,...,k_n) \cdot G$  contain exactly $n$ elements.
\end{lemma}
\begin{proof}
The \hfill statement \hfill $\# G_{(k_1,...,k_n)} \leq \text{GCD}(n,k) $ \hfill is \hfill equivalent \hfill to: \hfill 
$\# G_{(k_1,...,k_n)} | n $ \hfill and   \\ $\# G_{(k_1,...,k_n)} | k$. Let $H \subset G$ be some subgroup. By Lagrange's theorem (\cite[Ch.\,1.2, Kor.\,3]{Bosch2020}) $\# H | \#G = n$. Furthermore, as $G $ is a cyclic group (it is  generated by $(1 \; 2 \;... \; n)$), every subgroup $H$ is also cyclic: It is generated by some $(1 \; 2 \;... \; n)^{j}$ where $j | n$, $j \in \{ 1,..., n \}$ and $\#H = \frac{n}{j}$ (c.f. \cite[Ch.\,1.3, Ex.\,2]{Bosch2020} and \cite[P.\,333-334]{Bosch2020} for a solution). Let  $G_{(k_1,...,k_n)}$ be generated by $(1 \; 2 \;... \; n)^{j}$. By the definition of the stabilizer, this implies $k_i = k_{i\;\text{mod}\;j}$  for any $i \in \{1,...,n\}$, where $i\;\text{mod}\;j $ is chosen such that $i\;\text{mod}\;j \in \{1,...,j\}$. Consequently, using $j | n$ we obtain
\begin{equation*}
    k = \sum_{i = 1}^n k_i = \frac{n}{j} \sum_{i=1}^{j}k_j = \# G_{(k_1,...,k_n)} \cdot \sum_{i=1}^{j}k_j.
\end{equation*}
Thus, we also have $\#G_{(k_1,...,k_n)} | k$. Finally, the other statement $\#G_{(k_1,...,k_n)} | n$ follows from the orbit stabilizer theorem (see \cite[Ch.\,5.1, Remark\,6]{Bosch2020}): We have
\begin{equation*}
   \#[k_1,...,k_n] = \# (k_1,...,k_n) \cdot G = \frac{\#G}{\#G_{(k_1,...,k_n)}} = \frac{n}{\#G_{(k_1,...,k_n)}}.
\end{equation*}
For the second part of the Lemma, note that if $\text{GCD}(n,k) =  1$, then $\#G_{(k_1,...,k_n)} = 1$ and thus $\#[k_1,...,k_n] = n$ (by the first part and the orbit stabilizer theorem).
\end{proof}

\begin{lemma}
\label{lem: appendix_beta(n,m)}
Let $n \in \F{N}$. Consider the same setup as in Lemma \ref{lem: appendix_equivalence_classes} with $n = n$ and $k = n-1$, i.e. $M = \{ (k_1,...,k_n) \in \F{Z}_+^n \;|\; k_1 + ... + k_n = n-1 \}$. Furthermore, let $1 \leq m \leq \lfloor \frac{n}{2} \rfloor$. Then, the number $\beta(n,m)$ of equivalence classes $[k_1,...,k_n] \in M/G$ with exactly $m$ non-adjacent zeros is given by
\begin{equation}
\label{eq: appendix_lemma_beta(n,m)}
    \beta(n,m) = \frac{1}{n-1} \binom{n-1}{m} \binom{n-1-m}{n-2m}.
\end{equation}
\end{lemma}
\begin{proof}
As a formal, detailed proof tedious and not the main focus of this work, we leave some technical details to the reader and use intuitive arguments.
Let $A = \{ [k_1,...,k_n] \in M/G \, | \, [k_1,...,k_n] \text{ has exactly $m$ non-adjacent zeros} \}$. Note that by definition $\beta(n,m) = \# A$. Let $\rho = \exp(\frac{2 \pi i}{n-1})$, so that $1, \rho,..., \rho^{n-2} \in S^1$ denote the $n-1$-th primitive roots, where $S^1$ denotes the unit circle. To calculate $\#A$ we  identify uniquely and element $[k_1,...,k_n] \in A$ with a partition of the circle $S^1$ (more precisely a partition modulo  $\frac{2\pi}{n-1}$ rotation) into $n$ boxes as follows: Denote by $(1, \circ, \rho, \circ , ..., \rho^{n-2}, \circ , \rho^{n-1})$ the circle, where $\circ $ refer to the (curved) intervals $(\rho^k, \rho^{k+1})$ and by definition $\rho^{n-1} = 1$. For intuition, it is useful to imagine that the tuple is glued together at both ends. Using this notation we now define the partition of $S^1$ into boxes that contain the primitive roots. Denote by $|$ the edges of the boxes, and let $(1,b_1, \rho,b_2 , ...,\rho^{n-2},b_{n-1} ,\rho^{n-1})$ for $b_i \in \{\circ, |, ||, |||, ....\} \cong \F{N}_0$ denote a partition of $S^1$ into $\sum_{i = 1}^{n-1}b_i$ boxes  (the number of boxes equals the total number of edges $\sum_{i = 1}^{n-1} b_i$). These boxes contain in total the $n-1$ primitive roots $1,\rho,...,\rho^{n-2}$. Hereby, $||$ is interpreted as an empty box. For sake of notation we identify $(1,b_1, \rho,b_2 , ...,\rho^{n-2},b_{n-1} ,\rho^{n-1}) = (b_1,...,b_{n-1})$. Recall that for $[k_1,...,k_n] \in A$ we have $\sum_{i = 1}^n k_i = n-1$, so we want to identify $[k_1,...,k_n]$  with a partition of $S^1$ into $n$ boxes, where the corresponding boxes contain  exactly $k_i$ of the $n-1$ primitive roots. As $[k_1,...,k_n]$ is invariant under circular permutations this identification should by invariant, with respect to the $\frac{2\pi}{n-1}$ rotation of $S_1$ (which corresponds to a circular shifting of the boxes). Thus, we define equivalence classes $[b_1,...,b_{n-1}]$ via the (right) group action 
\begin{equation}
    (b_1,...,b_{n-1}) \cdot \sigma = (b_{\sigma(1)},...,b_{\sigma(n-1)}),
\end{equation} 
where $\sigma \in  \{ (1 \; 2 \; ... \; n-1),..., (1 \; 2 \;... \; n -1)^{n-2}, \text{Id} \} $ is a circular permutation of $n-1$ elements. This corresponds to a circular shifting of the boxes. As $[k_1,...,k_n]$ has exactly $m$ non-adjacent zeros, the corresponding partition $[b_1,....,b_{n-1}]$ must satisfy $b_i \in \{\circ,|,|| \} \cong \{0,1,2\} $ (otherwise there would be at least two adjacent empty boxes) and $m = \#\{ b_i \, |\,  b_i = ||, \, i = 1,...,n-1 \} $ (which corresponds to the $m$ zeros). Thus, we just consider 
\begin{equation}
    B =  \{(b_1,...,b_n) \in \{0,1,2\}^{n-1} \, | \,\sum_{i=1}^{n-1}b_i = n, \#\{ b_i \, |\,  b_i = ||, \, i = 1,...,n-1 \} = m\}
\end{equation}
and 
\begin{equation}
    B/H = \{[b_1,...,b_{n-1}] \, | \, (b_1,...,b_{n-1}) \in B \},
\end{equation}
where $H = \{ (1 \; 2 \; ... \; n-1),..., (1 \; 2 \;... \; n -1)^{n-2}, \text{Id} \} $. Note that there is a one to one correspondence between $B/H$ and $A$: Each $[k_1,...,k_n] \in A$ defines a unique $[b_1,...,b_{n-1}] \in B/H$ and vice versa (we leave the detailed proof of this intuitive statement to the reader). Intuitively, this is seen by imagining the boxes and the primitive roots $\rho^k$ on a circle and using the invariance with respect to circular shifts of the boxes. Hereby a box containing $l$ primitive roots corresponds to some $k_i = l$. Thus, 
\begin{equation}
    \#A = \#B/H.
\end{equation}
Finally, \eqref{eq: appendix_lemma_beta(n,m)} follows by the above identification and Lemma \ref{lem: appendix_equivalence_classes}: Let $(b_1,...,b_{n-1}) \in B$, then  $ \#\{ b_i \, |  b_i = 2, \, i = 1,...,n-1 \} = m$ and $\sum_{i = 1}^{n-1}b_i = n$ implies also $ \#\{ b_i \, |  b_i = 1, \, i = 1,...,n-1 \} = n-2m$ . Thus, a simple combinatorial argument shows (there are $\binom{n-1}{m}$ positions to choose the zeros and $\binom{n-1-m}{n-2m}$ remaining positions to chose the ones)
\begin{equation}
    \#B = \binom{n-1}{m} \binom{n-1-m}{n-2m}
\end{equation}
Furthermore, we always have $\text{GCD}(n,n-1) = 1$ so Lemma \ref{lem: appendix_equivalence_classes} (with $n = n-1$ and $k = n$) implies that $\# [b_1,...,b_{n-1}] = n-1 = \# H$ for any $[b_1,...,b_{n-1}] \in B/H$ so 
\begin{equation}
    \beta(n,m) = \#B/H = \frac{1}{n-1} \#B = \frac{1}{n-1} \binom{n-1}{m} \binom{n-1-m}{n-2m}.
\end{equation}
\end{proof}

\begin{lemma}
\label{lem: appendix_phi(x)}
Let $\beta(n,m)$ be defined as in Lemma \ref{lem: appendix_beta(n,m)} and define
\begin{equation}
    \beta_n =\sum_{m = 1}^{\lfloor \frac{n}{2} \rfloor} \beta(n,m).
\end{equation}
Then, the series 
\begin{equation}
 \sum_{n = 2}^\infty \beta_n x^n
\end{equation}
has convergence radius $\frac{1}{3}$ and
\begin{equation}
    \sum_{n = 2}^\infty \beta_n x^n = \Phi(x),
\end{equation}
where $\Phi(x) = \left ( \frac{1-x}{2} \right) \left( 1-  \sqrt{1-\frac{4x^2}{(1-x)^2}} \right) $.
\end{lemma}
\begin{proof}
Let $(m_n)_{n \in \F{Z}_+}$ be the Motzkin numbers defined by the quadratic equation (c.f \cite[Eq.(1)]{motzkin})
\begin{equation}
\label{eq: appendix_motzkin_quadratic_eq}
    m(x) = 1 + xm(x) + (xm(x))^2 = \sum_{n=0}^\infty m_n x^n,
\end{equation}
more precisely defined by the analytic branch of the two branches of solution of the above equation. An elementary calculation shows that
\begin{equation}
   m(x) = \frac{1-x - \sqrt{(1-x)^2 - 4 x^2}}{2x^2},
\end{equation}
for $x \in [-1,\frac{1}{3}]$ (the range where the quadratic equation \eqref{eq: appendix_motzkin_quadratic_eq} has real solutions). Note that the function $m(x)$ is analytic at $0$ ; we have 
\begin{equation}
\label{eq: app_motzkin_series}
    m(x) = \sum_{n = 0}^\infty m_n x^n,
\end{equation}
where the convergence radius of the above series is $\frac{1}{3}$. This can be seen as follows. It is well known, that the root function $\sqrt{y}$ has a holomorphic extension onto the sliced plane $\F{C}\backslash \F{R}_{\leq 0}$. Furthermore, an elementary calculation shows that the complex polynomial $p(z) = (1- z)^2 - 4 z^2$, satisfies $\text{Re}(p(z)) > 0$ for $\abs{z} < \frac{1}{3}$, thus the numerator $F(x) = 1- x - \sqrt{p(x)}$ has a holomorphic extension onto the open ball $B_{\frac{1}{3}}(0) \subset \F{C}$  with radius $\frac{1}{3}$ (using the holomorphic extension of the root). Thus, as all holomorphic functions are analytic we can write $F(x) = \sum_{n = 0}^\infty F_n x^n$ for $\abs{x} < \frac{1}{3}$. But , $F(0) = F'(0) = 0$, so $F_0 = F_1 = 0$ and consequently  we can write $m(x) = \frac{F(x)}{2x^2}$ as in \eqref{eq: app_motzkin_series} for $\abs{x} < \frac{1}{3}$. The convergence radius of the series \eqref{eq: app_motzkin_series} cannot be bigger than $\frac{1}{3}$, because otherwise the function $x^2 m(x) - 1 + x = \sqrt{p(x)}$ would have an holomorphic extension onto some ball $B_{\frac{1}{3}+\varepsilon}(0)$, which is not possible. That this is not possible follows for example from the fact that $\lim_{x \uparrow \frac{1}{3}} \frac{d}{dx} \sqrt{p(x)} = \infty$ (which would have to exists if $\sqrt{p(x)}$ had an holomorphic extension). We show now that 
\begin{equation}
    \beta_{n+2} = m_n
\end{equation}
for all $n \in \F{Z}_+$. For that we use the identity (see \cite[Eq.\,(3)]{motzkin}):
\begin{equation}
    m_n =  \sum_{m = 0}^\infty \binom{n}{2m} \frac{(2m)!}{m! (m+1)!}.
\end{equation}
As $\binom{n}{k} = 0$ if $k> n$ it follows that
\begin{equation}
\label{eq: appendix_m_n}
    m_n = \sum_{m = 0}^{\lfloor \frac{n}{2} \rfloor } \binom{n}{2m} \frac{(2m)!}{m! (m+1)!}
\end{equation}
By definition
\begin{equation}
\label{eq: appendix_beta_n+2}
    \beta_{n+2} = \sum_{m = 1}^{\lfloor \frac{n+2}{2} \rfloor} \beta(n+2, m) = \sum_{m = 1}^{\lfloor \frac{n}{2} \rfloor + 1} \beta(n+2, m)
    = \sum_{m = 0}^{\lfloor \frac{n}{2} \rfloor} \beta(n+2, m+1)
\end{equation}
and (c.f \eqref{eq: appendix_lemma_beta(n,m)})
\begin{align}
    \begin{split}
      &\beta(n+2,m+1) = \frac{1}{n+1} \binom{n+1}{m+1} \binom{n -m}{n - 2m } \\ 
      &= \frac{n!}{(m+1)!m! (n-2m)!} 
      = \binom{n}{2m} \frac{(2m)!}{m! (m+1)!},
    \end{split}
\end{align}
where in the last two equalities $\binom{l}{k} =  \frac{l!}{k! (l-k)!}$ was used. Combining the above equality  with \eqref{eq: appendix_m_n}  and \eqref{eq: appendix_beta_n+2} yields $\beta_{n+2} = m_n$, which implies
\begin{equation}
    \sum_{n = 2}^\infty \beta_n x^n = \sum_{n = 0}^\infty m_nx^{n+2} = x^2 m(x)^2.
\end{equation}
Thus, these series have the same convergence radius  $\frac{1}{3}$. Finally, an elementary calculation shows that $\Phi(x) = x^2m(x)$ for $0 \leq x \leq \frac{1}{3}$.
\end{proof}

%% file: Literatur.bib
@Article{horowitz,
  title = {Dissipation Bounds All Steady-State Current Fluctuations},
  author = {Gingrich, Todd R. and Horowitz, Jordan M. and Perunov, Nikolay and England, Jeremy L.},
  journal = {Phys. Rev. Lett.},
  volume = {116},
  issue = {12},
  pages = {120601},
  numpages = {5},
  year = {2016},
  month = {Mar},
  publisher = {American Physical Society},
  doi = {10.1103/PhysRevLett.116.120601},
  url = {https://link.aps.org/doi/10.1103/PhysRevLett.116.120601}
}

@Article{tur_1,
  title = {Thermodynamic Uncertainty Relation for Biomolecular Processes},
  author = {Barato, Andre C. and Seifert, Udo},
  journal = {Phys. Rev. Lett.},
  volume = {114},
  issue = {15},
  pages = {158101},
  numpages = {5},
  year = {2015},
  month = {Apr},
  publisher = {American Physical Society},
  doi = {10.1103/PhysRevLett.114.158101},
  url = {https://link.aps.org/doi/10.1103/PhysRevLett.114.158101}
}

@Article{free_energy,
title = {Efficient estimation of free energy differences from Monte Carlo data},
journal = {Journal of Computational Physics},
volume = {22},
number = {2},
pages = {245-268},
year = {1976},
issn = {0021-9991},
doi = {https://doi.org/10.1016/0021-9991(76)90078-4},
url = {https://www.sciencedirect.com/science/article/pii/0021999176900784},
author = {Charles H Bennett},
}

@Article{Chorin2002,
author={Chorin, Alexandre J.
and Hald, Ole H.
and Kupferman, Raz},
title={Optimal prediction with memory},
journal={Physica D: Nonlinear Phenomena},
year={2002},
volume={166},
number={3},
pages={239-257},
issn={0167-2789},
url={https://www.sciencedirect.com/science/article/pii/S0167278902004463}
}

@Article{schuette,
author={Sarich, Marco
and No{\'e}, Frank
and Sch{\"u}tte, Christof},
title={On the Approximation Quality of Markov State Models},
journal={Multiscale Modeling {\&} Simulation},
year={2010},
publisher={Society for Industrial and Applied Mathematics},
volume={8},
number={4},
pages={1154-1177},
issn={1540-3459},
doi={10.1137/090764049},
url={https://doi.org/10.1137/090764049}
}

@Article{Schuette2015,
author={Sch{\"u}tte, Ch.
and Sarich, M.},
title={A critical appraisal of Markov state models},
journal={The European Physical Journal Special Topics},
year={2015},
volume={224},
number={12},
pages={2445-2462},
issn={1951-6401},
doi={10.1140/epjst/e2015-02421-0},
url={https://doi.org/10.1140/epjst/e2015-02421-0}
}

@article{Maes_2008,
	doi = {10.1209/0295-5075/82/30003},
	url = {https://doi.org/10.1209/0295-5075/82/30003},
	year = 2008,
	month = {apr},
	publisher = {{IOP} Publishing},
	volume = {82},
	number = {3},
	pages = {30003},
	author = {C. Maes and K. Neto{\v{c}}n{\'{y}}},
	title = {Canonical structure of dynamical fluctuations in mesoscopic nonequilibrium steady states},
	journal = {{EPL} (Europhysics Letters)},
}

@article{kaiser,
author={Kaiser, Marcus
and Jack, Robert L.
and Zimmer, Johannes},
title={Canonical Structure and Orthogonality of Forces and Currents in Irreversible Markov Chains},
journal={Journal of Statistical Physics},
year={2018},
month={Mar},
day={01},
volume={170},
number={6},
pages={1019-1050},
issn={1572-9613},
url={https://doi.org/10.1007/s10955-018-1986-0},

}

@article {non_ergodic1,
	Title = {Single-molecule studies highlight conformational heterogeneity in the early folding steps of a large ribozyme},
	Author = {Xie, Zheng and Srividya, Narayanan and Sosnick, Tobin R and Pan, Tao and Scherer, Norbert F},
	Number = {2},
	Volume = {101},
	Month = {January},
	Year = {2004},
	Journal = {Proceedings of the National Academy of Sciences of the United States of America},
	ISSN = {0027-8424},
	Pages = {534—539},
	URL = {https://europepmc.org/articles/PMC327182},
}

@Article{non_ergodic2,
author={Ye, Weixiang
and G{\"o}tz, Markus
and Celiksoy, Sirin
and T{\"u}ting, Laura
and Ratzke, Christoph
and Prasad, Janak
and Ricken, Julia
and Wegner, Seraphine V.
and Ahijado-Guzm{\'a}n, Rub{\'e}n
and Hugel, Thorsten
and S{\"o}nnichsen, Carsten},
title={Conformational Dynamics of a Single Protein Monitored for 24 h at Video Rate},
journal={Nano Letters},
year={2018},
month={Oct},
day={10},
publisher={American Chemical Society},
volume={18},
number={10},
pages={6633-6637},
issn={1530-6984},
doi={10.1021/acs.nanolett.8b03342},
url={https://doi.org/10.1021/acs.nanolett.8b03342}
}

@article{berg_purcell,
title = {Physics of chemoreception},
journal = {Biophysical Journal},
volume = {20},
number = {2},
pages = {193-219},
year = {1977},
issn = {0006-3495},
doi = {https://doi.org/10.1016/S0006-3495(77)85544-6},
url = {https://www.sciencedirect.com/science/article/pii/S0006349577855446},
author = {H.C. Berg and E.M. Purcell},

}

@article{Seifert_2012,
	url = {https://doi.org/10.1088/0034-4885/75/12/126001},
	year = 2012,
	month = {nov},
	publisher = {{IOP} Publishing},
	volume = {75},
	number = {12},
	pages = {126001},
	author = {Udo Seifert},
	title = {Stochastic thermodynamics, fluctuation theorems and molecular machines},
	journal = {Reports on Progress in Physics},
	
}

@article{moro,
author = {Moro,Giorgio J. },
title = {Kinetic equations for site populations from the Fokker–Planck equation},
journal = {The Journal of Chemical Physics},
volume = {103},
number = {17},
pages = {7514-7531},
year = {1995},

URL = { 
        https://doi.org/10.1063/1.470320},
    


}

@Article{SP1,
author={Saxton, Michael J.},
title={Single-particle tracking: connecting the dots},
journal={Nature Methods},
year={2008},
day={01},
volume={5},
number={8},
pages={671-672},
issn={1548-7105},
doi={10.1038/nmeth0808-671},
url={https://doi.org/10.1038/nmeth0808-671}
}

@Article{SP2,
author ="Ernst, Dominique and Köhler, Jürgen and Weiss, Matthias",
title  ="Probing the type of anomalous diffusion with single-particle tracking",
journal  ="Phys. Chem. Chem. Phys.",
year  ="2014",
volume  ="16",
issue  ="17",
pages  ="7686-7691",
publisher  ="The Royal Society of Chemistry",
doi  ="10.1039/C4CP00292J",
url  ="http://dx.doi.org/10.1039/C4CP00292J",
}

@Article{SM1,
author ="Hughes, Megan L and Dougan, Lorna",
title  ="The physics of pulling polyproteins: a review of single molecule force spectroscopy using the AFM to study protein unfolding",
journal  ="Rep. Prog. Phys.",
year  ="2016",
volume  ="79",
number = "7",

url  ="https://iopscience.iop.org/article/10.1088/0034-4885/79/7/076601",
}

@Article{SM2,
author={Neuman, Keir C.
and Nagy, Attila},
title={Single-molecule force spectroscopy: optical tweezers, magnetic tweezers and atomic force microscopy},
journal={Nature Methods},
year={2008},
month={Jun},
day={01},
volume={5},
number={6},
pages={491-505},
issn={1548-7105},
url={https://doi.org/10.1038/nmeth.1218}
}

@article{motzkin,
title = {Motzkin numbers},
journal = {Journal of Combinatorial Theory, Series A},
volume = {23},
number = {3},
pages = {291-301},
year = {1977},
issn = {0097-3165},
doi = {https://doi.org/10.1016/0097-3165(77)90020-6},
url = {https://www.sciencedirect.com/science/article/pii/0097316577900206},
author = {Robert Donaghey and Louis W Shapiro},

}

@article{lezaud,
     author = {Lezaud, Pascal},
     title = {Chernoff and Berry-Ess\'een inequalities for Markov processes},
     journal = {ESAIM: Probability and Statistics},
     pages = {183--201},
     publisher = {EDP-Sciences},
     volume = {5},
     year = {2001},
     
     language = {en},
     url = {http://www.numdam.org/item/PS_2001__5__183_0/},
}

@article{wu,
     author = {Wu, Liming},
     title = {A deviation inequality for non-reversible Markov processes},
     journal = {Annales de l'I.H.P. Probabilit\'es et statistiques},
     pages = {435--445},
     publisher = {Gauthier-Villars},
     volume = {36},
     number = {4},
     year = {2000},
     language = {en},
     url = {http://www.numdam.org/item/AIHPB_2000__36_4_435_0/}
}

@article{guillin, 
title={deviation bounds for additive functionals of markov processes}, 
volume={12}, 
DOI={10.1051/ps:2007032}, 
journal={ESAIM: Probability and Statistics}, 
publisher={EDP Sciences},
author={Cattiaux, Patrick and Guillin, Arnaud},
year={2008},
pages={12–29}
}

@article{bernstein,
title={Bernstein-type Concentration Inequalities for Symmetric Markov Processes},
volume={58},
DOI={10.1137/s0040585x97986667},
number={3},
journal={Theory of Probability and Its Applications},
author={Gao, F. and Guillin, A. and Wu, L.},
year={2014},
pages={358–382}
}

@article{Lapolla,
   title={Spectral theory of fluctuations in time-average statistical mechanics of reversible and driven systems},
   volume={2},
   ISSN={2643-1564},
   url={http://dx.doi.org/10.1103/PhysRevResearch.2.043084},
   DOI={10.1103/physrevresearch.2.043084},
   number={4},
   journal={Physical Review Research},
   publisher={American Physical Society (APS)},
   author={Lapolla, Alessio and Hartich, David and Godec, Aljaž},
   year={2020},
   month={Oct}
 }

@article{Guillin2009,
author={Guillin, Arnaud
and L{\'e}onard, Christian
and Wu, Liming
and Yao, Nian},
title={Transportation-information inequalities for Markov processes},
journal={Probability Theory and Related Fields},
year={2009},
month={Jul},
day={01},
volume={144},
number={3},
pages={669-695},
doi={10.1007/s00440-008-0159-5},
url={https://doi.org/10.1007/s00440-008-0159-5}
}

@BOOK{Zwanzig2001,
	AUTHOR = {Zwanzig, Robert},
	YEAR = {2001},
	TITLE = {Nonequilibrium Statistical Mechanics - },
	EDITION = {},
	ISBN = {978-0-195-14018-7},
	PUBLISHER = {OUP USA},
	ADDRESS = {New York},
}

@BOOK{KAMPEN,
	AUTHOR = {KAMPEN, N. G. VAN},
	YEAR = {2007},
	TITLE = {Stochastic processes in physics and chemistry - },
	EDITION = {},
	ISBN = {},
	PUBLISHER = {Elsevier},
	ADDRESS = {Amsterdam},
}

@BOOK{horn_matrix_analysis,
	AUTHOR = {Horn, },
	YEAR = {2012},
	TITLE = {Matrix Analysis, Second Edition  },
	EDITION = {},
	ISBN = {978-1-139-79342-1},
	PUBLISHER = {Cambridge University Press},
	ADDRESS = {Cambridge},
}

@BOOK{Dembo2009,
	AUTHOR = {Dembo, Amir AND Zeitouni, Ofer},
	YEAR = {2009},
	TITLE = {Large Deviations Techniques and Applications - },
	EDITION = {},
	ISBN = {978-3-642-03310-0},
	PUBLISHER = {Springer Berlin Heidelberg},
	ADDRESS = {Wiesbaden},
}

@BOOK{learning_theory,
	AUTHOR = {Kearns, Michael J. AND Vazirani, Umesh Virkumar AND Vazirani, Umesh},
	YEAR = {1994},
	TITLE = {An Introduction to Computational Learning Theory - },
	EDITION = {},
	ISBN = {978-0-262-11193-5},
	PUBLISHER = {MIT Press},
	ADDRESS = {Cambridge},
}

@BOOK{ramdomized_algorithms,
	AUTHOR = {Alippi, Cesare},
	YEAR = {2014},
	TITLE = {Intelligence for Embedded Systems - A Methodological Approach},
	EDITION = {},
	ISBN = {978-3-319-05278-6},
	PUBLISHER = {Springer},
	ADDRESS = {Berlin, Heidelberg},
}

@BOOK{large_deviations_stroock,
	AUTHOR = {Deuschel, Jean-Dominique AND Stroock, Daniel W.},
	YEAR = {2001},
	TITLE = {Large Deviations - },
	EDITION = {},
	ISBN = {978-0-821-82757-4},
	PUBLISHER = {American Mathematical Soc.},
	ADDRESS = {Heidelberg},
}

@BOOK{Engel2006,
	AUTHOR = {Engel, Klaus-Jochen AND Nagel, Rainer AND Nagel, R.},
	YEAR = {2006},
	TITLE = {A Short Course on Operator Semigroups - },
	EDITION = {},
	ISBN = {978-0-387-31341-2},
	PUBLISHER = {Springer Science and Business Media},
	ADDRESS = {Berlin Heidelberg},
}

@book{Kato,
	AUTHOR = {Kato, Tosio},
	YEAR = {1995},
	TITLE = {Perturbation Theory for Linear Operators  },
	EDITION = {2},
	ISBN = {978-3-540-58661-6},
	PUBLISHER = {Springer Science and Business Media},
	ADDRESS = {Berlin Heidelberg},
}

@book{Bremaud,
	AUTHOR = {Bremaud, Pierre},
	YEAR = {1999},
	TITLE = {Markov Chains - Gibbs Fields, Monte Carlo Simulation, and Queues},
	EDITION = {},
	ISBN = {978-0-387-98509-3},
	PUBLISHER = {Springer Science and Business Media},
	ADDRESS = {Berlin Heidelberg},
}

@book{Anderson,
	AUTHOR = {Anderson, William James},
	YEAR = {1991},
	TITLE = {Continuous-time Markov Chains - An Applications-oriented Approach},
	EDITION = {},
	ISBN = {978-3-540-97369-0},
	PUBLISHER = {Springer-Verlag},
	ADDRESS = {Berlin Heidelberg New York},
}

@book{summer_school,
	AUTHOR = {(, CIMPA Summer School AND Picco, Pierre},
	YEAR = {2003},
	TITLE = {From Classical to Modern Probability - Cimpa Summer School 2001},
	EDITION = {},
	ISBN = {978-3-764-32169-7},
	PUBLISHER = {Springer Science and Business Media},
	ADDRESS = {Berlin Heidelberg},
}

@book{Kersten2,
	AUTHOR = {Kersten, Ina},
	YEAR = {2006},
	TITLE = {Analytische Geometrie und lineare Algebra 2},
	EDITION = {},
	ISBN = {978-3-938616-44-4},
	PUBLISHER = {Universitätsverlag Göttingen},
	ADDRESS = {Göttingen},
}

@book{Klenke2013,
	AUTHOR = {Klenke, Achim},
	YEAR = {2013},
	TITLE = {Wahrscheinlichkeitstheorie },
	EDITION = {},
	ISBN = {978-3-642-36018-3},
	PUBLISHER = {Springer-Verlag},
	ADDRESS = {Berlin Heidelberg New York},
}

@book{Levy_matters,
	AUTHOR = {Böttcher, Björn AND Schilling, René AND Wang, Jian},
	YEAR = {2014},
	TITLE = {Lévy Matters III - Lévy-Type Processes: Construction, Approximation and Sample Path Properties},
	EDITION = {},
	ISBN = {978-3-319-02684-8},
	PUBLISHER = {Springer},
	ADDRESS = {Berlin, Heidelberg},
}

@book{Hall2015,
	AUTHOR = {Hall, Brian},
	YEAR = {2015},
	TITLE = {Lie Groups, Lie Algebras, and Representations - An Elementary Introduction},
	EDITION = {2},
	ISBN = {978-3-319-13467-3},
	PUBLISHER = {Springer},
	ADDRESS = {Berlin, Heidelberg},
}

@book{Fischer2020,
	AUTHOR = {Fischer, Gerd AND Springborn, Boris},
	YEAR = {2020},
	TITLE = {Lineare Algebra - Eine Einführung für Studienanfänger},
	EDITION = {},
	ISBN = {978-3-662-61644-4},
	PUBLISHER = {Springer Berlin Heidelberg},
	ADDRESS = {Wiesbaden},
}

@book{Meyer2000,
	AUTHOR = {Meyer, Carl D.},
	YEAR = {2000},
	TITLE = {Matrix Analysis and Applied Linear Algebra - },
	EDITION = {},
	ISBN = {978-0-898-71951-2},
	PUBLISHER = {SIAM},
	ADDRESS = {Philadelphia},
}

@book{concentration2013,
	AUTHOR = {Boucheron, Stéphane AND Lugosi, Gábor AND Massart, Pascal},
	YEAR = {2013},
	TITLE = {Concentration Inequalities - A Nonasymptotic Theory of Independence},
	EDITION = {},
	ISBN = {978-0-199-53525-5},
	PUBLISHER = {OUP Oxford},
	ADDRESS = {New York, London},
}

@book{Norris1998,
	AUTHOR = {Norris, J. R.},
	YEAR = {1998},
	TITLE = {Markov Chains - },
	EDITION = {},
	ISBN = {978-0-521-63396-3},
	PUBLISHER = {Cambridge University Press},
	ADDRESS = {Cambridge},
}

@book{Liggett2010,
	AUTHOR = {Liggett, Thomas Milton},
	YEAR = {2010},
	TITLE = {Continuous Time Markov Processes - An Introduction},
	EDITION = {},
	ISBN = {978-0-821-84949-1},
	PUBLISHER = {American Mathematical Soc.},
	ADDRESS = {Heidelberg},
}

@book{Kallenberg2002,
	AUTHOR = {Kallenberg, Olav},
	YEAR = {2002},
	TITLE = {Foundations of Modern Probability - },
	EDITION = {},
	ISBN = {978-0-387-95313-7},
	PUBLISHER = {Springer Science and Business Media},
	ADDRESS = {Berlin Heidelberg},
}

@book{Yin2011,
	AUTHOR = {Yin, George G. AND Zhang, Qing},
	YEAR = {2011},
	TITLE = {Continuous-Time Markov Chains and Applications - A Singular Perturbation Approach},
	EDITION = {},
	ISBN = {978-1-461-20628-6},
	PUBLISHER = {Springer New York},
	ADDRESS = {Berlin-Heidelberg},
}

@BOOK{functional_inequalities,
	AUTHOR = {Wang, Fengyu},
	YEAR = {2006},
	TITLE = {Functional Inequalities Markov Semigroups and Spectral Theory - },
	EDITION = {},
	ISBN = {978-0-080-53207-3},
	PUBLISHER = {Elsevier},
	ADDRESS = {Amsterdam},
}

@book{convex_analysis,
	AUTHOR = {Borwein, Jonathan M. AND Lewis, Adrian S.},
	YEAR = {2013},
	TITLE = {Convex Analysis and Nonlinear Optimization - Theory and Examples},
	EDITION = {},
	ISBN = {978-1-475-79859-3},
	PUBLISHER = {Springer Science and Business Media},
	ADDRESS = {Berlin Heidelberg},
}

@BOOK{Bosch2020,
	AUTHOR = {Bosch, Siegfried},
	YEAR = {2020},
	TITLE = {Algebra - },
	EDITION = {},
	ISBN = {978-3-662-61649-9},
	PUBLISHER = {Springer-Verlag},
	ADDRESS = {Berlin Heidelberg New York},
}

@BOOK{eberle,
  author        ={Andreas Eberle},
  title         = {Lecture notes-Markov processes},

  year          = {2020-2021},
  URL ={ https://uni-bonn.sciebo.de/s/kzTUFff5FrWGAay#pdfviewer}
}
